\documentclass{amsart}

\usepackage{amssymb}
\usepackage{amsmath}
\usepackage{verbatim}
\usepackage{url}
\usepackage{graphicx}

\usepackage{amsaddr}

\def\un{{\mathrm{1~\hspace{-1.4ex}l}}}

\newcommand{\BigO}{\mathcal{O}}
\newcommand{\Spec}{\operatorname{Spec}}
\newcommand{\jvRe}{\operatorname{Re}}
\newcommand{\jvIm}{\operatorname{Im}}
\newcommand{\jvexp}{\operatorname{exp}}

\newcommand{\opnm}{\operatorname}
\newcommand{\eps}{\varepsilon}
\newcommand{\kap}{\varkappa}
\newcommand{\Bff}{\mathbf}

\newtheorem{theorem}{Theorem}[section]
\newtheorem{proposition}[theorem]{Proposition}
\newtheorem{corollary}[theorem]{Corollary}
\newtheorem{lemma}[theorem]{Lemma}
\theoremstyle{definition}

\theoremstyle{remark}
\newtheorem{remark}[theorem]{Remark}
\newtheorem{example}[theorem]{Example}
\newtheorem*{acknowledgements}{Acknowledgements}

\numberwithin{equation}{section}

\title[Spectral projections and resolvents for quadratic operators]{Spectral projections and resolvent bounds for partially elliptic quadratic differential operators}

\author{Joe Viola}

\email{Joseph.Viola@univ-nantes.fr}

\address{Laboratoire de Math\'{e}matiques Jean Leray \\ 2 rue de la Houssini\`{e}re \\ Universit\'{e} de Nantes \\ BP 92208 F-44322 Nantes Cedex 3}

\begin{document}

\begin{abstract}
We study resolvents and spectral projections for quadratic differential operators under an assumption of partial ellipticity. We establish exponential-type resolvent bounds for these operators, including Kramers-Fokker-Planck operators with quadratic potentials. For the norms of spectral projections for these operators, we obtain complete asymptotic expansions in dimension one, and for arbitrary dimension, we obtain exponential upper bounds and the rate of exponential growth in a generic situation. We furthermore obtain a complete characterization of those operators with orthogonal spectral projections onto the ground state. 


\end{abstract}

\maketitle

\setcounter{tocdepth}{1}
\tableofcontents

\section{Introduction}

\subsection{Overview}

An extensive body of recent work has focused on the size of resolvent norms, semigroups, and spectral projections for non-normal operators, where these objects are not controlled by the spectrum of the operator; see \cite{TrEmBook}.  Rapid resolvent growth for quadratic operators such as 
\[
	Pu = -\Delta u + V(x)u, \quad V(x) = i|x|^2
\]
along rays inside the range of the symbol has been shown \cite{Da1999}, \cite{Zw2001} and extended significantly \cite{DeSjZw2004}, \cite{PSDuke}.  Sharp upper bounds of exponential type were recently shown in \cite{HiSjVi2011}.  The spectral projections of these operators were explored in \cite{DaKu2004}, where precise rates of exponential growth were found.  We focus here on operators with purely quadratic symbols, which are useful as accurate approximations for many operators whose symbols have double characteristics.

A weaker hypothesis than ellipticity describes a broader class of operators which includes many operators important to kinetic theory \cite{HeSjSt2005}, \cite{HeNiBook}.  Hypotheses on the so-called singular space of the symbol, particularly when that space is trivial, have been used successfully to describe semigroups generated by such operators \cite{HiPS2009}, \cite{HiPS2010}, \cite{PS2011}, \cite{HiPS2012}, \cite{OtPaPS2012}.

The purpose of the present work is threefold: first, we extend the analysis of \cite{HiSjVi2011} to include these operators with trivial singular spaces, providing exponential-type upper bounds for resolvents.  Second, we describe the spectral projections of elliptic and partially elliptic operators in a concrete way.  Third, we exploit this description to obtain information related to spectral projections and their norms, including exponential upper bounds, the rate of exponential growth in a generic situation, a complete asymptotic expansion in dimension 1, and a characterization of those operators with orthogonal projection onto the ground state.

\subsection{Background on quadratic operators}\label{ssBackground}

The structure of quadratic forms 
\[
	q(x,\xi):\Bbb{R}^n_x\times\Bbb{R}^n_\xi \rightarrow \Bbb{C}
\]
and their associated differential operators is well-studied (see e.g.\ Chapter 21.5 of \cite{HoALPDO3}), and here we recall much of the standard terminology which will be used throughout this work.

With the formulation
\[
	q(x,\xi) = \sum_{|\alpha + \beta| = 2} q_{\alpha\beta}x^\alpha\xi^\beta, \quad q_{\alpha\beta} \in \Bbb{C},
\]
we can identify the semiclassical Weyl quantization of $q$, viewed as an unbounded operator on $L^2(\Bbb{R}^n)$, with the formula
\begin{equation}\label{eqwExplicit}
	q^w(x,hD_x) = \sum_{|\alpha+\beta| = 2} q_{\alpha\beta}\frac{x^\alpha(hD_x)^\beta + (hD_x)^\beta x^\alpha}{2}.
\end{equation}
The semiclassical parameter $h>0$ is generally considered to be small and positive.  Homogeneity of the symbol $q$ and the unitary (on $L^2(\Bbb{R}^n)$) change of variables
\begin{equation}\label{eScaleOuthOp}
	U_h u(x) = h^{n/4}u(h^{1/2}x)
\end{equation}
give the relation
\begin{equation}\label{eScaleOuthEqn}
	U_h q^w(x,hD_x) U_h^* = h q^w(x,D_x),
\end{equation}
demonstrating that the semiclassical quantization of quadratic forms is unitarily equivalent to a scaling of the classical ($h=1$) quantization.

We have the standard symplectic form
\[
	\sigma((x,\xi), (y,\eta)) = \xi\cdot y - \eta \cdot x.
\]
Associated with $q$ is the ``Hamilton map" or ``fundamental matrix"
\begin{equation}\label{eFDef}
	F = \frac{1}{2}H_q = \frac{1}{2}\left(\begin{array}{cc} \partial_\xi\partial_x q & \partial_\xi^2 q \\ -\partial_x^2 q & -\partial_x\partial_\xi q\end{array}\right),
\end{equation}
which is the unique linear operator on $\Bbb{C}^{2n}$, antisymmetric with respect to $\sigma$ in the sense that
\[
	\sigma((x,\xi), F(y,\eta)) = -\sigma(F(x,\xi), (y,\eta)), \quad \forall (x,\xi), (y,\eta) \in\Bbb{R}^{2n},
\]
for which
\[
	q(x,\xi) = \sigma((x,\xi),F(x,\xi)), \quad \forall (x,\xi) \in \Bbb{R}^{2n}.
\]
We will write $F = F(q)$ when the quadratic form is perhaps unclear.

We here consider $q$ which are partially elliptic both in that
\begin{equation}\label{eRealSemidef}
	\jvRe q(x,\xi) \geq 0, \quad \forall(x,\xi)\in \Bbb{R}^{2n}
\end{equation}
and in that the so-called singular space of $q$, defined in \cite{HiPS2009}, is trivial:
\begin{equation}\label{eTrivialS}
	S := \bigcap_{k=0}^\infty \ker \left((\jvRe F)(\jvIm F)^k\right) = \{0\}.
\end{equation}
We will say that $q$ is elliptic if there exists $C>0$ with
\begin{equation}\label{eEllipticDef}
	\jvRe q(x,\xi) \geq \frac{1}{C}|(x,\xi)|^2,\quad \forall(x,\xi)\in\Bbb{R}^{2n}.
\end{equation}
We note that any elliptic quadratic form has $\ker (\jvRe F) = \{0\}$, and so the conditions (\ref{eRealSemidef}), (\ref{eTrivialS}) generalize the elliptic case.  We also recall that, aside from some degenerate cases only occuring when $n=1$, the assumption $q^{-1}(\{0\}) = \{0\}$ suffices to establish that $zq$ is elliptic for some $z \in \Bbb{C}$ with $|z| = 1$ (see, e.g.,\ Lemma 3.1 of \cite{Sj1974}).

Under the assumption (\ref{eTrivialS}), define $k_0$ as the least nonnegative integer such that the intersection defining $S$ becomes trivial:
\begin{equation}\label{ek0Def}
	k_0 = \min\left\{K \in \Bbb{N}_0\::\:\bigcap_{k=0}^{K} \ker \left((\jvRe F)(\jvIm F)^k\right) = \{0\}\right\}.
\end{equation}
By the Cayley-Hamilton theorem, $k_0 \leq 2n-1$, and when $q$ satisfies (\ref{eRealSemidef}) and (\ref{eTrivialS}), $k_0 = 0$ if and only if $q$ is elliptic.

In the case where $q$ is partially elliptic and has trivial singular space as in (\ref{eRealSemidef}) and (\ref{eTrivialS}), we are assured that, counting with algebraic multiplicity,
\begin{equation}\label{eEigenvalUHPLHP}
	\# (\{\jvIm \lambda > 0\} \cap \opnm{Spec} F) = \# (\{\jvIm \lambda < 0\} \cap \opnm{Spec} F) = n.
\end{equation}
We write $V_\lambda = V_\lambda(q)$ for the generalized eigenspace of $F$ corresponding to the eigenvalue $\lambda$.  We then have the associated subspaces
\begin{equation}\label{eLambdaDef}
	\Lambda^{\pm} = \Lambda^{\pm}(q) = \bigoplus_{\pm \jvIm \lambda > 0} V_\lambda,
\end{equation}
which are Lagrangian, meaning that $\opnm{dim}\Lambda^\pm = n$ and $\sigma|_{\Lambda^\pm} \equiv 0$.  Furthermore, when $q$ is elliptic as in (\ref{eEllipticDef}), we have that $\Lambda^+$ is positive in the sense that
\begin{equation}\label{ePositiveLagrangianDef}
	-i\sigma(X,\overline{X}) > 0 \quad \forall X \in \Lambda^+ \backslash\{0\}.
\end{equation}
In the corresponding sense, $\Lambda^-$ is negative.  The extension of this fact to $q$ obeying (\ref{eRealSemidef}) and (\ref{eTrivialS}) is essentially known in previous works; see the end of Section 2 of \cite{HiPS2012} and references therein.  For completeness, we here include a proof in Proposition \ref{pLambdaPositive}.

For any quadratic $q$ obeying (\ref{eRealSemidef}) and (\ref{eTrivialS}), we may write the spectrum of $q^w(x,hD_x)$ as a lattice obtained from the eigenvalues of $F$ in the upper half plane, written $\lambda_1,\dots,\lambda_n$.  Define
\begin{equation}\label{emuDef}
	\mu_\alpha = \frac{h}{i}\sum_{j=1}^n (2\alpha_j+1)\lambda_j.
\end{equation}
Then we have the formula
\begin{equation}\label{eSpectrum}
	\opnm{Spec} q^w(x,hD_x) = \{\mu_\alpha \::\: \alpha \in \Bbb{N}_0^n\}.
\end{equation}
This was classically known in the elliptic case \cite{Sj1974}, \cite{BdM1974}.  In the partially elliptic case, this formula was proven in Theorem 1.2.2 of \cite{HiPS2009} under somewhat weaker hypotheses than (\ref{eTrivialS}).

We also study the spectral projections of these operators.  Following the notation of Theorem XV.2.1 of \cite{GoGoKrBook} (see also Chapter 6 of \cite{HiSiBook}), let us assume that $A$ is a closed densely defined operator on $\mathcal{H}$ a Hilbert space, and $\opnm{Spec}A = \Omega_1 \cup \Omega_2$, where $\Omega_1$ is contained in a bounded Cauchy domain $\Delta$ with $\overline{\Delta}\cap \Omega_2 = \emptyset$.  Let $\Gamma$ be the oriented boundary of $\Delta$.  Then we call
\begin{equation}\label{eSpectralProjection}
	P_{\Omega_1, A} = (2\pi i)^{-1} \int_\Gamma (\zeta-A)^{-1}\,d\zeta
\end{equation}
the spectral (or Riesz) projection for $A$ and $\Omega_1$.

Because the spectra we will study, given by (\ref{eSpectrum}), are discrete, we will generally use the definition in the case that $\Omega_1$ is finite.  We emphasize that facts about the spectral projections are independent of the semiclassical parameter after scaling, as (\ref{eScaleOuthEqn}) provides that the projection for the classical operator and $\Omega_1$ is unitarily equivalent to the projection for the semiclassical operator and $h\Omega_1$:
\begin{equation}\label{eRescalingProjections}
	U_h P_{h\Omega_1, q^w(x,hD_x)}U_h^* = P_{\Omega_1, q^w(x,D_x)}.
\end{equation}

We perform much of our analysis in weighted spaces of entire functions associated with FBI transforms (see for example \cite{MaBook}, \cite{Sj1996}, or Chapter 12 of \cite{SjLoR}).  For $\Phi:\Bbb{C}^n \rightarrow \Bbb{R}$, we define
\begin{equation}\label{eHPhiDef}
	H_\Phi(\Bbb{C}^n;h) = \opnm{Hol}(\Bbb{C}^n)\cap L^2(\Bbb{C}^n; e^{-2\Phi(x)/h}\,dL(x)).
\end{equation}
Here $dL(x) = d\jvRe x \,d\jvIm x$ is Lebesgue measure on $\Bbb{C}^n$, and we only need to consider the most elementary case where $\Phi$ is quadratic when regarded as a function of $(\jvRe x, \jvIm x) \in \Bbb{R}^{2n}$, real-valued, and strictly convex.  When functions do not need to be holomorphic, we refer to
\[
	L^2_\Phi(\Bbb{C}^n;h) = L^2(\Bbb{C}^n; e^{-2\Phi(x)/h}\,dL(x)),
\]
and we will often omit $(\Bbb{C}^n;h)$ where we hope it can be understood.

When working in weighted spaces $H_\Phi(\Bbb{C}^n;h)$ or $L^2_\Phi(\Bbb{C}^n;h)$, we assume unless otherwise stated that derivatives are holomorphic, meaning that
\[
	\partial_x = \frac{1}{2}(\partial_{\jvRe x} -i \partial_{\jvIm x}).
\]

We finish this section with some brief remarks on notation.  We use $(\cdot,\cdot)$ for a symmetric (bilinear) inner product, usually the dot product on $\Bbb{C}^n$, and $\langle \cdot, \cdot\rangle_{\mathcal{H}}$ for a Hermitian (sesquilinear) inner product on a Hilbert space $\mathcal{H}$.  We frequently refer to adjoints of operators on a Hilbert space.  When the space needs to be emphasized, we add it as a subscript, for example writing $A^*_{\mathcal{H}}$.  We use $\mathcal{L}(\mathcal{H})$ to denote the set of bounded linear operators mapping $\mathcal{H}$ to itself with the usual operator norm.  We frequently use a superscript $\dagger$ to indicate that an object is ``dual" in a loose sense, but the formal meaning may change from instance to instance.

Finally, when we say that a unitary operator $U$ quantizes a canonical transformation $\kappa$, we mean that
\[
	U p^w(x,hD_x)U^* = (p\circ \kappa^{-1})^w(x,hD_x)
\]
for appropriate symbols $p$ and an appropriate definition of the semiclassical Weyl quantization.  In this work, we only apply this notion to (complex) linear canonical transformations and to symbols which are homogeneous polynomials of degree no more than 2, in which case formulas like (\ref{eqwExplicit}) may be used.  We therefore use only the most rudimentary aspects of the theory of metaplectic operators; see for instance the Appendix to Chapter 7 of \cite{DiSjBook} or Chapter 3.4 of \cite{MaBook}.

\subsection{Statement of results}

We are now in a position to formulate the four main results of this work.  

First, we extend the central result of \cite{HiSjVi2011} to include partially elliptic operators, at the price of more rapid exponential growth.  In fact, the result here is identical to the main result in \cite{HiSjVi2011} save that exponential growth in $h^{-1}$ is replaced by exponential growth in $h^{-1-2k_0}$.  A remarkable recent estimate of Pravda-Starov \cite{PS2011} provides a subelliptic estimate sufficient to establish the following theorem, which gives exponential-type semiclassical resolvent bounds when the spectral parameter is bounded and avoids a rapidly shrinking neighborhood of the spectrum.

\begin{theorem}\label{tExtendToPE}

Let $q:\Bbb{R}^{2n} \rightarrow \Bbb{C}$ be a quadratic form which is partially elliptic with trivial singular space in the sense of (\ref{eRealSemidef}) and (\ref{eTrivialS}), and let $k_0$ be defined as in (\ref{ek0Def}).

If $F(q)$ is diagonalizable, then for any $C_1,C_2,L > 0$ there exist $h_0 > 0$ sufficiently small and $A > 0$ sufficiently large where, if $z \in \Bbb{C}$, $|z| \leq C_2$,
\[
	\opnm{dist}(z,\opnm{Spec}q^w(x,hD_x)) \geq \frac{1}{C_1}\jvexp(-Lh^{-1-2k_0}),
\]
and $h \in (0,h_0]$, we have the resolvent bound
\[
	||(q^w(x,hD_x)-z)^{-1}||_{\mathcal{L}(L^2(\Bbb{R}^n))} \leq A\jvexp\left(Ah^{-1-2k_0}\right).
\]

If $F(q)$ is not assumed to be diagonalizable, then for any $C_1,C_2,L > 0$ there exist $h_0 > 0$ sufficiently small and $A > 0$ sufficiently large where, if $z \in \Bbb{C}$, $|z| \leq C_2$,
\[
	\opnm{dist}(z,\Spec q^w(x,hD_x)) \geq \frac{h^L}{C_1},
\]
and $h \in (0, h_0]$, we have the resolvent bound
\[
	||(q^w(x,hD_x)-z)^{-1}||_{\mathcal{L}(L^2(\Bbb{R}^n))} \leq A\jvexp\left(Ah^{-1-2k_0}\log\frac{1}{h}\right).
\]

\end{theorem}

We also have in \cite{HiSjVi2011} a unitary equivalence between $L^2(\Bbb{R}^n)$ and a weighted space $H_{\Phi_2}(\Bbb{C}^n)$ of entire functions, defined in (\ref{eHPhiDef}), which reduces the symbol $q(x,\xi)$ to a normal form $\tilde{q}(x,\xi)$; we review this in Section \ref{ssNormalForm}.   In that weighted space, we have a simple characterization of the spectral projections for $\tilde{q}^w(x,hD_x)$ as truncations of the Taylor series.

\begin{theorem}\label{tCharacterizeProj}
Let $q:\Bbb{R}^{2n}\rightarrow\Bbb{C}$ be quadratic and partially elliptic with trivial singular space as in (\ref{eRealSemidef}) and (\ref{eTrivialS}).  Let $\mu_\alpha$ be as defined in (\ref{emuDef}).  As described in Proposition \ref{pNormalForm}, the operator $q^w(x,hD_x)$ acting on $L^2(\Bbb{R}^n)$ is unitarily equivalent to $\tilde{q}^w(x,hD_x)$ acting on $H_{\Phi_2}(\Bbb{C}^n;h)$ for some $\Phi_2:\Bbb{C}^n\rightarrow \Bbb{R}$ real-valued, quadratic, and strictly convex.  Using the notation (\ref{eSpectralProjection}), write
\[
	\Pi_z = P_{\{z\}, \tilde{q}^w(x,hD_x)}:H_{\Phi_2}(\Bbb{C}^n;h) \rightarrow H_{\Phi_2}(\Bbb{C}^n;h).
\]
Then
\begin{equation}\label{eCharacterizeProj}
	\Pi_z u(x) = \sum_{\alpha\::\:\mu_\alpha = z} (\alpha !)^{-1}(\partial^\alpha u(0))x^\alpha.
\end{equation}
\end{theorem}

The motivation behind establishing Theorem \ref{tCharacterizeProj} is to provide information about the spectral projections, particularly the operator norms thereof.  The approach of using dual bases for eigenvectors was used in \cite{DaKu2004} in finding exact rates of exponential growth for the operators described in Examples \ref{exDavies1} and \ref{exDavies2}; we follow a similar approach here.  The most tractable projections seem to be for eigenvalues with multiplicity 1, meaning the expansion in (\ref{eCharacterizeProj}) consists of a single term.  Note that this is true for every $\mu_\alpha$ simultaneously if and only if the eigenvalues of $F(q)$ which lie in the upper half-plane are rationally independent, which is a generic condition.  

As explained above in (\ref{eRescalingProjections}), there is no reason to describe the norms of spectral projections semiclassically; we therefore state the result with $h = 1$.  We furthermore see in Proposition \ref{pReversibility} that the set of $\Phi_2$ which may be obtained from Proposition \ref{pNormalForm} is exactly the set of strictly convex real-valued quadratic forms $\Phi$.  We therefore treat such a $\Phi$ as the object of study in the following theorem.

\begin{theorem}\label{tBoundProj}
Let $\Phi:\Bbb{C}^n\rightarrow\Bbb{R}$ be strictly convex, real-valued, and quadratic.  Write
\[
	\Pi_\alpha u(x) = (\alpha !)^{-1}(\partial^\alpha u(0))x^\alpha:H_\Phi(\Bbb{C}^n;1) \rightarrow H_\Phi(\Bbb{C}^n;1).
\]
Then there exists another quadratic strictly convex weight $\Phi^\dagger$ and a constant $C_\Phi$ for which, for all $\alpha \in \Bbb{N}^n_0$, we have the formula
\begin{equation}\label{eBoundProj}
	||\Pi_\alpha||_{\mathcal{L}(H_{\Phi}(\Bbb{C}^n;1))} = \frac{C_\Phi}{\alpha!2^{|\alpha|}}||x^\alpha||_{H_{\Phi}(\Bbb{C}^n;1)}||x^\alpha||_{H_{\Phi^\dagger}(\Bbb{C}^n;1)}.
\end{equation}
Here $C_\Phi = (2\pi)^{-n}\det(1-C_+^*C_+)^{-1/4}$ for $C_+ \in \Bbb{C}^{n\times n}$ a symmetric matrix associated with $\Phi$.  In view of Proposition \ref{pReversibility}, one may deduce the definitions of $C_+$ and $\Phi^\dagger$ from Remark \ref{rGCFromConvexity} and (\ref{ePhi2daggerMatrix}).
\end{theorem}

This result has three simple corollaries, which we formally state in Section \ref{ssCorollaries}.  First, we have exponential upper bounds for the spectral projections of any elliptic or partially elliptic operator.  Second, we have a complete asymptotic expansion for spectral projections in dimension one, where eigenvalues are automatically simple.  Finally, we have a formula for the rate of exponential growth, regardless of dimension, in the generic situation when eigenvalues are simple.

It is useful for analysis of $q^w(x,hD_x)$ to have some orthogonal decomposition of $L^2(\Bbb{R}^n)$ into $q^w(x,hD_x)$-invariant subspaces.  That collections of Hermite functions of fixed degree form such a decomposition for Kramers-Fokker-Planck operators with quadratic potential was known since \cite{RiBook}, as described in Section 5.5 of \cite{HeNiBook}.  We explore one such operator in Example \ref{exKFP1}, and we have the same decomposition for an operator whose Hamilton map has Jordan blocks in Example \ref{exJordan1}.

The question of orthogonal spectral projections for partially elliptic operators has been raised in the recent work \cite{OtPaPS2012}, which focuses on semigroup bounds for such operators.  Working under the assumptions that the ground state of $q^w(x,hD_x)$ matches that of $q^w(x,hD_x)^*$ and that the operator is totally real, the authors of \cite{OtPaPS2012} show strong similarity, on the level of semigroups, between the behavior of the spectral projection for $q^w(x,hD_x)$ and $\{\mu_0\}$ and the behavior of the orthogonal projection onto the span of the corresponding eigenfunction.

Inspired by this work, we observe that the analysis here beginning at Theorem \ref{tCharacterizeProj} and leading towards Theorem \ref{tBoundProj} puts us in a position to describe necessary and sufficient conditions on $q$ for this projection to be orthogonal.

\begin{theorem}\label{tOrthogonalProj}
Let $q:\Bbb{R}^{2n}\rightarrow\Bbb{C}$ be quadratic and partially elliptic with trivial singular space as in (\ref{eRealSemidef}) and (\ref{eTrivialS}).  Recall the definitions of $\Lambda^{\pm} = \Lambda^{\pm}(q)$ in (\ref{eLambdaDef}) and $\mu_\alpha$ in (\ref{emuDef}).  Let $\Pi_{\mu_0}$ be the spectral projection for $q^w(x,hD_x)$ and $\{\mu_0\}$, as in (\ref{eSpectralProjection}).  Then the following are equivalent:
\begin{enumerate}
	\item\label{iOrthogker} the ground states of the operator and the adjoint match, 
		\[
			\ker(q^w(x,hD_x)-\mu_0) = \ker(q^w(x,hD_x)^* - \overline{\mu_0});
		\]
	\item\label{iOrthogbar} the stable manifolds associated with $q$ are conjugate, $\Lambda^+ = \overline{\Lambda^-}$; and
	\item\label{iOrthogproj} the projection $\Pi_{\mu_0}$ is orthogonal on $L^2(\Bbb{R}^n)$.
\end{enumerate}
\end{theorem}

\begin{remark}\label{rFurtherDecomposition} A further decomposition immediately follows if any of these conditions hold.  Studying the unitarily equivalent operator $\tilde{q}^w(x,hD_x)$ acting on $H_{\Phi_2}(\Bbb{C}^n;h)$, we have that the spaces of polynomials homogeneous of fixed degree,
\[
	E_m = \opnm{span}\{x^\alpha\::\:|\alpha| = m\},
\]
are orthogonal $\tilde{q}^w(x,hD_x)$-invariant subspaces of $H_{\Phi_2}(\Bbb{C}^n;h)$ which together have dense span.  We also have that
\[
	\opnm{Spec}(\tilde{q}^w(x,hD_x)|_{E_m}) = \{\mu_\alpha\::\: |\alpha| = m\}.
\]
Some illustrations using this decomposition may be found in Section \ref{ssNumerics}.
\end{remark}

\subsection{Corollaries on the growth of spectral projections}\label{ssCorollaries}

In Section \ref{sComputations} we derive three simple corollaries of Theorem \ref{tBoundProj}.

First, we have an exponential upper bound for spectral projections for the quadratic operators we have been considering.  We note that, following Remark \ref{rNotSharp}, we do not expect this bound to be sharp in dimension $n \geq 2$ in general.

\begin{corollary}\label{cUpperBound}
Let $q:\Bbb{R}^{2n}\rightarrow\Bbb{C}$ be quadratic and partially elliptic with trivial singular space as in (\ref{eRealSemidef}) and (\ref{eTrivialS}). Let $\mu_\alpha$ be as defined in (\ref{emuDef}). Using (\ref{eSpectralProjection}), write
\[
	\Pi_z = P_{\{z\}, q^w(x,hD_x)}:L^2(\Bbb{R}^n) \rightarrow L^2(\Bbb{R}^n).
\]
For $\Phi_2$ taken from Proposition \ref{pNormalForm} and $\Phi_2^\dagger$ derived from $\Phi_2$ as in Theorem \ref{tBoundProj}, let
\[
	A_1 = \inf_{|\omega|=1} 4\Phi_2(\omega), \quad A_2 = \inf_{|\omega|=1} 4\Phi_2^\dagger(\omega).
\]
Then
\[
	||\Pi_{\mu_\alpha}||_{\mathcal{L}(L^2(\Bbb{R}^n))} \leq \BigO(1+|\alpha|^{n-1})\left(A_1A_2\right)^{-|\alpha|/2}.
\]
\end{corollary}

In (spatial) dimension 1, we have a complete asymptotic expansion for spectral projections as the size of the eigenvalue becomes large.

\begin{corollary}\label{cDim1Asymptotic}
Let $q:\Bbb{R}^{2}\rightarrow\Bbb{C}$ be quadratic and partially elliptic with trivial singular space as in (\ref{eRealSemidef}) and (\ref{eTrivialS}).  By (\ref{eEigenvalUHPLHP}), there exists only one (algebraically simple) eigenvalue of $F$ with positive imaginary part; call this eigenvalue $\lambda$.  Let
\[
	\mu_N = \frac{\lambda}{i}(2N+1).
\]
Using (\ref{eSpectralProjection}), write
\[
	\Pi_N = P_{\{\mu_N\}, q^w(x,D_x)}:L^2(\Bbb{R}^1) \rightarrow L^2(\Bbb{R}^1).
\]
In dimension 1, the $C_+$ of Theorem \ref{tBoundProj} must be a complex number with $|C_+| < 1$.  After dividing out the rate of exponential growth identified in \cite{DaKu2004}, there exists a complete asymptotic expansion
\[
	\left(\frac{1-|C_+|}{1+|C_+|}\right)^{N/2}||\Pi_N||_{\mathcal{L}(L^2(\Bbb{R}^1))} \sim \sum_{j=0}^\infty c_j N^{-j-1/2}
\]
as $N \rightarrow \infty$, for some $\{c_j\}_{j=0}^\infty$ a sequence of real numbers depending only on $C_+$.  We furthermore compute that
\[
	c_0 =(2\pi|C_+|)^{-1/2}\left(\frac{1+|C_+|}{1-|C_+|}\right)^{1/4}.
\]
\end{corollary}

In the case of higher dimensions, the maximization problem leading to Corollary \ref{cDim1Asymptotic} is much more difficult.  In the generic case of simple eigenvalues, we are nonetheless able to identify the rate of exponential growth for spectral projections along rays $\{\mu_{\lambda \beta}\::\: \lambda \beta \in \Bbb{N}_0^n\}$ for $\beta \in (\overline{\Bbb{R}_+})^n$ fixed as $\lambda \rightarrow \infty$.

While this provides significant information on the exponential growth of spectral projections for a broad class of non-normal quadratic operators, the author feels that this result in higher dimensions is rather preliminary and hopes to return to the subject in later work.

\begin{corollary}\label{cExponentialGrowth}
	Let $\Phi:\Bbb{C}^n\rightarrow\Bbb{R}$ be strictly convex, real-valued, and quadratic.  Write
	\[
		\Pi_\alpha u(x) = (\alpha !)^{-1}(\partial^\alpha u(0))x^\alpha:H_\Phi(\Bbb{C}^n;1) \rightarrow H_\Phi(\Bbb{C}^n;1).
	\]
	Let $\Phi^\dagger$ be the dual weight as in Theorem \ref{tBoundProj}.

	Consider $\beta \in (\overline{\Bbb{R}_+})^n\backslash \{0\}$ normalized so that $|\beta| = \sum_{j=1}^n\beta_j = 1$.  For those $\lambda \in \Bbb{R}_+$ for which $\lambda\beta \in \Bbb{N}_0^n$, we have the following exponential rate of growth in the limit $\lambda \rightarrow \infty$:
	\begin{multline*}
		\lambda^{-1}\log||\Pi_{\lambda\beta}||_{L^2(\Bbb{R}^n)} 
		\\ = \frac{1}{2}\log \left(\sup_{|\omega| = 1} (4\Phi(\omega))^{-1}|\omega^\beta|^2\right) + \frac{1}{2}\log \left(\sup_{|\omega| = 1} (4\Phi^\dagger(\omega))^{-1}|\omega^\beta|^2\right)
		\\ - \sum_{j\::\:\beta_j \neq 0} \beta_j \log \beta_j + \BigO(\lambda^{-1}\log \lambda).
	\end{multline*}
	As with multi-indices, we define $|\omega^\beta|^2 = \prod_{j=1}^n |\omega_j|^{2\beta_j}$.

	Furthermore, consider $q:\Bbb{R}^{2n}\rightarrow\Bbb{C}$ quadratic and partially elliptic with trivial singular space as in (\ref{eRealSemidef}) and (\ref{eTrivialS}), with $\mu_\alpha$ defined in (\ref{emuDef}).  Due to the unitary equivalence in Proposition \ref{pNormalForm} with $\Phi_2$ provided therein, the same rate of growth holds for the norm of the classical ($h = 1$) spectral projections
	\[
		||P_{\{\mu_{\lambda \beta}\}, q^w(x,D_x)}||_{\mathcal{L}(L^2(\Bbb{R}^n))} = ||\Pi_{\lambda\beta}||_{\mathcal{L}(H_{\Phi_2}(\Bbb{C}^n;1))}
	\]
	so long as we assume that the eigenvalue $\mu_{\lambda\beta} \in \opnm{Spec}q^w(x,D_x)$ is simple.
\end{corollary}

\subsection{Plan of the paper}

Section \ref{sResolvent} is devoted to proving Theorem \ref{tExtendToPE} and recapitulating the necessary machinery used in \cite{HiSjVi2011}.  Also included are examples in Section \ref{ssExamples} and illustrations of partial ellipticity in Section \ref{ssNumerics}.  Section \ref{sCharProj} contains the proof of Theorem \ref{tCharacterizeProj} as well as an elementary exponential upper bound for spectral projections which is related to the work \cite{DaKu2004}.  Section \ref{sDualBases} focuses on the properties of dual bases for projection onto monomials in weighted spaces, and it contains proofs of Theorems \ref{tBoundProj} and \ref{tOrthogonalProj}.  Finally, Section \ref{sComputations} contains computations based on these results which prove Corollaries \ref{cUpperBound}, \ref{cDim1Asymptotic}, and \ref{cExponentialGrowth} and numerical computations based on Corollary \ref{cExponentialGrowth}.

\section{Resolvent bounds in the partially elliptic case}\label{sResolvent}

One may extend the upper bounds obtained in \cite{HiSjVi2011} for resolvents of elliptic quadratic operators to upper bounds for partially elliptic quadratic operators after two steps: duplicating the reduction to normal form and finding some replacement for an elliptic estimate.  The former can be done after demonstrating that the stable (linear Lagrangian) manifolds $\Lambda^\pm$ defined in (\ref{eLambdaDef}) are positive and negative Lagrangian planes as defined in (\ref{ePositiveLagrangianDef}), which follows more or less directly from reasoning in \cite{Sj1974}.  A subelliptic estimate, sufficient to establish Theorem \ref{tExtendToPE}, may be deduced from a remarkable recent result of Pravda-Starov \cite{PS2011}.

In this section, we will assume that our quadratic symbol
\[
	q:\Bbb{R}^{2n} \rightarrow \Bbb{C}
\]
is partially elliptic with trivial singular space as in (\ref{eRealSemidef}) and (\ref{eTrivialS}).

We begin by proving sign definiteness of $\Lambda^\pm(q)$.  Afterwards, we recall the reduction to normal form in \cite{HiSjVi2011} and remark on some additional information which may be derived from this reduction.  Following this, we present three examples which will be used throughout the rest of the paper.  Next, we prove the weak elliptic estimate for high-energy functions.  Finally, recalling the low-energy finite dimensional analysis of \cite{HiSjVi2011}, we are able to prove Theorem \ref{tExtendToPE}.

Afterwards, in Section \ref{ssNumerics}, we see some evidence that the elliptic estimate in Proposition \ref{pPartiallyEllipticEstimate} may not give a sharp rate of growth in Theorem \ref{tExtendToPE}.  However, the phenomenon of subellipticity formalized in \cite{PS2011} appears to be sharp, presenting a genuine obstacle in adapting the standard ellipticity argument found in Proposition \ref{pPartiallyEllipticEstimate}.

\subsection{Sign definiteness of $\Lambda^\pm(q)$}

To reproduce the reduction to normal form in \cite{HiSjVi2011}, one must have sign definiteness of $\Lambda^\pm(q)$ defined in (\ref{eLambdaDef}).  We include a direct proof here.

\begin{proposition}\label{pLambdaPositive}
Let $q:\Bbb{R}^{2n}\rightarrow \Bbb{C}$ be a quadratic form and $F = \frac{1}{2}H_q$ its Hamilton map defined in (\ref{eFDef}).  Assume that $q$ satisfies (\ref{eRealSemidef}) and (\ref{eTrivialS}), and recall the associated manifolds $\Lambda^\pm(q)$ defined in (\ref{eLambdaDef}).  Then $\Lambda^\pm(q)$ are positive and negative Lagrangian planes in $\Bbb{C}^{2n}$ as defined in (\ref{ePositiveLagrangianDef}).
\end{proposition}

\begin{proof}
We begin by noting, as in the proof of Lemma 3 in \cite{PS2008} or in Remark 2.2 of \cite{Vi2012b}, that (\ref{eRealSemidef}) implies that, whenever $X \in \Bbb{R}^{2n}$,
\begin{equation}\label{eReqReFSame}
	\jvRe q(X) = 0 \iff (\jvRe F) X = 0.
\end{equation}

We also recall that, when $q$ obeys (\ref{eRealSemidef}) and (\ref{eTrivialS}), there exists $\delta_0 > 0$ and a continuous family of complex linear canonical transformations $\{\kappa_\delta\}_{0 \leq \delta \leq \delta_0}$ acting on $\Bbb{C}^{2n}$ beginning with the identity, $\kappa_0 = 1_{\Bbb{C}^{2n}}$, and positive constants $\{C_\delta\}_{0 < \delta \leq \delta_0}$ for which
\[
	\jvRe (q\circ \kappa_\delta)(X) \geq \frac{1}{C_\delta}|X|^2, \quad \delta > 0.
\]
Details may be found in, for example, \cite{HiPS2010}, Section 2, or \cite{Vi2012b}, Section 2.1.  These canonical transformations induce a similarity transformation on the Hamilton map,
\[
	F(q\circ\kappa_\delta) = \kappa_\delta^{-1}F(q)\kappa_\delta,
\]
and so
\[
	\Lambda^\pm_\delta := \Lambda^\pm(q\circ\kappa_\delta),\quad 0 \leq \delta \leq \delta_0
\]
enjoy the relation
\[
	\Lambda^\pm_\delta = \kappa_\delta^{-1}(\Lambda^\pm_0).
\]
It immediately follows that $\{\Lambda^\pm_\delta\}_{0\leq \delta \leq \delta_0}$ is a continuous family of Lagrangian planes.  When $\delta > 0$, positivity of $\Lambda^+_\delta$ and negativity of $\Lambda^-_\delta$ follow from ellipticity of $q\circ\kappa_\delta$.  To apply a deformation argument in \cite{Sj1974}, we wish to show that $\Lambda^\pm_0\cap \Bbb{R}^{2n} = \{0\}$.

We know from Lemma 3.7 of \cite{Sj1974} that $\lambda \neq -\mu$ implies that $V_\lambda(q)$ and $V_\mu(q)$ are orthogonal with respect to $\sigma$, and therefore that $\Lambda^\pm_0$ are Lagrangian planes.  (This may also be seen by applying $\kappa_\delta$ to $\Lambda_\delta^\pm$.)  

Because generalized eigenspaces of an operator are invariant under that operator, we see that $X \in \Lambda^+_0$ implies that $FX \in \Lambda^+_0$.  Since $\Lambda_0^+$ is Lagrangian, we see that
\[
	X \in \Lambda_0^+ \implies q(X) = \sigma(X,FX) = 0.
\]
If we assume furthermore that $X \in \Lambda^+_0 \cap \Bbb{R}^{2n}$, we have that $\jvRe q(X) = 0$ and so $(\jvRe F)X = 0$ as well.  But then 
\[
	-iFX = (\jvIm F)X \in \Lambda^+_0 \cap \Bbb{R}^{2n}.
\]
By induction we therefore see that, whenever $X \in \Lambda^+_0 \cap \Bbb{R}^{2n}$, we have that
\[
	(\jvIm F)^k X \in \Lambda^+_0\cap\Bbb{R}^{2n}, \quad k = 0,\dots,2n-1.
\]
We have already seen that $\ker \jvRe F$ contains $\Lambda^+_0 \cap \Bbb{R}^{2n}$, and so we conclude that, whenever $X \in \Lambda_0^+ \cap \Bbb{R}^{2n}$, we have $X \in S = \{0\}$.

We may then appeal to the deformation argument following Lemma 3.8 in \cite{Sj1974}, which shows that if $\{W_\delta\}_{0 \leq \delta \leq \delta_0}$ is a continuous family of Lagrangian planes for which $W_\delta \cap \Bbb{R}^{2n} = \{0\}$, then all the $W_\delta$ are positive so long as one is.  Since $\Lambda^+_\delta$ is positive for $\delta \in (0,\delta_0]$, we know that $\Lambda^+_0 = \Lambda^+(q)$ is positive.  The same reasoning provides that $\Lambda^-(q)$ is a negative Lagrangian plane, completing the proof.
\end{proof}

\subsection{Review of reduction to normal form}\label{ssNormalForm}

Having established sign definiteness of $\Lambda^\pm$, a reduction to normal form may then proceed exactly as in Section 2 of \cite{HiSjVi2011}.  We state the result as a proposition, following Proposition 2.1 in that work, and review the proof to record some minor details.  We then make some minor remarks providing further information which will be used in the sequel.  The relevant symbol classes are 
\[
	S(\Bbb{R}^{2n},\langle(x,\xi)\rangle^m) := \{p \in C^\infty(\Bbb{R}^{2n})\::\: |\partial^\alpha_{x,\xi}p(x,\xi)|\leq \BigO_{\alpha\beta}(\langle(x,\xi)\rangle^m)\}.
\]
However, as mentioned in Section \ref{ssBackground}, in this work we only require the use of symbols which are polynomials in $(x,\xi)$.

\begin{proposition}\label{pNormalForm}
Let $q(x,\xi):\Bbb{R}^{2n}\rightarrow \Bbb{C}$ be quadratic and partially elliptic in the sense of (\ref{eRealSemidef}) and (\ref{eTrivialS}).  Then there exists a complex linear canonical transformation $\kap$ for which
\[
	\tilde{q}(x,\xi) := (q\circ\kap^{-1})(x,\xi) = (Mx)\cdot\xi
\]
for $M \in \Bbb{C}^{n\times n}$ block-diagonal with each block being a Jordan one.  Furthermore, the eigenvalues of $M$ are precisely those of $2F$ in the upper half-plane.  Associated with the transformation $\kap$ are a real-valued quadratic strictly convex weight function $\Phi_2:\Bbb{C}^n\rightarrow \Bbb{R}$ and a unitary operator
\[
	\mathcal{T}:L^2(\Bbb{R}^n)\rightarrow H_{\Phi_2}(\Bbb{C}^n)
\]
quantizing $\kap$ in that
\begin{equation}\label{eEgorov}
	\mathcal{T}p^w(x,hD_x)\mathcal{T}^* = (p\circ\kap^{-1})^w(x,hD_x), \quad \forall p \in S(\Bbb{R}^{2n},\langle (x,\xi)\rangle^m).
\end{equation}
\end{proposition}

\begin{proof}
We repeat the proof of Proposition 2.1 in \cite{HiSjVi2011} solely to make certain small details and minor changes of notation explicit.  There are three pieces in the reduction to normal form: quantizing a real canonical transformation straightening $\Lambda^-$, an FBI-Bargmann transform reducing $q$ to a polynomial simultaneously homogeneous of degree 1 in $x$ and of degree 1 in $\xi$, and a change of variables reducing the matrix in the resulting symbol to Jordan normal form.

That $\Lambda^- = \Lambda^-(q)$ is a negative Lagrangian plane is equivalent to having
\[
	\Lambda^- = \{(y,A_-y)\::\:y\in\Bbb{C}^n\}
\]
for some
\[
	A_- \in \Bbb{C}^{n\times n}, \quad A_-^t = A_-, \quad -\jvIm A_- > 0
\]
with the last in the sense of positive definite matrices.  A real linear canonical transformation such as
\begin{equation}\label{ekappaDef}
	\kappa = \left(\begin{array}{cc} (-\jvIm A_-)^{1/2} & 0 \\ -(-\jvIm A_-)^{-1/2}\jvRe A_- & (-\jvIm A_-)^{-1/2}\end{array}\right)
\end{equation}
gives $\kappa(\Lambda^-) = \{(y,-iy)\}$. We have that $\kappa$ may be quantized by a standard unitary operator on $L^2(\Bbb{R}^n)$, which reduces $\Lambda^{\pm}(q)$ to $\kappa(\Lambda^\pm(q))$ accordingly.

Since $\kappa$ is real canonical, $\kappa(\Lambda^+)$ remains positive, and so $\kappa(\Lambda^+) = \{(y,A_+y)\}$ for some symmetric $A_+ \in \Bbb{C}^{n\times n}$ with positive definite imaginary part.  Straightening 
\[
	\kappa(\Lambda^+)\mapsto \{(0,\xi)\}_{\xi\in\Bbb{C}^n}
\]
while simultaneously straightening 
\[
	\kappa(\Lambda^-)\mapsto \{(x,0)\}_{x\in\Bbb{C}^n}
\]
is accomplished by an FBI-Bargmann transform
\begin{equation}\label{eFBIT}
	\mathcal{T}_{A_+}u(x) = C_{A_+}h^{-3n/4} \int_{\Bbb{R}^n} e^{\frac{i}{h}\varphi_{A_+}(x,y)}u(y)\,dy
\end{equation}
for
\[
	\varphi_{A_+}(x,y) = \frac{i}{2}(x-y)^2 -\frac{1}{2}(x, (1-iA_+)^{-1}A_+ x).
\]
This FBI-Bargmann transform quantizes the canonical transformation, in block matrix form,
\begin{equation}\label{ekapADef}
	\kap_{A_+} = \left(\begin{array}{cc} 1 & -i \\ -(1-iA_+)^{-1}A_+ & (1-iA_+)^{-1}\end{array}\right).
\end{equation}

From \cite{HiSjVi2011}, we have that the range of the FBI-Bargmann transform (\ref{eFBIT}) is $H_{\Phi_1}(\Bbb{C}^n)$, where
\[
	\Phi_1(x) = \frac{1}{2}\left((\jvIm x)^2 + \jvIm(x,Bx)\right)
\]
for
\[
	B = (1-iA_+)^{-1}A_+.
\]
We rearrange $\Phi_1$ as follows:
\begin{multline*}
	\Phi_1(x) = -\frac{1}{8}(x-\bar{x})^2 + \frac{1}{4i}\left((x,Bx) - (\bar{x}, \overline{B}\bar{x})\right)
	\\ = \frac{1}{4}(x,\bar{x}) - \frac{1}{8}(x, (1+2iB)x) - \frac{1}{8}(\bar{x}, (1-2i\overline{B})\bar{x})
	\\ = \frac{1}{4}(|x|^2 - \jvRe(x,(1+2iB)x)).
\end{multline*}
We will use the expression
\begin{equation}\label{ePhi1Def}
	\Phi_1(x) = \frac{1}{4}\left(|x|^2 - \jvRe (x,C_+x)\right)
\end{equation}
with
\begin{equation}\label{eCplusDef}
	C_+ = 1+2iB = (1-iA_+)^{-1}(1+iA_+).
\end{equation}

We may see that $\Phi_1$ is strictly convex through the following useful computation, recalling that $A_+$ is symmetric:
\begin{multline}\label{e1MinusAbsCplus}
	1-C_+^*C_+ = 1-(1+iA_+^*)^{-1}(1-iA_+^*)(1+iA_+)(1-iA_+)^{-1}
	\\ = (1+iA_+^*)^{-1}\left( (1+iA_+^*)(1-iA_+) - (1-iA_+^*)(1+iA_+)\right)(1-iA_+)^{-1}
	\\ = (1+iA_+^*)^{-1}(4\jvIm A_+)(1-iA_+)^{-1}.
\end{multline}
We then see through a change of variables that positive definiteness of $\jvIm A_+$ is equivalent to positive definiteness of $1-C_+^*C_+$.  We therefore have that $|C_+x| < |x|$ for all $x \in \Bbb{C}^n\backslash \{0\}$.  Then, by the Cauchy-Schwarz inequality, when $x \neq 0$ we have
\[
	\Phi_1(x) \geq \frac{1}{4}(|x|^2 - |x|\:|C_+ x|) > 0,
\]
establishing strict convexity of $\Phi_1$.

The canonical transformation (\ref{ekapADef}) relates symbols $p:\Bbb{R}^{2n}\rightarrow \Bbb{C}$ with FBI-side symbols $(p\circ\kap_{A_+}^{-1}):\Lambda_{\Phi_1}\rightarrow \Bbb{C}$, with
\begin{equation}\label{eLambdaPhiDef}
	\Lambda_{\Phi_1} = \kap_{A_+}(\Bbb{R}^{2n}) = \left\{\left(x,\frac{2}{i}\partial_x \Phi_1(x)\right)\::\: x\in \Bbb{C}^n\right\}
\end{equation}
where derivatives here are holomorphic.  After conjugation with the FBI-Bargmann transform (\ref{eFBIT}), we have reduced $q$ to $(M_1x)\cdot\xi$, where $M_1$ is not necessarily in Jordan normal form.

Finally, for some invertible $G \in \Bbb{C}^{n\times n}$ chosen so that $G^{-1}M_1G$ is in Jordan normal form, we use a final linear change of variables 
\begin{equation}\label{eChangeOfVarsFinal}
	H_{\Phi_1} \ni u(x)\mapsto |\det G|u(Gx) \in H_{\Phi_2}
\end{equation}
quantizing the canonical transformation
\begin{equation}\label{ekappaGDef}
	\kappa_{G} : \Bbb{C}^{2n}\ni (x,\xi) \mapsto (G^{-1}x, G^t \xi) \in \Bbb{C}^{2n}.
\end{equation}
The resulting weight is $\Phi_2(x) = \Phi_1(Gx)$, which is strictly convex since $\Phi_1(x)$ is.

We note that a real-valued quadratic form $\Phi$ is uniquely determined by the two matrices $\Phi''_{xx}$ and $\Phi''_{\bar{x}x}$, since in this case
\[
	\Phi(x) = (\Phi''_{\bar{x}x}x,\bar{x}) + \jvRe (\Phi''_{xx}x,x).
\]
(See Section \ref{ssReversibility} for more details.)  We therefore only need to record that
\begin{equation}\label{ePhi2Matrix}
	(\Phi_2)''_{\bar{x}x} = \frac{1}{4} G^*G, \quad (\Phi_2)''_{xx} = -\frac{1}{4}G^tC_+G.
\end{equation}

The associated canonical transformation, using (\ref{ekappaDef}), (\ref{ekapADef}), and (\ref{ekappaGDef}), is
\[
	\kap = \kappa_G \circ \kap_A \circ \kappa.
\]
As in \cite{HiSjVi2011}, we note that $\kap|_{\Lambda^+}$ is an isomorphism
\[
	\kap:\Lambda^+\rightarrow \{\xi = 0\} = \{(x,0)\::\: x \in \Bbb{C}^n\}.
\]
Since
\[
	F(\tilde{q}) = \frac{1}{2}\left(\begin{array}{cc} M & 0 \\ 0 & -M^t\end{array}\right)
\]
and $F(\tilde{q}) = \kap F(q) \kap^{-1}$, we see that $M$ acting on $\Bbb{C}^n$ and $2F(q)$ acting on $\Lambda^+$ are similar linear operators and therefore isospectral.  By the definition (\ref{eLambdaDef}) of $\Lambda^+$, we then have
\[
	\opnm{Spec} M = \opnm{Spec} (2F(q)|_{\Lambda^+}) = (\opnm{Spec} 2F(q)) \cap \{\jvIm \lambda > 0\}.
\]

We furthermore remark that the change of variables (\ref{eChangeOfVarsFinal}) is a degree-preserving isomorphism on polynomials.  For this reason, it will sometimes be simpler to work on $H_{\Phi_1}$ instead of $H_{\Phi_2}$.
\end{proof}

\begin{remark}\label{rqtilde}
From Section 4 of \cite{HiSjVi2011} we record the specific formula
\[
	\tilde{q}^w(x,hD_x) = \tilde{q}_D^w(x,hD_x) + \tilde{q}^w_N(x,hD_x)
\]
with
\[
	\tilde{q}^w_D(x,hD_x) = \sum_{j=1}^n 2\lambda_j x_j hD_{x_j} + \frac{h}{i}\sum_{j=1}^n \lambda_j
\]
and
\[
	\tilde{q}^w_N(x,hD_x) = \sum_{j=1}^{n-1} \gamma_j x_{j+1}hD_{x_j}, \quad \gamma_j\in\{0,1\}.
\]
As usual, the $\lambda_j$ are the eigenvalues of $F(q)$ for which $\jvIm \lambda_j > 0$.  We furthermore remark that it is clear from the fact that $M$ is in Jordan normal form that $\gamma_j = 0$ when $\lambda_j \neq \lambda_{j+1}$.
\end{remark}

\begin{remark}\label{rGaussianInvariance}
In order to see how complex Gaussians
\begin{equation}\label{eGaussian}
	u_0(x) = \jvexp\left(\frac{i}{2h}(x,Fx)\right), \quad F \in \Bbb{C}^{n\times n}, \quad F^t = F
\end{equation}
transform under $\mathcal{T}$, or under any unitary transformation quantizing a complex linear canonical transformation $\kap$, it suffices to note that such a Gaussian may be uniquely identified (up to a constant factor) as an ODE solution of the equation
\[
	\ell^w(x,hD_x) u = 0, \quad \ell(x,\xi) = Fx-\xi.
\]
Therefore $\mathcal{T}u_0$ must satisfy a similar equation with $\tilde{\ell} = \ell \circ \kap^{-1}$; writing
\[
	\kap^{-1}(x,\xi) = (Ax+B\xi, Cx+D\xi),
\]
we have that
\[
	\tilde{\ell}(x,\xi) = (FA - C)x - (D - FB)\xi.
\]
When $(D-FB)^{-1}$ exists, we see that $\mathcal{T}u_0$ must also be a complex Gaussian 
\[
	\mathcal{T}u_0(x) = C\jvexp\left(\frac{i}{2h}(x,\tilde{F}x)\right)
\]
with a new 
\[
	\tilde{F} = (D-FB)^{-1}(FA-C).
\]
Symmetry of $\tilde{F}$, recalling that $F$ is symmetric, may be checked by noting that
\begin{multline*}
	(D-FB)(\tilde{F} - \tilde{F}^t)(D^t - B^tF) 
	\\ = F(BA^t - AB^t)F + F(AD^t - BC^t) + (CB^t - DA^t)F + (DC^t - CD^t).
\end{multline*}
That 
\[
	(\kap^{-1})^t \textnormal{ is canonical} \iff \left\{\begin{array}{l} AB^t - BA^t = 0 \\ AD^t - BC^t = 1 \\  CD^t - DC^t = 0 \end{array}\right.
\]
follows as usual from the equivalent statment
\[
	\kap^{-1} J (\kap^{-1})^t = J, \quad J = \left(\begin{array}{cc} 0 & -1 \\ 1 & 0\end{array}\right).
\]
This is seen to be equivalent to having $\kap$ canonical by taking inverses of both sides of $\kap^{-1} J (\kap^{-1})^t = J$, which completes the proof of symmetry of $\tilde{F}$.

The case $F = 0$, where $u_0$ is constant, will play an important role in the sequel as the ground state of $\tilde{q}^w(x,hD_x)$.  

The case where $D-FB$ above is not invertible should then degenerate into having $u_0$ behave as a delta function in certain directions, as may be seen by taking the Fourier transform as an example, but we will not encounter that situation here.  This fact may be deduced from the fact that the unitary transformation quantizing the real canonical transformation $\kappa$ in (\ref{ekappaDef}) preserves $L^2(\Bbb{R}^n)$ functions and therefore preserves the class of Gaussians given by $\{F\::\: F^t = F, \jvIm F > 0\}$ and from the fact that the FBI transform takes these Gaussians to entire functions on $\Bbb{C}^n$, precluding the delta function situation.

For completeness, particularly for the application to Hermite functions in Section \ref{ssSubelliptic}, we explicitly compute the matrix $DF-B$ for the transformations which constitute $\kap$ in Proposition (\ref{pNormalForm}). We then can see that $DF-B$ obtained from $\kap$ will be invertible whenever $\jvIm F > 0$ in the sense of positive definite matrices.  From (\ref{ekappaDef}) we have
\[
	\kappa^{-1} = \left(\begin{array}{cc} (-\jvIm A_-)^{-1/2} & 0 \\ (\jvRe A_-)(-\jvIm A_-)^{-1/2} & (-\jvIm A_-)^{1/2}\end{array}\right),
\]
meaning that $B = 0$ and so $D-FB = (-\jvIm A_-)^{1/2}$ which is always invertible.  Furthermore, for $\kappa$, we have
\[
	\tilde{F} = (-\jvIm A_-)^{-1/2}(F -\jvRe A_-)(-\jvIm A_-)^{-1/2},
\]
and since $(-\jvIm A_-)^{-1/2}$ is a real positive definite matrix, we see that $\jvIm F > 0$ if and only if $\jvIm \tilde{F} > 0$.  Next, from (\ref{ekapADef}), we see that
\[
	\kap_{A_+}^{-1} = \left(\begin{array}{cc} (1-iA_+)^{-1} & i \\ A_+(1-iA_+)^{-1} & 1\end{array}\right),
\]
and so $D-FB = 1-iF$ which is certainly invertible if $\jvIm F > 0$.  (One can furthermore easily check that $\jvIm F > 0$ means that the resulting $\tilde{F} = 0$ if and only if $\tilde{F} = A_+$.)  Finally, the transformation $\kappa_G^{-1}$ given by (\ref{ekappaGDef}) obviously has $B = 0$ and $D = G^{-t}$ which is always invertible; naturally, the formula for $\tilde{F}$ here may be more easily obtained from the associated change of variables.

We are therefore assured that $\mathcal{T}$ from Proposition \ref{pNormalForm} always carries a Gaussian given by (\ref{eGaussian}) to another Gaussian.

We make a final note that 
\[
	x^\alpha \jvexp(i(x,Fx)/2h)\stackrel{\mathcal{T}}\mapsto (k^w(x,hD_x))^\alpha \jvexp(i(x,\tilde{F}x)/2h)
\]
for some linear $k(x,\xi):\Bbb{R}^{2n}\rightarrow \Bbb{C}^n$.  As a consequence, if $p(x)$ is a polynomial of degree $N$,
\[
	p(x)\jvexp\left(\frac{i}{2h}(x,Fx)\right) \stackrel{\mathcal{T}}\mapsto \tilde{p}(x)\jvexp\left(\frac{i}{2h}(x,\tilde{F}x)\right)
\]
with $\tilde{p}$ a polynomial of degree less than or equal to $N$.  Because this procedure may reversed, we see that $\opnm{deg}\tilde{p} = N$.
\end{remark}

\begin{remark}\label{rPolynomialsDense}
In our first application of the proof of Proposition \ref{pNormalForm} and Remark \ref{rGaussianInvariance}, we can now easily show that the set $\Bbb{C}[x_1,\dots,x_n]$ of polynomials in $n$ variables is dense in $H_{\Phi_2}(\Bbb{C}^n;h)$.  Since the invertible linear change of variables (\ref{eChangeOfVarsFinal}) induces an isomorphism on the space of polynomials, it suffices to show that polynomials are dense in $H_{\Phi_1}$.  As discussed in Remark \ref{rGaussianInvariance}, the constant functions are uniquely determined by the equation
\[
	\ell^w(x,hD_x)u = 0, \quad \ell(x,\xi) = \xi.
\]
Inverting the FBI-Bargmann transform quantizing $\kap_{A_+}$ provides a unitary map
\[
	H_{\Phi_1}(\Bbb{C}^n;h)\supseteq \Bbb{C}[x_1,\dots,x_n] \rightarrow \Bbb{C}[x_1,\dots,x_n]u_0 \subseteq L^2(\Bbb{R}^n),
\]
where the image of $1 \in H_{\Phi_1}(\Bbb{C}^n;h)$, denoted $u_0 \in L^2(\Bbb{R}^n)$, is uniquely determined up to constants by the equation
\[
	(\ell \circ \kap_{A_+})^w(x,hD_x)u_0 = 0.
\]
We compute that
\[
	\ell\circ\kap_{A_+} = (1-iA_+)^{-1}(-A_+x + \xi),
\]
so (up to a constant factor)
\[
	u_0(x) = \jvexp\left(\frac{i}{2h}(x,A_+x)\right).
\]
Density of $\Bbb{C}[x_1,\dots,x_n]u_0$ in $L^2(\Bbb{R}^n)$ is established in Lemma 3.12 of \cite{Sj1974}.  Density of $\Bbb{C}[x_1,\dots,x_n]$ in $H_{\Phi_1}(\Bbb{C}^n;h)$ follows by the unitary equivalence, completing the proof.

In view of Proposition \ref{pReversibility}, we see that polynomials are dense in $H_\Phi(\Bbb{C}^n;h)$ for any real-valued quadratic strictly convex $\Phi$.
\end{remark}

\subsection{Examples}\label{ssExamples}

We begin by describing the reduction to normal form in three examples: the rotated harmonic oscillator studied in \cite{DaKu2004}, a Kramers-Fokker-Planck operator with quadratic potential like those studied in \cite{HeSjSt2005} (among many other works), and small perturbations of an operator for which $F(q)$ has Jordan blocks, studied in \cite{HiSjVi2011}.  The rotated harmonic oscillator is a model operator in dimension 1 against which Corollary \ref{cDim1Asymptotic} may be checked. Next, the Kramers-Fokker-Planck operator is of physical interest, is partially elliptic but not elliptic, and may be analyzed via an orthogonal decomposition in view of Theorem \ref{tOrthogonalProj}.  Finally, the perturbations of the operator with a Jordan block admit a similar decomposition, and the unperturbed operator has rapid resolvent growth while having orthogonal spectral projections with large ranges.

The reader who wishes to continue with the proof of Theorem \ref{tExtendToPE} is invited to skip this section.

\begin{example}\label{exDavies1} We first consider the operator on $L^2(\Bbb{R}^1)$ given by
\begin{equation}\label{eDavies}
	Q_1(h) = -e^{-2i\theta}h^2\frac{d^2}{dx^2} + e^{2i\theta} x^2
\end{equation}
with symbol
\begin{equation}\label{eDaviesSymbol}
	q_1(x,\xi) = e^{-2i\theta}\xi^2 + e^{2i\theta}x^2.
\end{equation}
The symbol is elliptic for $\theta \in (-\pi/4, \pi/4)$.  In \cite{DaKu2004}, exact rates of the exponential growth for the spectral projections associated with this operator were computed, and we will compare the results in this paper to those previously known results.

We begin with the observation that
\[
	F = \left(\begin{array}{cc} 0 & e^{-2i\theta} \\ -e^{2i\theta} & 0\end{array}\right),
\]
with eigenvalues $\pm i$ and eigenspaces
\[
	\Lambda^{\pm} = \{(y,\pm ie^{2i\theta}y)\::\:y\in\Bbb{C}^n\}.
\]

We have then that $\Lambda^- = \{(y, A_- y)\}_{y\in\Bbb{C}^n}$ for $A_- = -ie^{2i\theta}$, and so following (\ref{ekappaDef}) gives
\[
	\kappa = \left(\begin{array}{cc} (\cos 2\theta)^{1/2} & 0 \\ -(\cos 2\theta)^{-1/2}(\sin 2\theta) & (\cos 2\theta)^{-1/2}\end{array}\right).
\]
As a consequence,
\[
	\kappa(\Lambda^-) = \{(y,-iy)\}_{y\in\Bbb{C}^n}, \quad \kappa(\Lambda^+) = \{(y, (i-2\tan 2\theta)y)\}_{y\in\Bbb{C}^n}.
\]
Therefore we can compute from
\[
	A_+ = i-2\tan 2\theta
\]
that
\[
	C_+ = (1-iA_+)^{-1}(1+iA_+) = \frac{1}{2}(-1+e^{-4i\theta}).
\]
In one dimension, there is no need for an FBI-side change of variables reducing to Jordan normal form, so it suffices to use
\[
	\Phi_1(x) = \frac{1}{4}(|x|^2 - \jvRe(x,C_+x)).
\]
We conclude that $Q_1(h)$ from (\ref{eDavies}) is unitarily equivalent to
\[
	\tilde{Q}_1(h) = 2i xhD_x + h
\]
acting on $H_{\Phi_1}(\Bbb{C}^1)$.  We note that the symbol of $\tilde{Q}_1(h)$ is $\theta$-independent.
\end{example}

\begin{example}\label{exKFP1}
We next consider a specific semiclassical Kramers-Fokker-Planck operator with quadratic potential
\begin{equation}\label{eKFPOperator}
	Q_2(h) = \frac{1}{2}(v^2 - h^2\partial_v^2) + hv\partial_x - \frac{1}{2}hx\partial_v
\end{equation}
acting on $L^2(\Bbb{R}^2)$, for simplicity.  The classical derivation from \cite{Kr1940} may be found in, for instance, Section 13 of \cite{HeSjSt2005}.  We have chosen $\gamma = 1$ and $V(x) = x^2/4$.  The corresponding symbol is
\[
	q_2(x,v,\xi,\eta) = \frac{1}{2}(v^2 + \eta^2) + i(v\xi - \frac{1}{2}x\eta).
\]
In this situation we have
\[
	F_2 = \left(\begin{array}{cccc} 0 & i/2 & 0 & 0 \\ -i/4 & 0 & 0 & 1/2 \\ 0 & 0 & 0 & i/4 \\ 0 & -1/2 & -i/2 & 0\end{array}\right).
\]

We may then check that the eigenvalues of $F_2$ are given by
\[
	\opnm{Spec}(F_2) = \left\{\pm \frac{1}{4} \pm \frac{i}{4}\right\},
\]
with eigenvectors determined by
\[
	\ker(F_2-\lambda) = \opnm{span}v_\lambda, \quad v_\lambda := \left(1,\frac{2\lambda}{i},\frac{i}{4\lambda}\left(\frac{4\lambda^2}{i}+\frac{i}{2}\right),\frac{4\lambda^2}{i}+\frac{i}{2}\right).
\]
We can easily determine $A_-$ by writing
\[
	\left(\begin{array}{ccc} - & v_{1/4-i/4} & \rightarrow \\ - & v_{-1/4-i/4} & \rightarrow \end{array}\right) = \left(\begin{array}{cc}B_1 & B_2\end{array}\right)
\]
for $B_1, B_2 \in \Bbb{C}^{2\times 2}$.  Then 
\[
	A_- = B_1^{-1}B_2 = \left(\begin{array}{cc} -i/2 & 0 \\ 0 & -i\end{array}\right).
\]
The same procedure applied to eigenvalues with positive imaginary part shows that $\Lambda^+(q_2) = \overline{\Lambda^-(q_2)}$, which is condition \ref{iOrthogbar} in Theorem \ref{tOrthogonalProj}.

We therefore take
\[
	\kappa = \left(\begin{array}{cccc} 1/\sqrt{2} & 0 & 0 & 0 \\ 0 & 1 & 0 & 0 \\ 0 & 0 & \sqrt{2} & 0 \\ 0 & 0 & 0 & 1\end{array}\right)
\]
to obtain
\[
	q_{2,R}(x,v,\xi,\eta) := (q_2\circ\kappa^{-1})(x,v,\xi,\eta) = \frac{1}{2}(v^2 + \eta^2) +\frac{i}{\sqrt{2}}(v\xi - x\eta).
\]
The same procedure applied to $F(q_{2,R}) = \kappa F(q_2)\kappa^{-1}$ shows that
\[
	A_+ = \left(\begin{array}{cc}i & 0 \\ 0 & i\end{array}\right).
\]
We therefore may conjugate with an FBI-Bargmann transform quantizing the linear transformation
\[
	\kap_i = \left(\begin{array}{cccc} 1 & 0 & -i & 0 \\ 0 & 1 & 0 & -i \\ -i/2 & 0 & 1/2 & 0 \\ 0 & -i/2 & 0 & 1/2\end{array}\right)
\]
obtained from (\ref{ekapADef}).  This results in the symbol
\begin{equation}\label{eKFPPhi1}
	\tilde{q}_{2,0}(x,v,\xi,\eta) = (q_{2,R}\circ \kap_i^{-1})(x,v,\xi,\eta) = iv\eta -\frac{i}{\sqrt{2}}x\eta + \frac{i}{\sqrt{2}}v\xi,
\end{equation}
where $\tilde{q}_{2,0}^w(x,v,hD_x,hD_v)$ is viewed as an operator on
\[
	H_{\Phi_1}(\Bbb{C}^2;h), \quad \Phi_1(x) = \frac{1}{4}|(x,v)|^2
\]
since $C_+ = 0$ from (\ref{eCplusDef}).  Note that
\[
	\tilde{q}_{2,0}(x,v,\xi,\eta) = (M_1(x,v))\cdot(\xi,\eta),\quad M_1 = \left(\begin{array}{cc} 0 & i/\sqrt{2} \\ -i/\sqrt{2} & i\end{array}\right),
\]
and that the eigenvalues of $M_1$ are the same as those of $2F(q)$ which lie in the upper half-plane.

Then the matrix $G$ corresponding to the change of variables $\kappa_G$ may be computed via the assumption that
\[
	(\tilde{q}_{2,0}\circ \kappa_G^{-1})(x,v,\xi,\eta) = (M(x,v))\cdot(\xi,\eta)
\]
where $M$ is diagonal because the eigenvalues of $M_1$ are distinct.  Letting
\begin{equation}\label{eGKFP}
	G = \left(\begin{array}{cc} (1+i)/\sqrt{2} & (1-i)/\sqrt{2} \\ 1 & 1\end{array}\right)
\end{equation}
gives our final result, that $Q_2(h)$ is unitarily equivalent to the operator
\[
	\tilde{Q}_2(h) = \tilde{q}_2^w(x,v,hD_x,hD_v)
\]
for
\[
	\tilde{q}_2(x,v,\xi,\eta) = \frac{1+i}{2}x\xi + \frac{-1+i}{2}v\eta
\]
acting on $H_{\Phi_2}(\Bbb{C}^2;h)$ with
\[
	\Phi_2(x,v) = \frac{1}{4}|G(x,v)|^2 = \frac{1}{2}|(1+i)x+(1-i)v|^2+|x+v|^2.
\]
We also note that our choice of $G$ was not unique.  For instance, replacing $G$ with $rG$ for any $r > 0$ would suffice.
\end{example}

\begin{example}\label{exJordan1}
As a final example, we discuss an elliptic quadratic form whose Hamilton map contains Jordan blocks, also studied in Section 4 of \cite{HiSjVi2011}.  The example is notable for exhibiting rapid resolvent growth while having orthogonal spectral projections whose ranges have high dimension.  Under small perturbations which split the eigenvalues of the Hamilton map, the resulting split spectral projections are no longer orthogonal and in fact have norms with rapid exponential growth in $1/h$.

The easiest place to begin is on the FBI transform side, as we are free to write
\begin{equation}\label{eJordanSymbol}
	\tilde{q}_3(x,\xi) = (Mx)\cdot\xi, \quad (x,\xi) \in \Bbb{R}^4
\end{equation}
for $M$ already in Jordan normal form:
\[
	M = \left(\begin{array}{cc} 2i & 1 \\ 0 & 2i\end{array}\right).
\]
We regard $\tilde{q}_3^w(x,hD_x)$ as acting on
\[
	H_{\Phi_1}(\Bbb{C}^2;h), \quad \Phi_1(x) = \frac{1}{4}|x|^2.
\]
If one wishes, one may invert the FBI-Bargmann transform and canonical transformation used in Example \ref{exKFP1} to obtain a unitarily equivalent operator on $L^2(\Bbb{R}^2)$; a similar formula in terms of creation-annihilation operators is in found in \cite{HiSjVi2011}.

As an example of a small perturbation, we consider the operator
\[
	\tilde{q}_{3,\eps}(x,\xi) = \tilde{q}_3(x,\xi) + \eps x_2 \xi_2.
\]
Instead of passing back to the real side and beginning to straighten $\Lambda^-$ anew, we merely note that 
\[
	M_\eps = \left(\begin{array}{cc} 2i & 1 \\ 0 & 2i+\eps\end{array}\right)
\]
may be diagonalized by the change of variables given by
\begin{equation}\label{eGJordan}
	G_\eps = \left(\begin{array}{cc}1 & 1 \\ 0 & \eps \end{array}\right).
\end{equation}
Ellipticity of the operator on $H_{\Phi_1}$ given by $(M_\eps x)\cdot\xi$ when $\eps \in \Bbb{R}$ is easy to check on the FBI transform side.  By (\ref{eLambdaPhiDef}), we have that 
\[
	\xi(x) = \frac{2}{i}\partial_x \Phi_1(x) = \frac{2}{i}\partial_x\left(\frac{1}{4}x\bar{x}\right) = \frac{1}{2i}\bar{x},
\]
and so
\[
	\jvRe((M_\eps x)\cdot\xi) = |x|^2 + \jvRe\left(\frac{1}{2i}x_2 \bar{x}_1\right),\quad \forall (x,\xi) \in \Lambda_{\Phi_1}.
\]
Ellipticity follows from the Cauchy-Schwarz inequality.

We will see in Section \ref{ssTransferringdxStar} that, in this case, Theorem \ref{tBoundProj} depends on
\[
	\Phi_2(x) = \frac{1}{4}|G_\eps x|^2, \quad \Phi_2^\dagger(x) = \frac{1}{4}|(G_\eps^*)^{-1}x|^2.
\]
This means that the large condition number $||G_\eps||\,||(G_\eps^*)^{-1}||$, which occurs because $M_\eps$ is nearly a Jordan block, should result in a very large rate of exponential growth in $h^{-1}$ for spectral projections for $\tilde{q}_{3,\eps}^w(x,hD_x)$.  This dependence is verified for $0 < \eps \ll 1$ in Section \ref{ssNumericsExponentialGrowth}.

We furthermore may examine the case where the operator $\tilde{q}_{3,0}^w(x,hD_x)$, obtained by setting $\eps = 0$, acts instead on $H_{\Phi_2}$ for $\Phi_2(x) = \frac{1}{4}|Gx|^2$ where $G \in \Bbb{C}^{2\times 2}$ is an invertible matrix for which $\tilde{q}_{3,0}(x,\xi)$ is elliptic along $\Lambda_{\Phi_2}$.  (While this ellipticity condition is necessary, the collection of such $G$ is certainly an open set containing the identity matrix and is therefore nontrivial.)  By Theorem \ref{tCharacterizeProj}, we see that the ranges of spectral projections associated with $\tilde{q}_{3,0}^w(x,hD_x)$ are precisely the spaces
\[
	E_m = \opnm{span}\{x^\alpha\::\: |\alpha| = m\} \subset H_{\Phi_2}.
\]
The spectral projections of this non-selfadjoint operator are orthogonal, yet Corollaries \ref{cUpperBound} and \ref{cElementaryExponential} indicate exponential growth.  Furthermore, following \cite{HiSjVi2011}, we expect the resolvent of the operator to be quite large, particularly when close to the spectrum.  This demonstrates the significant complications in the case where $F(q)$ has Jordan blocks, where the dimension of the range of spectral projections becomes large.

\end{example}

\subsection{A high-energy subelliptic estimate}\label{ssSubelliptic}

The elliptic estimate in Section 3 of \cite{HiSjVi2011}, like those in many other works, relies essentially upon a lower bound for $\jvRe \langle u, \tilde{q}^w(x,hD_x)u\rangle_{H_{\Phi_2}}$ for $u\in H_{\Phi_2}$ supported away from the origin, and this lower bound is a consequence of a lower bound for the symbol on $\opnm{supp}u$.  Lower bounds for $\tilde{q}^w(x,hD_x) - z$ are then obtained from the triangle inequality.  As we have already seen in Example \ref{exKFP1}, quadratic symbols satisfying (\ref{eRealSemidef}) and (\ref{eTrivialS}) are generally not bounded from below away from the origin.  Numerics presented in Section \ref{ssNumerics} suggest that the lower bounds which hold for elliptic operators are false in the partially elliptic case, since exponential resolvent growth appears to persist for energies $C/h$ even when $C$ is taken large.

Instead, we may use a result of Pravda-Starov \cite{PS2011}.  In the elliptic case, one may bound $\tilde{q}^w(x,D_x)$ from below (with error) by $x^2+D_x^2$, which corresponds to the symbol bound
\[
	q(x,\xi) \geq \frac{1}{C}|(x,\xi)|^2.
\]
In the non-elliptic case, one is forced to accept a lower bound given by a more slowly-growing symbol.

We recall from Theorem 1.2.1 of \cite{PS2011} the non-semiclassical estimate
\[
	||(\langle (x,\xi)\rangle^{2/(2k_0+1)})^w u||_{L^2(\Bbb{R}^n)}\leq C(||q^w(x,D_x)u||_{L^2(\Bbb{R}^n)} + ||u||_{L^2(\Bbb{R}^n)}).
\]
Here, $k_0$ is defined in (\ref{ek0Def}).  One might convert this to a semiclassical estimate via conjugation with the usual unitary change of variables (\ref{eScaleOuthOp}).  This induces a unitary equivalence between
\[
	q^w(x,D_x) \sim h^{-1}q^w(x,hD_x)
\]
and
\[
	(\langle(x,\xi)\rangle^{2/(2k_0+1)})^w \sim h^{-1/(2k_0+1)}p_0^w(x,hD_x)
\]
for
\[
	p_0(x,\xi) = (h+x^2+\xi^2)^{1/(2k_0+1)}.
\]
Derivatives of this symbol grow rapidly at the origin as $h\rightarrow 0^+$, and therefore one could possibly modify the symbol calculus in some way to obtain lower bounds directly.

To avoid these issues, we follow \cite{OtPaPS2012} in passing to the functional calculus for the operator
\[
	\Lambda^2 = 1+x^2+D_x^2,
\]
since the study of the semiclassical harmonic oscillator $x^2+(hD_x)^2$ falls well within the subject of the present work.  From (43) in \cite{OtPaPS2012} with $\nu = 0$ and $\delta = 2k_0/(2k_0+1)$ for $k_0$ as in (\ref{ek0Def}), we have
\[
	||(\Lambda^2)^{1-\delta} u||_{L^2(\Bbb{R}^n)} \leq C\left(||q^w(x,D_x)u||_{L^2(\Bbb{R}^n)} + ||u||_{L^2(\Bbb{R}^n)}\right).
\]
Using the same unitary change of variables (\ref{eScaleOuthOp}), which may be passed inside the exponent $1-\delta$ via the functional calculus, we have
\begin{equation}\label{eFunctionalKarel}
	||(\Lambda_h^2)^{1-\delta} u||_{L^2(\Bbb{R}^n)}\leq C\left(h^{-1}||q^w(x,hD_x)u||_{L^2(\Bbb{R}^n)} + ||u||_{L^2(\Bbb{R}^n)}\right)
\end{equation}
for
\begin{equation}\label{eHOLambdah}
	\Lambda_h^2 = 1+h^{-1}(x^2+(hD_x)^2).
\end{equation}

We recall that, for all $\alpha \in \Bbb{N}_0^n$, there exist Hermite polynomials $f_\alpha(x;h)$ of degree $|\alpha|$ where $\{f_\alpha(x;h) e^{-x^2/2h}\}_{\alpha \in \Bbb{N}_0^n}$ orthonormally diagonalizes $x^2+(hD_x)^2$ acting on $L^2(\Bbb{R}^n)$, with eigenvalues
\[
	(x^2+(hD_x)^2)f_\alpha(x;h)e^{-x^2/2h} = h(2|\alpha|+n)f_\alpha(x;h)e^{-x^2/2h}.
\]
Explicitly, we recall that $f_\alpha(x;h)e^{-x^2/(2h)}$ may be obtained as the normalization in $L^2(\Bbb{R}^n)$ of $(x-ihD_x)^\alpha e^{-x^2/(2h)}$. Conjugating with the unitary transformation in Proposition \ref{pNormalForm} which maps $L^2(\Bbb{R}^n)$ to $H_{\Phi_2}(\Bbb{C}^n)$ takes the collection $\{f_\alpha e^{-x^2/2h}\}_{\alpha \in \Bbb{N}_0^n}$ to some orthonormal collection of
\begin{equation}\label{evalpha}
	v_\alpha(x) = p_\alpha(x;h)e^{-Q_0(x)/h}
\end{equation}
in $H_{\Phi_2}(\Bbb{C}^n)$ with $\opnm{deg}p_\alpha = |\alpha|$.  (See Remark \ref{rGaussianInvariance}.)  Furthermore, because 
\[
	v_0(x) = Ce^{-Q_0(x)/h} \in H_{\Phi_2}
\]
for some $h$-dependent constant $C$, we have that
\[
	||v_0||_{H_\Phi} = |C|^2 \int_{\Bbb{C}^n} \jvexp\left(-\frac{1}{h}Q_0(x) - \frac{1}{h}\overline{Q_0(x)} - \frac{2}{h}\Phi_2(x)\right)\,dL(x) < \infty.
\]
We conclude from this that
\begin{equation}\label{etildePhiEllipticity}
	\tilde{\Phi}(x) := \Phi_2(x) + \jvRe Q_0(x)
\end{equation}
is a strictly convex real-valued quadratic weight.

As in \cite{HiSjVi2011}, we divide $H_{\Phi_2}$ into low and high energy subspaces, where the energy of a monomial $x^\alpha$ is considered to be $h^{-1}|\alpha|$.  We will here use the notation
\begin{equation}\label{eLowEnergyDef}
	\mathcal{K}_{C_0} = \mathcal{K}_{C_0, \Phi} = \{u\in H_{\Phi}\::\: \partial^\alpha u(0) = 0, \forall |\alpha| > C_0h^{-1}\},
\end{equation}
\begin{equation}\label{eHighEnergyDef}
	\mathcal{H}_{C_0} = \mathcal{H}_{C_0, \Phi} = \{u\in H_{\Phi}\::\: \partial^\alpha u(0) = 0, \forall |\alpha| \leq C_0h^{-1}\}.
\end{equation}
Recall that $\mathcal{K}_{C_0}$ and $\mathcal{H}_{C_0}$ (with weight function $\Phi_2$) are $\tilde{q}^w(x,hD_x)$-invariant.  We effectively repeat the proof of Proposition 3.2 in \cite{HiSjVi2011}, making simple modifications, to establish a characterization of the localization of high- and low-energy functions in terms of the spectral projections for a semiclassical harmonic oscillator.  We remark that the semiclassical harmonic oscillator could easily be replaced by the Weyl quantization of any real-valued elliptic quadratic form with only minor changes due to changing eigenvalues of the operator.  We also remark that the upper bound $e^{-1/(4h)}$ is essentially arbitrary in that it could be replaced with $e^{-C/h}$ for any fixed $C$ at the price of increasing $C_0$.

Some of the computations in polar coordinates below were inspired by similar considerations in Lemma 4.4 of \cite{KaKe2000}.

\begin{lemma}\label{lEllipticSpectralProj}
Let $\Phi:\Bbb{C}^n\rightarrow \Bbb{R}$ be strictly convex and quadratic, and let $\tilde{A}$ be an unbounded operator on $H_\Phi(\Bbb{C}^n)$ which is unitarily equivalent to the semiclassical harmonic oscillator $x^2+(hD_x)^2$ acting on $L^2(\Bbb{R}^n)$.  Specifically, assume that there exists some $Q_0:\Bbb{C}^n\rightarrow \Bbb{C}$ holomorphic and quadratic and $\{p_\alpha(x;h)\}_{\alpha\in \Bbb{N}_0^n}$ holomorphic polynomials for which the following hold:
\begin{itemize}
	\item the weight given by $\tilde{\Phi}(x) := \Phi(x) + \jvRe Q_0(x)$ is strictly convex,
	\item for each $\alpha \in \Bbb{N}^n_0$, we have $\opnm{deg}p_\alpha = |\alpha|$, 
	\item the collection $\{v_\alpha\}_{\alpha \in \Bbb{N}_0^n}$ defined via $v_\alpha(x) = p_\alpha(x)e^{-Q_0(x)/h}$ form an orthonormal basis for $H_\Phi$, and
	\item the $\{v_\alpha\}$ obey
	\[
		\tilde{A}v_\alpha = h(2|\alpha|+n)v_\alpha, \quad \forall \alpha\in\Bbb{N}_0^n,
	\]
	and thus diagonalize $\tilde{A}$.
\end{itemize}
With $C_1 > 0$, write $\Pi_{C_1}$ for the spectral projection associated with $\tilde{A}$ and the interval $[0,C_1]$; this may be written explicitly as
\begin{equation}\label{eHOProjDecomp}
	\Pi_{C_1}u = \sum_{\alpha\::\:2|\alpha|+n \leq C_1h^{-1}} \langle u, v_\alpha\rangle_{H_\Phi}v_\alpha.
\end{equation}

Then, for every $C_1 > 0$, there exists $C_0 > 0$ sufficiently large and $h_0 > 0$ sufficiently small for which, for all $h \in (0,h_0]$ and with $\mathcal{H}_{C_0}$ defined in (\ref{eHighEnergyDef}), we have the estimate
\[
	||\Pi_{C_1} u||_{H_\Phi} \leq \BigO(1)e^{-1/(4h)}||u||_{H_\Phi}, \quad \forall u \in \mathcal{H}_{C_0}.
\]
\end{lemma}

\begin{proof} Because the $\{v_\alpha\::\: 2|\alpha|+n \leq C_1h^{-1}\}$ are low-energy and the $u\in \mathcal{H}_{C_0}$ are high-energy, it is natural to use a radial cutoff function (which does not need to be smooth) and the Cauchy-Schwarz inequality:
\begin{eqnarray*}
	|\langle u,v_\alpha\rangle_{H_\Phi}| &\leq& |\langle \un_{\{|x|\leq K\}}u, v_\alpha\rangle_{L^2_\Phi}| + |\langle u, \un_{\{|x|>K\}} v_\alpha\rangle_{L^2_\Phi}|
	\\ &\leq & ||\un_{\{|x|\leq K\}}u||_{L^2_\Phi} + ||\un_{\{|x|>K\}}v_\alpha||_{L^2_\Phi}\,||u||_{H_\Phi}.
\end{eqnarray*}
From (3.3) in \cite{HiSjVi2011} we have the estimate
\[
	||\un_{\{|x|\leq K\}}u||_{L^2_\Phi} \leq \BigO_K(1)e^{-1/(2h)}||u||_{H_\Phi}
\]
when $u \in \mathcal{H}_{C_0}$ for $C_0$ taken sufficiently large depending on $K$.

We will establish the corresponding bound
\begin{equation}\label{eLowEnergyFinalLowerBound}
	||\un_{\{|x| > K\}}v_\alpha||_{L^2_\Phi} \leq \BigO_{C_1}(1)e^{-1/(2h)}
\end{equation}
for $K$ sufficiently large depending on $C_1$ and for all $\alpha$ with $2|\alpha|+n \leq C_1h^{-1}$.  With these two bounds, we may choose $C_0$ sufficiently large depending on the $K$ obtained from $C_1$ to ensure that
\[
	\langle u,v_\alpha\rangle_{H_\Phi} \leq \BigO_{C_1}(1)e^{-1/(2h)}||u||_{H_\Phi},
\]
uniformly when $u \in \mathcal{H}_{C_0}$, when $2|\alpha|+n\leq C_1h^{-1}$, and when $h$ is sufficiently small.  The lemma then follows from the simple observation that there are at most $\BigO(h^{-n})$ such $\alpha$: since $h^{-n}e^{-1/(2h)} \ll e^{-1/(4h)}$ as $h\rightarrow 0^+$, the lemma is established by the triangle inequality.

The bound (\ref{eLowEnergyFinalLowerBound}) follows from the same method as in the proof of Proposition 3.3 in \cite{HiSjVi2011}.  We begin by switching to the weight $\tilde{\Phi}$ by noting that
\[
	||\un_{\{|x| > K\}}v_\alpha||_{L^2_\Phi} = ||\un_{\{|x| > K\}}p_\alpha||_{L^2_{\tilde{\Phi}}}
\]
and that $\{p_\alpha\}_{\alpha\in\Bbb{N}_0^n}$ forms an orthonormal sequence in $H_{\tilde{\Phi}}$.  Furthermore, it follows from the assumption $2|\alpha|+n \leq C_1h^{-1}$ that $p_\alpha \in \mathcal{K}_{C_1/2,\tilde{\Phi}}$.

Strict convexity of $\tilde{\Phi}$ means that there exist $C_\ell, C_u > 0$ for which
\[
	C_\ell |x|^2 \leq \tilde{\Phi}(x) \leq C_u |x|^2, \quad \forall x \in \Bbb{C}^n.
\]
We will write $\Phi_\ell(x) = C_\ell |x|^2$ and $\Phi_u(x) = C_u|x|^2$.  The reasoning leading to (3.14) in \cite{HiSjVi2011} shows that, when $w = \sum_{|\alpha| \leq N}a_\alpha x^\alpha$ is a polynomial,
\begin{equation}\label{eLowEnergyLowerBound}
	||w||_{H_{\tilde{\Phi}}}^2 \geq ||w||_{H_{\Phi_u}}^2 = \sum_{|\alpha|\leq N} |a_\alpha|^2 \left(\frac{h}{2C_u}\right)^{n+|\alpha|}\pi^n\alpha!.
\end{equation}
To obtain the reverse estimate for $||\un_{\{|x|>K\}}w||_{L^2_{\tilde{\Phi}}}^2$, we use a similar bound with the radial weight $\Phi_\ell$, which is convenient because then $x^\alpha \perp x^\beta$ when $\alpha \neq \beta$:
\begin{equation}\label{eLowEnergyNearInfinity}
	||\un_{\{|x|>K\}}w||_{L^2_{\tilde{\Phi}}}^2 \leq ||\un_{\{|x|>K\}}w||_{L^2_{\Phi_\ell}}^2 = \sum_{|\alpha|\leq N} |a_\alpha|^2\int_{|x|>K} |x^\alpha|^2e^{-2C_\ell |x|^2/h}\,dL(x).
\end{equation}

In one (complex) dimension and integrating on $\{|x_j| > R\}$, we obtain an upper bound with exponential decay by factoring out the maximum value of $e^{-C_\ell |x_j|^2/h}$ on $\{R \leq |x_j| < \infty\}$:
\begin{eqnarray*}
	\int_{|x_j| > R} |x_j|^{2\alpha_j} e^{-\frac{2}{h}C_\ell |x_j|^2}\,dL(x_j) &=& 2\pi \int_R^\infty r^{2\alpha_j+1}e^{-2C_\ell r^2/h}\,dr
	\\ &\leq& 2\pi e^{-C_\ell R^2/h}\int_R^\infty r^{2\alpha_j+1} e^{-C_\ell r^2/h}\,dr
	\\ &\leq& 2\pi e^{-C_\ell R^2/h}\int_0^\infty r^{2\alpha_j+1} e^{-C_\ell r^2/h}\,dr
	\\ &=& 2\pi e^{-C_\ell R^2/h} 2^{\alpha_j} \alpha_j ! \left(\frac{h}{2C_\ell}\right)^{\alpha_j+1}.
\end{eqnarray*}
Returning to $\Bbb{C}^n$, we note that in order to have $|x| > K$ there must exist at least one $j$ for which $|x_j| > Kn^{-1/2}$.  We therefore estimate
\begin{multline*}
	\int_{|x| > K} |x^\alpha|^2 e^{-\frac{2}{h}C_\ell |x|^2}\,dL(x) \leq \sum_{j=1}^n \int_{|x_j| > Kn^{-1/2}} |x^\alpha|^2 e^{-\frac{2}{h}C_\ell |x|^2}\,dL(x)
	\\ = \sum_{j=1}^n \left(\int_{|x_j|>Kn^{-1/2}}|x_j|^{2\alpha_j}e^{-\frac{2}{h}C_\ell |x_j|^2}\,dL(x_j) \prod_{k\neq j} \int_{\Bbb{C}}|x_k|^{2\alpha_k}e^{-\frac{2}{h}C_\ell |x_k|^2}\,dL(x_k)\right)
	\\ \leq \sum_{j=1}^n \left(2\pi e^{-C_\ell K^2/(nh)} 2^{\alpha_j} \alpha_j ! \left(\frac{h}{2C_\ell}\right)^{\alpha_j+1}\prod_{k\neq j} (2\pi)2^{\alpha_k} \alpha_k ! \left(\frac{h}{2C_\ell}\right)^{\alpha_k+1}\right)
	\\ \leq n e^{-C_\ell K^2/(nh)} \left(\frac{h}{C_\ell}\right)^{n+|\alpha|}\pi^n\alpha!.
\end{multline*}
Applying this to (\ref{eLowEnergyNearInfinity}) yields the estimate
\begin{equation}\label{eLowEnergyNearInfinityUpperBound}
	||\un_{\{|x| > K\}} w||^2_{L^2_{\tilde{\Phi}}}\leq \sum_{|\alpha|\leq N} |a_\alpha|^2 n e^{-C_\ell K^2/(nh)} \left(\frac{h}{C_\ell}\right)^{n+|\alpha|}\pi^n\alpha!.
\end{equation}

Uniformly in $\alpha$ for which $|\alpha| \leq h^{-1}C_1/2$, we have that
\[
	\left(\frac{h}{C_\ell}\right)^{n+|\alpha|} \leq \left(\frac{2C_u}{C_\ell}\right)^{n+C_1/(2h)}\left(\frac{h}{2C_u}\right)^{n+|\alpha|}.
\]
Combining (\ref{eLowEnergyLowerBound}) and (\ref{eLowEnergyNearInfinityUpperBound}) with $w = p_\alpha \in \mathcal{K}_{C_1/2,\tilde{\Phi}}$ then yields
\[
	||\un_{\{|x|> K\}}p_\alpha||_{L^2_{\tilde{\Phi}}}^2 \leq n e^{-C_\ell K^2/(nh)} \left(\frac{2C_u}{C_\ell}\right)^{n+C_1/(2h)}||p_\alpha||_{H_{\tilde{\Phi}}}^2.
\]
The corresponding estimate holds upon replacing $p_\alpha \in H_{\tilde{\Phi}}$ with $v_\alpha \in H_\Phi$ and changing norms accordingly.  Taking a square root and a logarithm reveals that (\ref{eLowEnergyFinalLowerBound}) is certainly established for $h$ sufficiently small if
\[
	\frac{1}{2}\left(-\frac{C_\ell}{n} K^2 + \frac{C_1}{2}\log\left(\frac{2C_u}{C_\ell}\right)\right) < -\frac{1}{2}.
\]
This is accomplished by setting $K$ sufficiently large, and this completes the proof.
\end{proof}

With the preceding lemma, we are in a position to prove an elliptic-type estimate upon restricting to $u \in \mathcal{H}_{C_0}$.  However, the weaker ellipticity forces us to choose our spectral parameter in a set of size $h^{\delta} \ll 1$.

\begin{proposition}\label{pPartiallyEllipticEstimate} Let $q(x,\xi):\Bbb{R}^{2n}\rightarrow \Bbb{C}$ be a quadratic form which is partially elliptic and has trivial singular space in the sense of (\ref{eRealSemidef}) and (\ref{eTrivialS}).  Let $\Phi_2:\Bbb{C}^n\rightarrow \Bbb{R}$, strictly convex, real-valued, and quadratic, and let
\[
	\tilde{q}(x,\xi) = (Mx)\cdot\xi,
\]
with $M$ in Jordan normal form, be as in the reduction to normal form in Proposition \ref{pNormalForm}.  Let $\delta = 2k_0/(2k_0+1)$ with $k_0$ from (\ref{ek0Def}).

For every $C_2 > 0$ there exists $C_0 > 0$ sufficiently large and $h_0 > 0$ sufficiently small that, for every $u \in \mathcal{H}_{C_0} = \mathcal{H}_{C_0, \Phi_2}$ defined in (\ref{eHighEnergyDef}), every $h \in (0,h_0]$, and every $z\in \Bbb{C}$ with $|z| \leq C_2h^{\delta}$, we have the \emph{a priori} estimate
\[
	||u||_{H_{\Phi_2}} \leq \BigO(1) h^{-\delta}||(\tilde{q}^w(x,hD_x)-z)u||_{H_{\Phi_2}}.
\]
\end{proposition}

\begin{proof} Let $\mathcal{T}:L^2(\Bbb{R}^n)\rightarrow H_{\Phi_2}(\Bbb{C}^n)$ be the unitary operator in Proposition \ref{pNormalForm}.  Then write $\Lambda_h^2$ as in (\ref{eFunctionalKarel}) and write
\[
	\tilde{\Lambda}^2_h = \mathcal{T}\Lambda^2_h\mathcal{T}^*.
\]
From (\ref{eHOLambdah}) and the discussion following, we see that
\[
	\tilde{A} = h(\tilde{\Lambda}^2_h - 1)
\]
is unitarily equivalent to the semiclassicassical harmonic oscillator in the sense of Lemma \ref{lEllipticSpectralProj}; write $\Pi_{C_1}$ for the associated spectral projection onto $[0, C_1]$ and $\{v_\alpha\}_{\alpha\in \Bbb{N}_0^n}$ for the associated orthonormal eigenbasis there.  Therefore 
\[
	\tilde{\Lambda}^2_h v_\alpha = (2|\alpha|+n+1)v_\alpha.
\]
It follows immediately from the fact that the $\{v_\alpha\}$ are orthonormal and the characterization (\ref{eHOProjDecomp}) that
\begin{eqnarray*}
	||(\tilde{\Lambda}_h^2)^{1-\delta} u||_{H_{\Phi_2}}^2 &=& \sum_{\alpha\in \Bbb{N}_0^n} (2|\alpha|+n+1)^{2(1-\delta)} |\langle u,v_\alpha\rangle|^2
	\\ &\geq& (1+C_1h^{-1})^{2(1-\delta)} ||(1-\Pi_{C_1})u||^2.
\end{eqnarray*}
Upon specifying $C_1$ later, we will choose $C_0$ as in Lemma \ref{lEllipticSpectralProj} and $h_0 > 0$ sufficiently small such that
\begin{equation}\label{eFinalEllipticity1}
	||(1-\Pi_{C_1})u|| \geq \frac{1}{2}||u||, \quad \forall u \in \mathcal{H}_{C_0}, \quad \forall h \in (0,h_0].
\end{equation}
If we apply this inequality to (\ref{eFunctionalKarel}), after conjugation by the same $\mathcal{T}$ afforded by Proposition \ref{pNormalForm}, we obtain the estimate
\begin{equation}\label{eFinalEllipticity2}
	\frac{1}{2}(1+C_1h^{-1})^{1-\delta}||u|| \leq ||(\tilde{\Lambda}^2_h)^{1-\delta}u|| \leq C_3(h^{-1}||\tilde{q}^w(x,hD_x)u||+||u||)
\end{equation}
for all $u \in \mathcal{H}_{C_0}$ and for all $h \in (0,h_0]$.

From the triangle inequality and the hypothesis $|z|\leq C_2h^{\delta}$, we have, for $u \in \mathcal{H}_{C_0}$ and $h \in (0,h_0]$,
\begin{multline}\label{eFinalEllipticity3}
	||(\tilde{q}^w(x,hD_x) - z) u|| \geq ||\tilde{q}^w(x,hD_x) u|| - |z|\, ||u||
	\\ \geq \left(\frac{h}{2C_3}(1+C_1h^{-1})^{1-\delta}-h-C_2h^\delta\right)||u||.
\end{multline}
We are then free to choose $C_1>0$ such that $C_1^{1-\delta}/(2C_3) = C_2 + 1$ and choose $C_0, h_0 > 0$ to establish (\ref{eFinalEllipticity1}) by Lemma \ref{lEllipticSpectralProj}.  We thus have from (\ref{eFinalEllipticity3}) that
\[
	||(\tilde{q}^w(x,hD_x) - z) u|| \geq (h^\delta - h) ||u||,
\]
which establishes the proposition so long as $h_0 < 1$.
\end{proof}

\subsection{Finite-dimensional analysis and proof of Theorem \ref{tExtendToPE}}

The finite-dimensional analysis proceeds identically to the analysis in Section 4 of \cite{HiSjVi2011}, as it relies only on the formula
\[
	\tilde{q}(x,\xi) = (Mx)\cdot \xi
\]
for $M$ in Jordan normal form (see Remark \ref{rqtilde} above).  We obtain the following analogous proposition to Proposition 4.2, and the remark afterwards, in \cite{HiSjVi2011}.

\begin{proposition}\label{pFiniteDim}
Let $q(x,\xi):\Bbb{R}^{2n}\rightarrow\Bbb{C}$ be a quadratic form which is partially elliptic and has trivial singular space in the sense of (\ref{eRealSemidef}) and (\ref{eTrivialS}), and as in Proposition \ref{pNormalForm}, consider the operator $\tilde{q}^w(x,hD_x)$ acting on $H_{\Phi_2}(\Bbb{C}^n)$ which is unitarily equivalent to $q^w(x,hD_x)$ acting on $L^2(\Bbb{R}^n)$.  Fix any $C_0, C_1 > 0$ and $L \geq 1$.  We use $\Phi_0 = |x|^2/2$ as a reference weight.  We assume that $\mathcal{K}_{C_0} = \mathcal{K}_{C_0,\Phi_0}$, defined in (\ref{eLowEnergyDef}), is equipped with the $H_{\Phi_0}$ norm.  First, assume that $z\in\Bbb{C}$ satisfies
\begin{equation}\label{eFiniteDimHyp1}
	\opnm{dist}(z,\opnm{Spec} \tilde{q}^w(x,hD_x)) \geq \frac{h^L}{C_1}.
\end{equation}
Then there exist implicit constants and $h_0 > 0$ sufficiently small where, for all $h \in (0,h_0]$, we have the following operator norm estimate
\[
	||(z-\tilde{q}^w(x,hD_x))^{-1}||_{\mathcal{L}(\mathcal{K}_{C_0,\Phi_0})} = \BigO(1)\jvexp \left(\BigO(1)h^{-1}\log\frac{1}{h}\right).
\]

If we assume instead that $F(q)$ is diagonalizable with no assumptions on $z \in \Bbb{C}$, we have
\[
	||(z-\tilde{q}^w(x,hD_x))^{-1}||_{\mathcal{L}(\mathcal{K}_{C_0,\Phi_0})} \leq \left(\opnm{dist}(z,\opnm{Spec} \tilde{q}^w(x,hD_x))\right)^{-1}
\]
\end{proposition}

\begin{remark} In \cite{HiSjVi2011}, exponential upper bounds (with no logarithmic loss) could also be obtained from the assumption
\[
	\opnm{dist}(z,\opnm{Spec} \tilde{q}^w(x,hD_x)) \geq \frac{1}{C_1}.
\]
However, from the rescaling argument (\ref{eScaleHyp}) leading to the proof of Theorem \ref{tExtendToPE}, we would need to require
\[
	\opnm{dist}(z,\opnm{Spec} \tilde{q}^w(x,hD_x)) \geq \frac{h^{-2k_0}}{C_1}.
\]
In the non-elliptic case, where $k_0 > 0$, this combined with the assumption that $|z|\leq C_2$ makes the estimate vacuous for $h$ sufficiently small.
\end{remark}

We may then follow the proof of Theorem 1.1 in \cite{HiSjVi2011} to establish Theorem \ref{tExtendToPE} in the present work.  We begin by assuming our spectral parameter $\zeta \in \Bbb{C}$ obeys $|\zeta|\leq C_2h^{\delta}$, and we choose $C_0 > 0$ sufficiently large and $h_0 > 0$ sufficiently small that the conclusion of Proposition \ref{pPartiallyEllipticEstimate} holds. We henceforth consider only $h \in (0,h_0]$.

Recalling the definitions (\ref{eLowEnergyDef}) and (\ref{eHighEnergyDef}), write $\tau$ for the projection uniquely characterized by the decomposition of $H_{\Phi_2}$ into $\mathcal{K}_{C_0}\oplus \mathcal{H}_{C_0}$:
\[
	H_{\Phi_2} \ni u \mapsto (\tau u, (1-\tau)u) \in \mathcal{K}_{C_0}\oplus \mathcal{H}_{C_0}.
\]
As in Proposition 3.3 of \cite{HiSjVi2011}, we have the operator norm bounds
\begin{equation}\label{etauBounded}
	||\tau||_{\mathcal{L}(H_{\Phi_2})}, ||1-\tau||_{\mathcal{L}(H_{\Phi_2})} \leq Ce^{C/h}
\end{equation}
with constants depending on $\Phi_2$ and $C_0$.  (We will allow $C$ to change from line to line.)  Using analogous considerations which establish (4.13) in \cite{HiSjVi2011}, we have with $\Phi_0(x) = |x|^2/2$ the estimates
\[
	||u||_{H_{\Phi_2}}\leq Ce^{C/h}||u||_{H_{\Phi_0}}, \quad ||u||_{H_{\Phi_0}}\leq Ce^{C/h}||u||_{H_{\Phi_2}}, \quad \forall u \in \mathcal{K}_{C_0},
\]
where constants depend again only on $\Phi_2$ and $C_0$.

We introduce notation for the restriced resolvent norm, with norms in $H_{\Phi_0}$, which is the quantity bounded in Proposition \ref{pFiniteDim}:
\[
	B(\zeta;h) = ||(\tilde{q}^w(x,hD_x)-\zeta)^{-1}||_{\mathcal{L}(\mathcal{K}_{C_0,\Phi_0})}.
\]
Because $\tau$ and $\tilde{q}^w(x,hD_x)$ commute and because we have the exponential bounds in (\ref{etauBounded}), we see that
\[
	||\tau u||_{H_{\Phi_2}} \leq Ce^{C/h} B(\zeta;h) ||(\tilde{q}^w(x,hD_x)-\zeta)u||_{H_{\Phi_2}}.
\]
Using also the estimate from Proposition \ref{pPartiallyEllipticEstimate}, we have
\begin{eqnarray*}
	||(1-\tau)u||_{H_{\Phi_2}} \leq Ch^{-\delta}e^{C/h}||(\tilde{q}^w(x,hD_x)-\zeta)u||_{H_{\Phi_2}}.
\end{eqnarray*}
We combine these estimates to obtain
\[
	||u||_{H_{\Phi_2}}\leq Ce^{C/h}(B(\zeta;h)+h^{-\delta})||(\tilde{q}^w(x,hD_x)-\zeta)u||_{H_{\Phi_2}}.
\]

We then rescale using a change of variables like (\ref{eScaleOuthOp}).  Assuming $|z|\leq C_2$ and writing $\zeta = h^{\delta}z$, the above estimate provides that
\begin{eqnarray*}
	||u||_{H_{\Phi_2}} &\leq& Ce^{C/h}(B(\zeta;h)+h^{-\delta})||(\tilde{q}^w(x,hD_x)-h^{\delta}z)u||_{H_{\Phi_2}}
	\\ & = & h^{\delta}Ce^{C/h}(B(\zeta;h)+h^{-\delta})||(h^{-\delta}\tilde{q}^w(x,hD_x)-z)u||_{H_{\Phi_2}}.
\end{eqnarray*}
Writing $\tilde{h} = h^{1-\delta}$, a change of variables provides
\begin{equation}\label{eExtendToPENextToLast}
	||u||_{H_{\Phi_2}(\Bbb{C}^n;\tilde{h})} \leq h^{\delta}Ce^{C/h}(B(\zeta;h)+h^{-\delta})||(\tilde{q}^w(x,\tilde{h}D_x)-z)u||_{H_{\Phi_2}(\Bbb{C}^n;\tilde{h})}
\end{equation}

All that remains is to apply Proposition \ref{pFiniteDim} to bound $B(\zeta;h)$.  We will make $\tilde{h} = h^{1-\delta}$ our semiclassical parameter, and a change of variables shows that
\[
	\opnm{dist}(\zeta,\opnm{Spec} \tilde{q}^w(x,hD_x)) = h^{\delta}\opnm{dist}(z, \opnm{Spec} \tilde{q}^w(x,\tilde{h}D_x)).
\]
Our strategy will be to show that the hypotheses in Theorem \ref{tExtendToPE}, in terms of spectral parameter $z$ and semiclassical parameter $\tilde{h}$, imply conditions on $\zeta$ and $h$ which are sufficient to establish the hypotheses in Proposition \ref{pFiniteDim}.  Recalling that $\delta = 2k_0/(1+2k_0)$, we obtain the general rule that
\begin{multline}\label{eScaleHyp}
	\opnm{dist}(\zeta, \opnm{Spec} \tilde{q}^w(x,hD_x)) \geq f(h) 
	\\ \iff \opnm{dist}(z, \opnm{Spec} \tilde{q}^w(x,\tilde{h}D_x)) \geq \tilde{h}^{-2k_0} f(\tilde{h}^{1+2k_0}).
\end{multline}

  For $C_1, \tilde{L} > 0$, under the hypothesis
\begin{equation}\label{eUnscaledHyp1}
	\opnm{dist}(\zeta, \opnm{Spec}\tilde{q}^w(x,hD_x)) \geq \frac{h^{\tilde{L}}}{C_1},
\end{equation}
we have for $h$ sufficiently small
\begin{equation}\label{eExtendToPEh1}
	h^{\delta}Ce^{C/h}(B(\zeta;h)+h^{-\delta}) \leq \BigO(1)\jvexp \left(\BigO(1)h^{-1}\log\frac{1}{h}\right).
\end{equation}
From (\ref{eScaleHyp}), we see that (\ref{eUnscaledHyp1}) is equivalent to having
\[
	\opnm{dist}(z, \opnm{Spec} \tilde{q}^w(x,\tilde{h}D_x)) \geq \frac{\tilde{h}^{(1+ 2k_0)\tilde{L}-2k_0}}{C_1}.
\]
This is implied by the assumption
\[
	\opnm{dist}(z, \opnm{Spec} \tilde{q}^w(x,\tilde{h}D_x)) \geq \frac{\tilde{h}^L}{C_1},
\]
taken from Theorem \ref{tExtendToPE}, if we set $\tilde{L} = (L+2k_0)/(1+2k_0)$.

On the other hand, if we assume that $F(q)$ is diagonalizable and that
\begin{equation}\label{eUnscaledHyp2}
	\opnm{dist}(\zeta, \opnm{Spec}\tilde{q}^w(x,hD_x)) \geq \frac{1}{C_1}\jvexp(-\tilde{L}h^{-1}),
\end{equation}
we have for $h$ sufficiently small
\begin{equation}\label{eExtendToPEh2}
	h^{\delta}Ce^{C/h}(B(\zeta; h)+h^{-\delta}) \leq \BigO(1)\jvexp \left(\BigO(1)h^{-1}\right).
\end{equation}
Again applying (\ref{eScaleHyp}), we have that (\ref{eUnscaledHyp2}) is equivalent to
\[
	\opnm{dist}(z, \opnm{Spec} \tilde{q}^w(x,\tilde{h}D_x)) \geq \frac{\tilde{h}^{-2k_0}}{C_1}\jvexp(-\tilde{L}\tilde{h}^{-1-2k_0}).
\]
This follows for $h$ sufficiently small under the assumption
\[
	\opnm{dist}(z, \opnm{Spec} \tilde{q}^w(x,\tilde{h}D_x)) \geq \frac{1}{C_1}\jvexp(-L\tilde{h}^{-1-2k_0})
\]
so long as we choose any fixed $\tilde{L} > L$ to absorb the polynomially growing $\tilde{h}^{-2k_0}$.

Replacing $h$ with $\tilde{h}^{1+2k_0}$ in the upper bounds (\ref{eExtendToPEh1}) and (\ref{eExtendToPEh2}) completes the proof of Theorem \ref{tExtendToPE}.

\subsection{Computation of resolvent norms on energy shells}\label{ssNumerics}

In Examples \ref{exKFP1} and \ref{exJordan1} we have representatives of Theorem \ref{tOrthogonalProj} where one has an orthogonal decomposition into $\tilde{q}^w(x,hD_x)$-invariant subspaces.  As mentioned in the proof of Proposition \ref{pNormalForm}, it is easy to see that linear changes of variables leave
\[
	E_m := \opnm{span}\{x^\alpha\::\:|\alpha| = m\}
\]
invariant.  Similarly, if $\tilde{q}(x,\xi) = (M_1x)\cdot\xi$, then the $E_m$ are $\tilde{q}^w(x,hD_x)$-invariant, even if $M_1$ is not in Jordan normal form.  We will use this decomposition to illustrate the differences in resolvent norm behavior for partially elliptic and fully elliptic operators, when restricted to high-energy functions.

For this reason, we will analyze both examples acting on the space
\[
	H_{\Phi_1}(\Bbb{C}^n;h),\quad \Phi_1(x) = \frac{1}{4}|x|^2.
\]
Following Section 4 of \cite{HiSjVi2011}, it is easy to check that
\begin{equation}\label{evarphiDef}
	\varphi_\alpha(x):= (2\pi h)^{-n/2}\left(\alpha! (2h)^{|\alpha|}\right)^{-1/2}x^\alpha
\end{equation}
forms an orthonormal basis for $H_{\Phi_1}$.  In order to have simpler formulas, we use the notation that $\varphi_\alpha = 0$ if $\alpha_j < 0$ for some $j$.

It is then elementary that, if $\tilde{q}(x,\xi)$ is such that the orthogonal subspaces $E_m\subseteq H_{\Phi_1}$ are $\tilde{q}^w(x,hD_x)$-invariant, 
\[
	||(\tilde{q}^w(x,hD_x) -z)^{-1}||_{\mathcal{L}(H_{\Phi_1})} = \sup_{m \in\Bbb{N}_0} ||(q^w(x,hD_x)|_{E_m}-z)^{-1}||_{\mathcal{L}(E_m,||\cdot||_{H_{\Phi_1}})}.
\]
Furthermore, writing $Q_m$ as the matrix representation of $\tilde{q}^w(x,hD_x)|_{E_m}$ with respect to an $H_{\Phi_1}$-orthonormal basis, we have the restricted resolvent norm as the inverse of the smallest singular value:
\begin{multline*}
	||(q^w(x,hD_x)|_{E_m}-z)^{-1}||_{\mathcal{L}(E_m,||\cdot||_{H_{\Phi_1}})} 
	\\ = \left(\inf\{\sqrt{\lambda}\::\: \lambda \in \opnm{Spec}(Q_m-z)^*(Q_m-z)\}\right)^{-1}.
\end{multline*}

In Example \ref{exJordan1} we have the operator
\[
	\tilde{Q}_3(h) = 2i(x_1 hD_{x_1}+x_2 hD_{x_2}) + h + x_2 hD_{x_1}
\]
acting on $H_{\Phi_1}$.  A simple computation yields a bidiagonal matrix, represented with $\alpha = (\alpha_1,\alpha_2)$:
\[
	\tilde{Q}_3(h)\varphi_\alpha = h(2|\alpha|+1)\varphi_\alpha + \frac{h}{i}(\alpha_1(\alpha_2+1))^{1/2}\varphi_{(\alpha_1-1,\alpha_2+1)}.
\]

Turning to Example \ref{exKFP1}, from (\ref{eKFPPhi1}) we have that (\ref{eKFPOperator}) acting on $L^2(\Bbb{R}^2)$ is unitarily equivalent to
\[
	\tilde{Q}_{2,0}(h) = \frac{i}{\sqrt{2}}\left(v\cdot hD_x - x \cdot hD_v\right) + iv\cdot hD_v + \frac{h}{2}
\]
acting on $H_{\Phi_1}$. Therefore, with $\varphi_\alpha$ defined in (\ref{evarphiDef}) and $\alpha = (\alpha_1,\alpha_2)$, we have a tridiagonal matrix:
\begin{multline*}
	\tilde{Q}_{2,0}(h)\varphi_\alpha = \frac{h}{2}(2\alpha_2 + 1)\varphi_\alpha \\ + \frac{h}{\sqrt{2}}\left((\alpha_1(\alpha_2+1))^{1/2}\varphi_{(\alpha_1-1,\alpha_2+1)}-((\alpha_1+1)\alpha_2)^{1/2}\varphi_{(\alpha_1+1,\alpha_2-1)}\right).
\end{multline*}

In Figure \ref{f1}, we compare $||(\tilde{Q}(h)|_{E_m}-z)^{-1}||$ for $\tilde{Q}_{2,0}(h)$ at $z = 0.5+0.1i$ and $\tilde{Q}_3(h)$ at $z=2+0.4i$, taken as a function of the energy, $mh$.  We remark that the difference in $|z|$ is explained by the difference in the real part of the eigenvalues of the two operators.  For reference we include the harmonic oscillator with symbol $2ix\cdot\xi$ acting on $H_{\Phi_1}$, with resolvent norm taken at $z = 2+0.4i$.

It seems apparent that both operators have maximum resolvent norm at a bounded value of energy $mh$.  This is proven for non-normal elliptic operators in \cite{HiSjVi2011}, but the author knows no proof in the partially elliptic case.  There is also an apparent marked difference between the behavior after the energy passes this peak: in the elliptic case, the non-selfadjoint behavior is seen in an exponentially large peak, but for high energies the resolvent norm of the non-selfadjoint operator behaves quite similarly to that of the harmonic oscillator.  For the partially elliptic Kramers-Fokker-Planck operator, it seems that exponentially large resolvent norms persist for energies significantly larger than the energy at which the maximum resolvent norm occurs.

\begin{figure}
\caption{Log-log graph comparing resolvent norm vs.\ energy $mh$ for (left) a partially elliptic operator and (right) a fully elliptic operator for $h = 0.1, 0.05, 0.025$. Dotted lines are for the resolvent norm of a harmonic oscillator.}
\label{f1}
\centering
\includegraphics[scale=0.20]{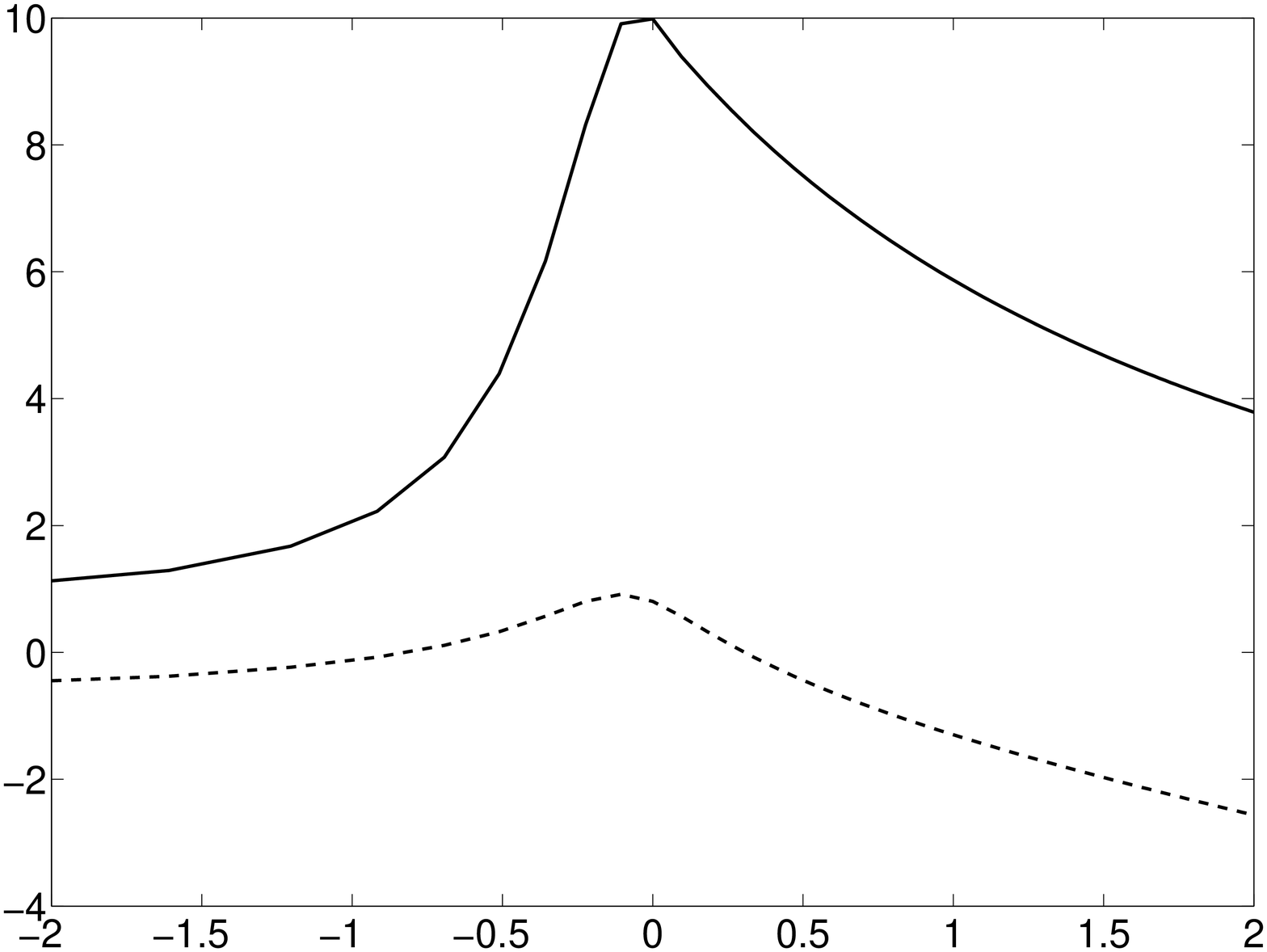}
\includegraphics[scale=0.20]{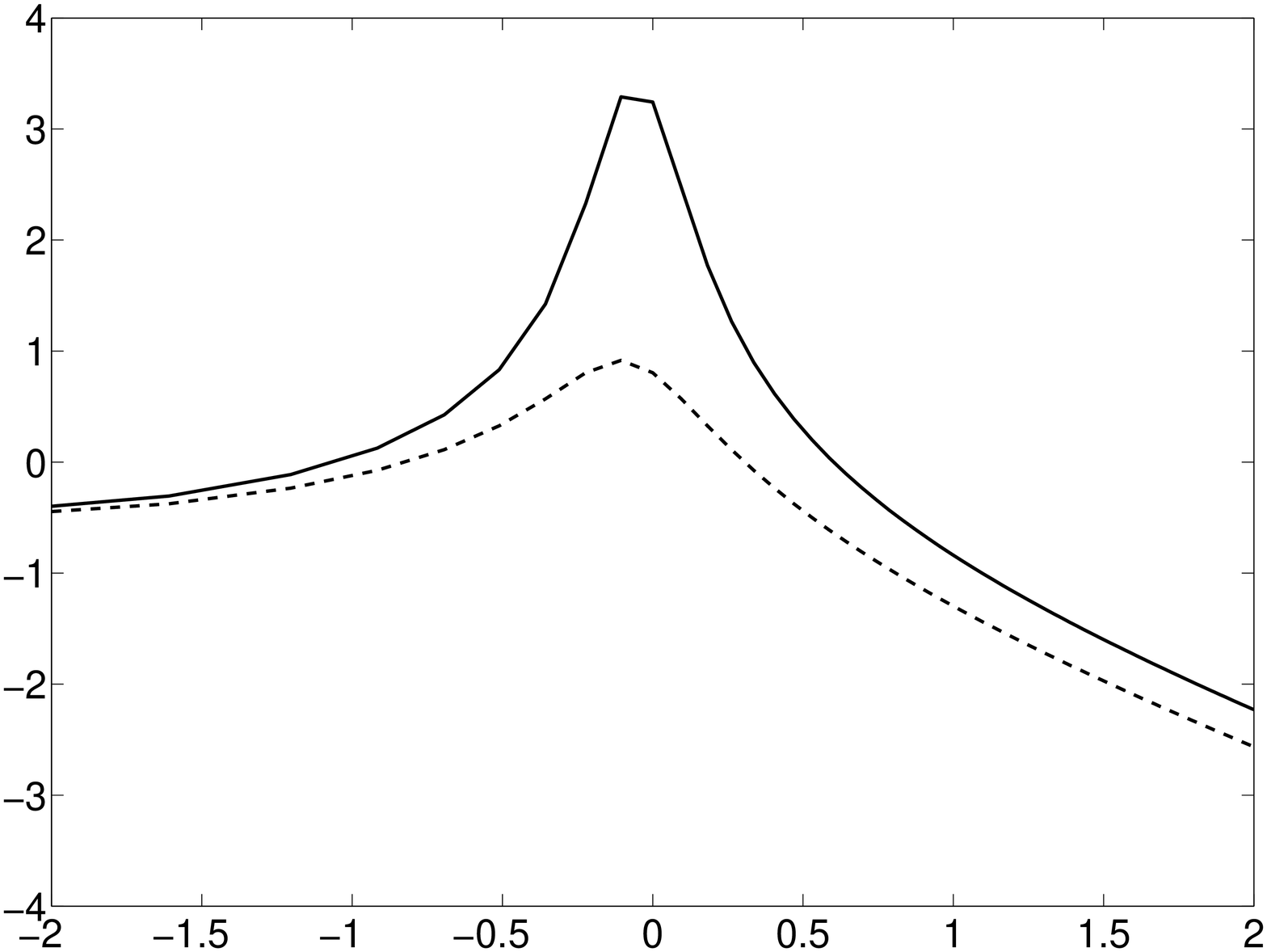}\\
\includegraphics[scale=0.20]{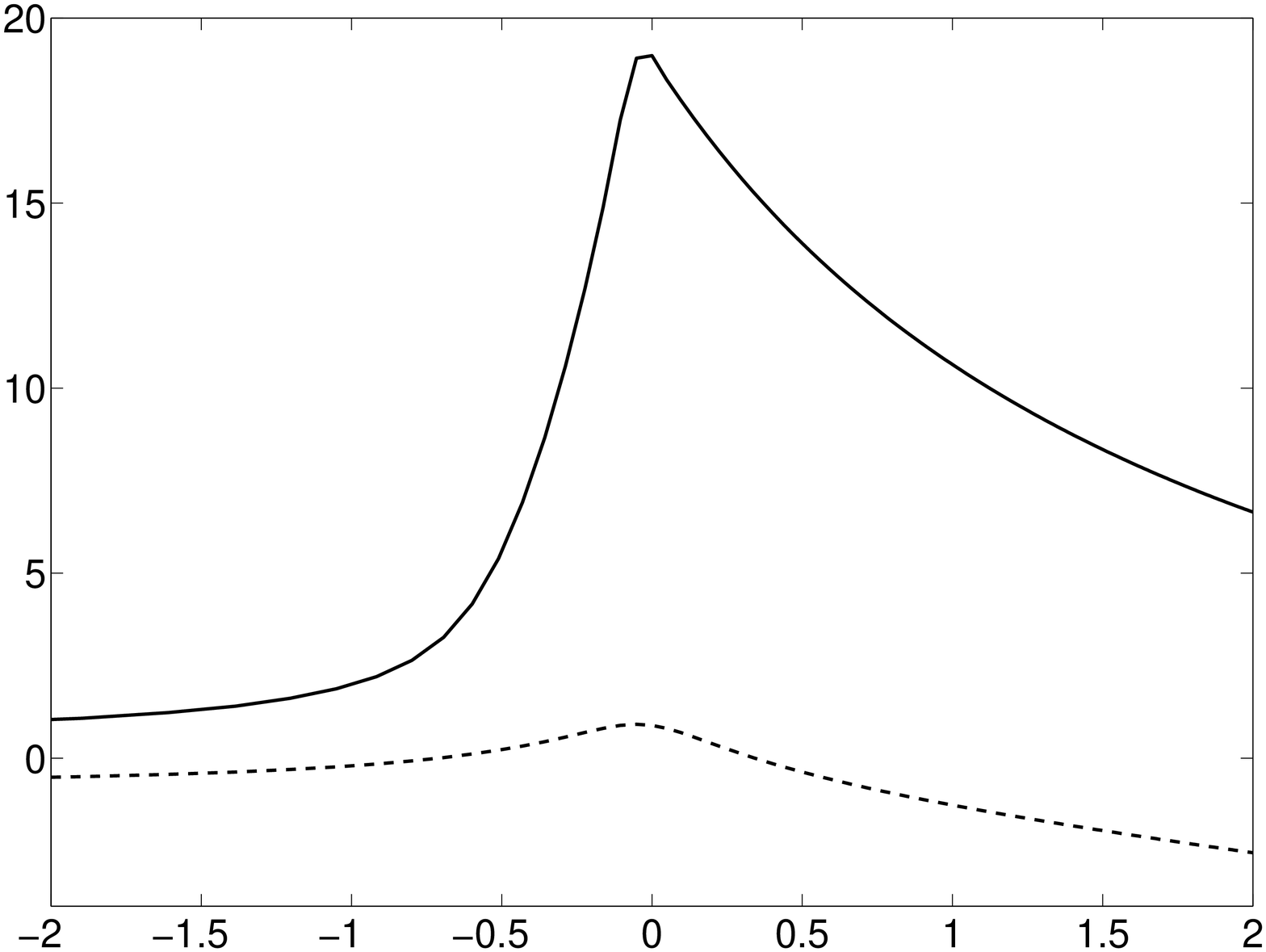}
\includegraphics[scale=0.20]{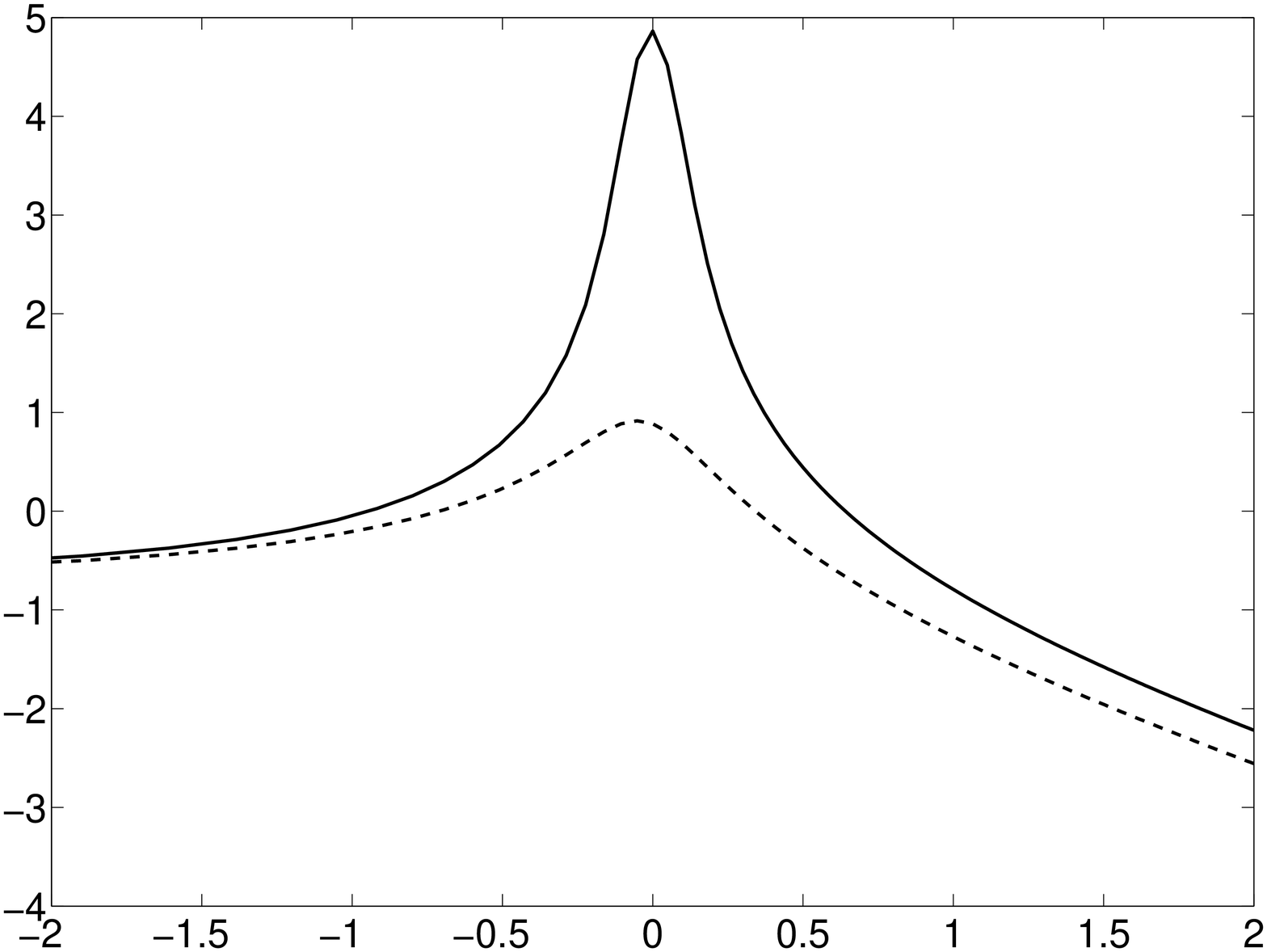}\\
\includegraphics[scale=0.20]{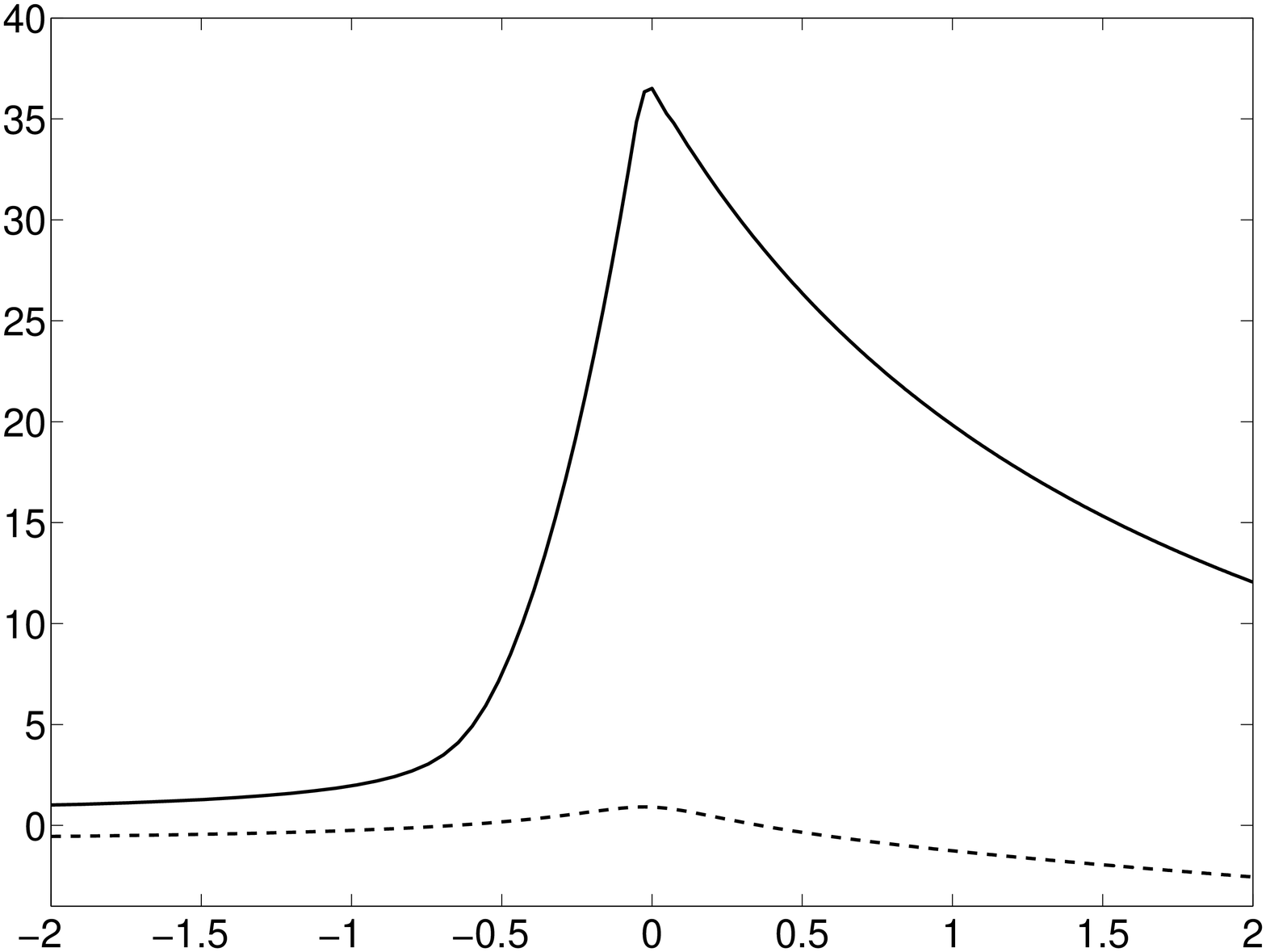}
\includegraphics[scale=0.20]{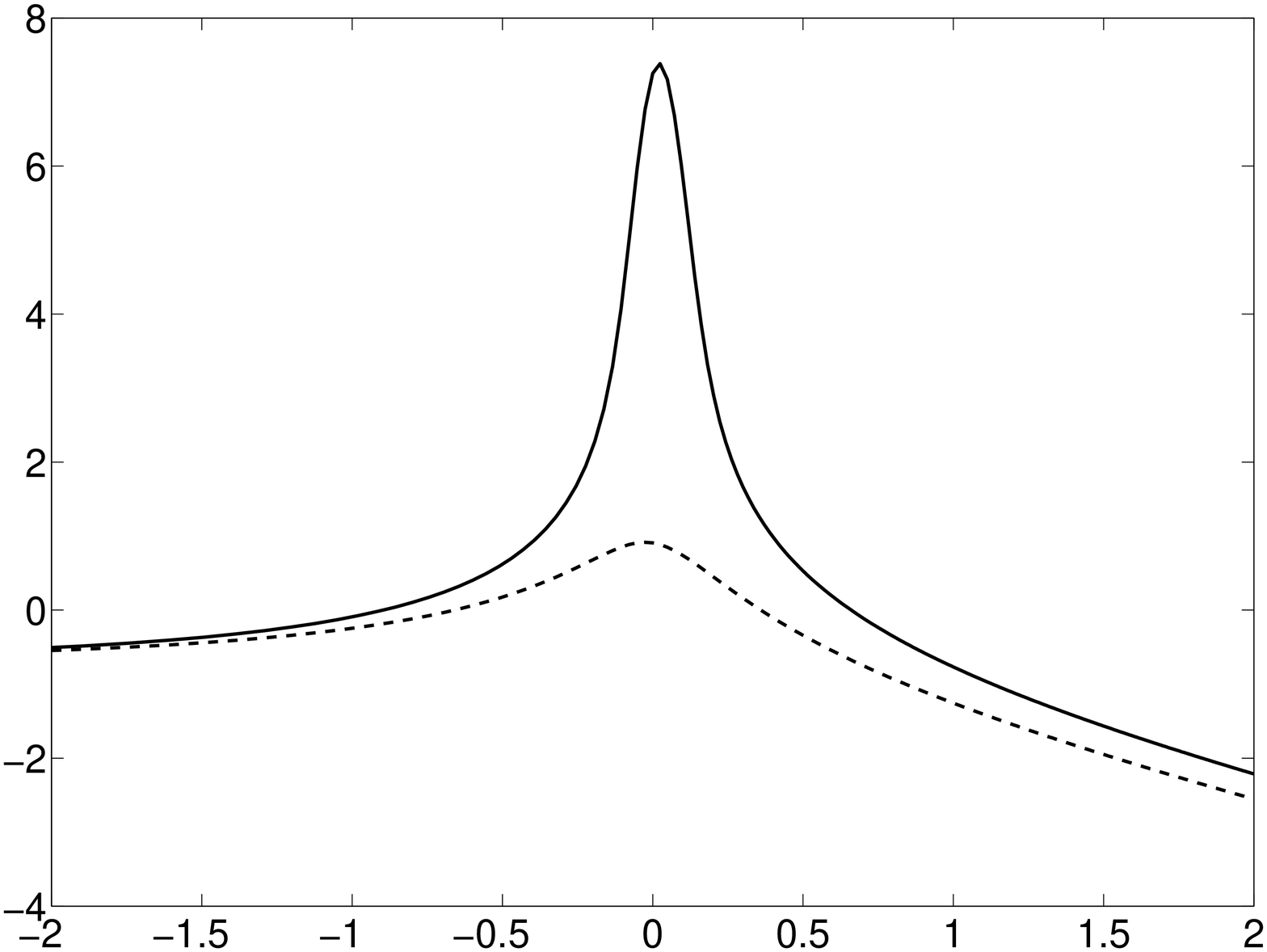}
\end{figure}

\section{Characterization of spectral projections}\label{sCharProj}

In Proposition \ref{pNormalForm}, we reduce our partially elliptic operators to a normal form whose action on polynomials is easy to describe.  In fact, the monomials form a basis of generalized eigenvectors for $\tilde{q}^w(x,hD_x)$ on the corresponding weighted space $H_{\Phi_2}$.

The goal of this section is to prove Theorem \ref{tCharacterizeProj}, which confirms that the spectral projections defined in (\ref{eSpectralProjection}) respect the natural Taylor series decomposition of $H_{\Phi_2}$ into monomials.  Following the methods of \cite{HiSjVi2011}, we then have an elementary exponential upper bound for the spectral projections.  We compare with the known rate of exponential growth for Example \ref{exDavies1}, discovered in \cite{DaKu2004}, to see that this elementary upper bound is in fact sharp for dimension 1.

\subsection{Characterization of the spectral projections}

We begin the proof of Theorem \ref{tCharacterizeProj} with two elementary lemmas.  First, a spectral projection respects a decomposition, even when not orthogonal, of a Hilbert space $\mathcal{H}$ into invariant subspaces of an operator $A$.

\begin{lemma}\label{lRieszProjDecomp}
Let $\mathcal{H}_1, \mathcal{H}_2$ be closed subspaces of a Hilbert space $\mathcal{H}$, complementary in the sense that $\mathcal{H} = \mathcal{H}_1 \oplus \mathcal{H}_2$.  We then have a unique decomposition for $v \in \mathcal{H}$ into $v = v_1 + v_2$ with $v_j \in \mathcal{H}_j$, and we will write $\pi_j v = v_j$.  Let $A$ be a closed densely defined operator on $\mathcal{H}$ such that $\mathcal{H}_1, \mathcal{H}_2$ are $A$-invariant subspaces.  Let $A_j = A|_{\mathcal{H}_j}$.  As in the assumptions for (\ref{eSpectralProjection}), let $\Omega \subseteq \Bbb{C}$ be such that $\opnm{Spec}A = \Omega \cup \Omega_2$, where $\Omega$ is contained in a bounded Cauchy domain $\Delta$ with $\overline{\Delta}\cap \Omega_2 = \emptyset$. Then
\begin{equation}\label{eRieszProjDecomp}
	P_{\Omega, A} = P_{\Omega, A_1}\pi_1 + P_{\Omega, A_2}\pi_2.
\end{equation}
\end{lemma}

\begin{proof}
By the Closed Graph Theorem and the fact that the $\mathcal{H}_j$ are closed, clearly the $\pi_j$ are continuous; we record the standard facts that $\pi_j^2 = \pi_j$, that $\pi_1 + \pi_2 = 1$, and that $\pi_1 \pi_2 = \pi_2 \pi_1 = 0$.  The statement that the $\mathcal{H}_j$ are $A$-invariant is equivalent to the statement that $[A, \pi_1] = [A, \pi_2] = 0$.  We see that, when $\zeta \notin \opnm{Spec}(A)$, the resolvent $(\zeta-A_j)^{-1}$ exists as a bounded linear operator on $\mathcal{H}_j$ by writing
\[
	(\zeta-A)^{-1}v = w \iff v = (\zeta-A)w
\]
and observing that $x = y$ if and only if both $\pi_1 x = \pi_1 y$ and $\pi_2 x = \pi_2 y$.  The same facts imply that
\begin{equation}\label{eResolventProjDecomp}
	(\zeta - A)^{-1} = (\zeta - A_1)^{-1}\pi_1 + (\zeta - A_2)^{-1}\pi_2,
\end{equation}
from which (\ref{eRieszProjDecomp}) immediately follows after integrating.
\end{proof}

\begin{remark} One may view (\ref{eRieszProjDecomp}) as part of a natural extension of the relation
\[
	A = A_1\pi_1 + A_2\pi_2
\]
to the functional calculus:
\[
	f(A) = f(A_1)\pi_1 + f(A_2)\pi_2.
\]
On the level of Taylor series, this follows immediately from noting that $A_j\pi_j = A\pi_j$ and using that $[\pi_j, A] = 0$ and $\pi_1\pi_2 = \pi_2\pi_1 = 0$.  For non-analytic $f$ satisfying appropriate hypotheses, this follows from (\ref{eResolventProjDecomp}) and the Helffer-Sj{\"o}strand formula (\cite{DiSjBook}, Theorem 8.1 and remarks thereafter).
\end{remark}

The finite-dimensional analysis in \cite{HiSjVi2011} summarized in Section \ref{ssNormalForm} rests on the fact that $\tilde{q}^w(x,hD_x)|_{\mathcal{K}_{C_0}}$, with $\mathcal{K}_{C_0}$ defined in (\ref{eLowEnergyDef}), resembles a matrix in Jordan normal form.  The following lemma establishes that such matrices have the usual spectral projections, given by picking out those basis vectors associated with the Jordan block corresponding to the eigenvalue.

\begin{lemma}\label{lJordanLike}
Let $\opnm{dim}\mathcal{H} = N < \infty$, let $A$ be an operator on $\mathcal{H}$, and let $\{\Bff{v}_j\}_{j=1}^N$ be a basis for $\mathcal{H}$ with respect to which the matrix $(a_{jk})_{j,k=1}^N$ representing $A$ is Jordan-like in the sense that $(a_{jk})$ is upper-triangular and $a_{jk} \neq 0 \implies a_{jj} = a_{kk}$.  Then, following the notation of (\ref{eSpectralProjection}),
\[
	P_{\{\zeta\}, A}(\alpha_1 \Bff{v}_1 + \dots + \alpha_N\Bff{v}_N) = \sum_{j\::\: a_{jj} = \zeta} \alpha_j \Bff{v}_j.
\]
\end{lemma}

\begin{proof} The assumption that $a_{jk} \neq 0 \implies a_{jj} = a_{kk}$ implies that we may write $\mathcal{H}$ as a finite direct sum of the $A$-invariant subspaces
\[
	E_z = \opnm{span}\{\Bff{v}_j\::\: a_{jj} = z\}.
\]
In view of Lemma \ref{lRieszProjDecomp}, it suffices to show that
\[
	P_{\{\zeta\}, A|_{E_z}} = \left\{ \begin{array}{ll} 1_{E_z}, & \zeta = z \\ 0, & \zeta \neq z \end{array}\right..
\]
Restricted to each such subspace, we have
\[
	A|_{E_z} = z+N
\]
with $N$ nilpotent because $A$ is upper-triangular.  We expand the integrand in (\ref{eSpectralProjection}) in a finite Neumann series:
\[
	(\lambda - A|_{E_z})^{-1} = (\lambda - z)^{-1}+\sum_{j=1}^{\BigO(1)} (\lambda-z)^{-j-1}N^j.
\]
The lemma then follows from the elementary fact that, for $j \in \Bbb{N}$,
\[
	\int_{|\zeta - z| = \eps} (\lambda - z)^{-j}\,dz
\]
is zero for $\eps$ sufficiently small unless $j = 1$ and $\zeta = z$, in which case one obtains $2\pi i$.
\end{proof}

We combine these lemmas to form the characterization of the spectral projections for $\tilde{q}^w(x,hD_x)$ acting as an operator on $H_{\Phi_2}(\Bbb{C}^n;h)$ as in Proposition \ref{pNormalForm}, thus proving Theorem \ref{tCharacterizeProj}.  Because the spectral projections are continuous and polynomials are dense in $H_{\Phi_2}$ (following Remark \ref{rPolynomialsDense}), it suffices to compute a spectral projection for $u$ a polynomial.  Fixing such a $u$ and an $h>0$, we may then choose $C_0$ sufficiently large that $u \in \mathcal{K}_{C_0}$ defined in (\ref{eLowEnergyDef}).  We recall from \cite{HiSjVi2011} that $\mathcal{K}_{C_0}$ is $\tilde{q}^w(x,hD_x)$-invariant, and it then follows from Lemma \ref{lRieszProjDecomp} that spectral projections for $\tilde{q}^w(x,hD_x)$, acting on $u$, are identical to the spectral projections for $\tilde{q}^w(x,hD_x)|_{\mathcal{K}_{C_0}}$.

We recall the characterization of $\tilde{q}^w(x,hD_x)$ reviewed in Remark \ref{rqtilde}.  In particular, with respect to the basis $\{x^\alpha\}_{|\alpha|\leq C_0h^{-1}}$ for $\mathcal{K}_{C_0}$, we have that
\[
	\tilde{q}^w(x,hD_x)x^\alpha = \mu_\alpha x^\alpha + \sum_{j=1}^{n-1}\frac{\gamma_j h\alpha_j}{i} x^{\alpha'_j},
\]
writing $\alpha_j' \in \Bbb{N}_0^n$ for the multi-index obtained from $\alpha$ by decreasing $\alpha_j$ by 1 and increasing $\alpha_{j+1}$ by 1.

Because $\gamma_j = 0$ if $\lambda_{j+1}\neq \lambda_j$, we have that $q^w(x,hD_x)x^\alpha$ is a linear combination only of certain $x^{\alpha'_j}$ for which $\mu_{\alpha'_j} = \mu_\alpha$. In the language of Lemma \ref{lJordanLike}, this means that the matrix representation of $\tilde{q}^w(x,hD_x)$ acting on $\mathcal{K}_{C_0}$ with respect to the basis $\{x^\alpha\}_{|\alpha|\leq C_0h^{-1}}$ is Jordan-like so long as it is upper-triangular: that is to say, writing
\[
	\tilde{q}^w(x,hD_x)x^\alpha  = \sum_{\beta} a_{\beta\alpha} x^\beta,
\]
we have that $a_{\alpha \alpha} = \mu_\alpha$ for all $\alpha$ and that $a_{\beta \alpha} = 0$ if $\mu_\beta \neq \mu_\alpha$.  To ensure that the matrix of $\tilde{q}^w(x,hD_x)|_{\mathcal{K}_{C_0}}$ is upper triangular with respect to this basis, it suffices to equip the $\alpha$ with $|\alpha| \leq C_0h^{-1}$ with an ordering $\prec$ in such a way that $\alpha_j' \prec \alpha$.  Since the degree of a monomial is preserved by $\tilde{q}^w(x,hD_x)$, we only need to order the $\alpha$ with $|\alpha|$ fixed, and we do so by saying that
\begin{equation}\label{eMultiindexOrdering}
	\alpha \prec \beta \iff \sum_{j=1}^n j\alpha_j > \sum_{j=1}^n j\beta_j.
\end{equation}
Note that this simply reverses the ordering used in the proof of Lemma 4.1 in \cite{HiSjVi2011}.

Because each polynomial certainly has a unique expression as
\[
	u(x) = \sum_{|\alpha|\leq N} (\alpha!)^{-1}(\partial^\alpha u(0))x^\alpha,
\]
Theorem \ref{tCharacterizeProj} is therefore proven for any polynomial by applying Lemma \ref{lJordanLike} to $\tilde{q}^w(x,hD_x)$ acting on $\mathcal{K}_{C_0}$ with $C_0$ sufficiently large.  This extends to all of $H_{\Phi_2}$ via density of polynomials and the fact that
\[
	\partial^\alpha u(0) = 0
\]
is an $H_{\Phi_2}$-closed condition, as it is $L^2_{\rm loc}(\Bbb{C}^n)$ closed for entire functions and the $H_{\Phi_2}$ topology is finer.

\subsection{An elementary exponential upper bound}

We present the following minor extension of Proposition 3.3 in \cite{HiSjVi2011}.

\begin{lemma}\label{lMainBounds}
Let $\Phi:\Bbb{C}^n \rightarrow \Bbb{R}$ be strictly convex, real-valued, and quadratic.   Let $C_\ell, C_u > 0$ be such that 
\[
	C_\ell |x|^2 \leq \Phi(x) \leq C_u|x|^2.
\]
For $S \subseteq \Bbb{N}_0^n$ a finite collection of multi-indices, define $M = \max_{\alpha \in S}|\alpha|$.  Write
\[
	\tau_S u(x) = \sum_{\alpha \in S} (\alpha!)^{-1}(\partial^\alpha u(0)) x^\alpha.
\]
Then
\[
	||\tau_S||_{\mathcal{L}(H_\Phi)} \leq \left(\frac{C_u}{C_\ell}\right)^{\frac{n+M}{2}}.
\]
\end{lemma}

\begin{proof}
Note that polynomials are dense in $H_\Phi$; see Remark \ref{rPolynomialsDense}.  It therefore suffices to consider $u \in H_\Phi$ a polynomial, since continuity of $\tau_S$ follows from the closed graph theorem.  Replacing $1/C_1$ with $C_\ell$ turns (3.12) of \cite{HiSjVi2011} into
\[
	||\tau_S u||^2_{H_\Phi} \leq \sum_{\alpha \in S} |a_\alpha|^2 \left(\frac{h}{2 C_\ell}\right)^{n+|\alpha|} \pi^n \alpha!,
\]
and similarly (3.14) becomes
\[
	||u||^2_{H_\Phi} \geq \sum_{\alpha \in S} |a_\alpha|^2 \left(\frac{h}{2 C_u}\right)^{n+|\alpha|} \pi^n \alpha!.
\]
When examining the ratio $||\tau_S u||^2_{H_\Phi}/||u||^2_{H_\Phi}$, we may factor out $(C_u/C_\ell)^{n+M}$.  This replaces the $C_\ell^{-n-|\alpha|}$ in the numerator with $C_\ell^{M-|\alpha|}$ and likewise for $C_u^{-n-|\alpha|}$ in the denominator.  Since $M-|\alpha| \geq 0$ and $C_u \geq C_\ell$, the conclusion follows after taking a square root.
\end{proof}

This can be immediately applied to the spectral projections for $\tilde{q}^w(x,hD_x)$: because
\[
	\jvRe \mu_\alpha \geq (2|\alpha|+n)h\min\jvIm\lambda_j \geq \frac{h}{C}|\alpha|,
\]
those $\alpha$ appearing in the expression (\ref{eCharacterizeProj}) for $\Pi_{\zeta_0}$ have modulus bounded by $Ch^{-1}$, with constant depending only on $\tilde{q}$ and an upper bound for $|\zeta_0|$.  This establishes the following corollary.

\begin{corollary}\label{cElementaryExponential}
With $\Pi_{z}$ the spectral projection for $\tilde{q}^w(x,hD_x)$ and $\{z\}$ as in Theorem \ref{tCharacterizeProj},
\[
	||\Pi_z||_{\mathcal{L}(H_\Phi)} \leq Ce^{C/h}.
\]
(The constants may be made uniform for $z \in \Omega$ so long as $\Omega \subset \Bbb{C}$ is bounded.)
\end{corollary}

\begin{example}\label{exDavies2}
We recall from Example \ref{exDavies1} that $q_1(x,\xi)$ is diagonalized on $H_{\Phi_1}(\Bbb{C}^1)$ for 
\[
	\Phi_1(x) = \frac{1}{4}(|x|^2 - \jvRe (x,C_+x)), \quad C_+ = \frac{1}{2}(-1+e^{-4i\theta}).
\]
Notice that
\begin{equation}\label{eAbsCplusDavies}
	|C_+| = \sqrt{\frac{1}{4}(2-2\cos 4\theta)} = |\sin 2\theta|.
\end{equation}
The principal advantage of the estimate in Lemma \ref{lMainBounds} is that $C_\ell$ and $C_u$ are easy to compute as the least and greatest eigenvalues of the $2n \times 2n$ real symmetric matrix $\frac{1}{2}\nabla^2_{\jvRe x,\jvIm x}\Phi$.  In this one-dimensional example, and with $S = \{N\}$, we easily compute that
\[
	\nabla^2_{\jvRe x,\jvIm x}\Phi_1 = \frac{1}{2}\left(\begin{array}{cc} 1-\jvRe C_+ & \jvIm C_+ \\ \jvIm C_+ & 1+\jvRe C_+\end{array}\right)
\]
with eigenvalues $(1\pm |C_+|)/2$ and therefore, using (\ref{eAbsCplusDavies}),
\begin{equation}\label{eElementaryDaviesBound}
	||\tau_{\{N\}}||_{\mathcal{L}(H_{\Phi_1})} \leq \left(\frac{1+|\sin 2\theta|}{1-|\sin 2\theta|}\right)^{\frac{n+N}{2}}.
\end{equation}
In view of Theorem \ref{tCharacterizeProj} and the unitary equivalence in Proposition \ref{pNormalForm} described in Example \ref{exDavies1}, the same upper bound holds for the operator norm of the spectral projection
\[
	\Pi_N := P_{q_1^w(x,D_x),\{2N+1\}}:L^2(\Bbb{R}^1)\rightarrow L^2(\Bbb{R}^1)
\]
using the definition in (\ref{eSpectralProjection}).

The other notable property of the estimate for this operator is that this rate of exponential growth is optimal, which is proven in \cite{DaKu2004} and reaffirmed in Corollary \ref{cDim1Asymptotic}.  Because the formulas are not obviously identical, we indicate the necessary computations to show equality.

In Theorem 3 of \cite{DaKu2004} (with the necessary changes of notation), Davies and Kuijlaars proved the exponential rate of growth
\[
	\lim_{N\rightarrow\infty} N^{-1}\log ||\Pi_N|| = 2\jvRe[f(r(\theta)e^{i\theta})]
\]
with
\[
	f(z) = \log(z+(z^2-1)^{1/2})-z(z^2-1)^{1/2}
\]
and
\[
	r(\theta) = (2\cos(2\theta))^{-1/2}.
\]
Here, we are taking $0 < \theta < \pi/4$.

First, with $\zeta = r(\theta)e^{i\theta}$,
\begin{multline*}
	\zeta^2 - 1 = (2\cos 2\theta)^{-1}(\cos 2\theta + i\sin 2\theta) - 1  \\ = (2\cos 2\theta)^{-1}(-\cos 2\theta + i\sin 2\theta) = -\frac{e^{-2i\theta}}{2\cos 2\theta}.
\end{multline*}
Recalling the definition of $r(\theta)$, we therefore have $(\zeta^2 - 1)^{1/2} = \pm ir(\theta)e^{-i\theta}$.  We expand
\[
	f(r(\theta)e^{i\theta}) = \log(r(\theta)e^{i\theta} \pm ir(\theta)e^{-i\theta}) \mp ir(\theta)^2.
\]
We are taking a real part of $f(r(\theta)e^{i\theta})$, so the second term is irrelevant.  Since $2\jvRe \log u = \log |u|^2$, we obtain, again using the definition of $r(\theta)$, that
\[
	2\jvRe[f(r(\theta)e^{i\theta}))] = -\log (2\cos 2\theta) + \log\left(|e^{i\theta} \pm ie^{-i\theta}|^2\right) = \log\left(\frac{1\pm \sin 2\theta}{\cos 2\theta}\right).
\]
The fact that projections should grow exponentially quickly indicates that the positive branch is the right choice.  One may then easily check that
\[
	\left(\frac{1 + \sin 2\theta}{\cos 2\theta}\right)^2 = \frac{1+\sin 2\theta}{1-\sin 2\theta}
\]
for $0 < \theta < \pi/4$, indicating that (\ref{eElementaryDaviesBound}) is optimal there.  That (\ref{eElementaryDaviesBound}) is optimal for $\theta = 0$ is obvious from the fact that 1 is the norm of any spectral projection for a normal operator, and the extension to $-\pi/4<\theta<0$ is easily seen by using the Fourier transform to interchange $x$ and $\xi$ in the definition (\ref{eDaviesSymbol}) of $q_1(x,\xi)$.

\end{example}

\begin{remark}\label{rNotSharp}
It is clear that neither Corollary \ref{cElementaryExponential} nor Corollary \ref{cUpperBound} should be sharp in general in higher dimensions. Consider some $0 < r_1 \leq r_2 \leq \dots \leq r_n$ and
\[
	\Phi(x) = \frac{1}{4}\sum_{j=1}^n r_j|x_j|^2,
\]
and note that it is easy to check using polar coordinates that $\{x^\alpha\}_{\alpha \in \Bbb{N}_0^n}$ form an orthogonal basis for $H_\Phi$.  Therefore, normalizing the $x^\alpha$, we see that $\tilde{q}^w(x,hD_x)$ is orthonormally diagonalizable and therefore normal whenever $\tilde{q}(x,\xi) = (Mx)\cdot \xi$ for $M$ diagonal.  (Alternately, $\Phi$ may be viewed as coming from the standard weight $\Phi_0(x) = |x|^2/4$ after an anisotropic change in semiclassical parameter similar to (\ref{eScaleOuthOp}).)

However, both Corollary \ref{cElementaryExponential} and Corollary \ref{cUpperBound} provide an exponentially growing upper bound like $(r_n/r_1)^{|\alpha|}$.  Since normal operators have orthogonal spectral projections, we see that these estimates are not necessarily sharp.
\end{remark}

\section{Dual bases for projections onto monomials}\label{sDualBases}

Following the methods of \cite{DaKu2004}, a continuous projection $\Pi$ onto a one-dimensional subspace $\opnm{span}\psi$ of a Hilbert space $\mathcal{H}$ may be analyzed by treating the resulting coefficient of $\psi$ as a continuous linear functional on $\mathcal{H}$.  Therefore there exists some unique $\psi^\dagger$ where
\[
	\Pi u = \langle u, \psi^\dagger\rangle \psi, \quad \forall u \in \mathcal{H}.
\]
The operator norm of $\Pi$ may then be computed from $\psi$ and $\psi^\dagger$:
\[
	||\Pi||_{\mathcal{L}(\mathcal{H})} = ||\psi||\,||\psi^\dagger||.
\]

When the ranges of spectral projections described in Theorem \ref{tCharacterizeProj} have higher dimension, formulas for $||\Pi||$ like the one above are unavailable.  Example \ref{exJordan1} demonstrates that one may have a weighted space where the natural projections onto monomials grow exponentially quickly, yet spectral projections associated with a quadratic operator acting on that weighted space may be orthogonal.  We therefore focus on the simple one-dimensional case
\[
	\Pi_\alpha u(x) = (\alpha!)^{-1}(\partial^\alpha u(0))x^\alpha : H_\Phi\rightarrow H_\Phi.
\]

We are interested in those $\Phi$ which can be obtained from Proposition \ref{pNormalForm}, but we begin by observing that it is equivalent to assume that $\Phi$ is real-valued, quadratic, and strictly convex.  Next, we write a elementary formula describing the family $\{\phi^\dagger_\alpha\}$, which is defined by the relations
\[
	\langle x^\alpha, \phi^\dagger_\beta\rangle_{H_\Phi} = \delta_{\alpha\beta},
\]
in terms of adjoints on $H_\Phi$, and we derive formulas for these adjoints.  After relating these $\{\phi^\dagger_\alpha\}$ to the eigenfunctions of $q^w(x,hD_x)^*$, we are in a position to prove Theorem \ref{tOrthogonalProj}.  Finally, we describe how $\{\phi^\dagger_\alpha\}$ are unitarily equivalent to $\{x^\alpha\}$ after a change of weight, thus proving Theorem \ref{tBoundProj}.

\subsection{Reversibility of the reduction to normal form}\label{ssReversibility}

In Section \ref{ssNormalForm}, we have a reduction to an operator on $H_{\Phi_2}(\Bbb{C}^n;h)$ for $\Phi_2$ a strictly convex real-valued quadratic weight described in (\ref{ePhi2Matrix}).  Here, we show that every strictly convex real-valued quadratic weight $\Phi$ may be written as $\Phi_2$ for some appropriate $G$ and $C_+$.

We begin with some general facts about real-valued quadratic forms and a useful decomposition thereof.  We may write
\begin{equation}\label{ePhiDecompose}
	\Phi(x) = \frac{1}{2}(x,\Phi_{xx}'' x) + (x, \Phi''_{x\bar{x}}\bar{x}) + \frac{1}{2}(\bar{x},\Phi''_{\bar{x}\bar{x}}\bar{x}).
\end{equation}
That $\Phi$ is real-valued is equivalent to the two statements
\[
	\overline{\Phi''_{xx}} = \Phi''_{\bar{x}\bar{x}}, \quad (\Phi''_{x\bar{x}})^* = \Phi''_{x\bar{x}}.
\]

It is natural (see for instance \cite{CaGrHiSj2012}, Appendix A) to decompose $\Phi$ into Hermitian and pluriharmonic parts, obtaining $\Phi = \Phi_{\textnormal{herm}}+\Phi_{\textnormal{plh}}$ with
\begin{equation}\label{ePhihermDef}
	\Phi_{\textnormal{herm}}(x) = (x,\Phi''_{x\bar{x}}\bar{x}) = \frac{1}{2}(\Phi(x) + \Phi(ix))
\end{equation}
and
\begin{equation}\label{ePhiplhDef}
	\Phi_{\textnormal{plh}}(x) = \jvRe(x,\Phi_{xx}'' x) = \frac{1}{2}(\Phi(x) - \Phi(ix)).
\end{equation}
We also note that strict convexity of $\Phi$ implies strict plurisubharmonicity of $\Phi$, which is in the quadratic case case equivalent to the strict convexity of $\Phi_{\textnormal{herm}}$.  This in turn is equivalent to having $\Phi''_{\bar{x}x}$ be a positive definite Hermitian matrix.

We now prove reversibility of Proposition \ref{pNormalForm}, summarized in the following proposition.

\begin{proposition}\label{pReversibility}
Let $\Phi:\Bbb{C}^n\rightarrow\Bbb{R}$ be a strictly convex real-valued quadratic weight.  Then there exists some $G$ and $C_+$, with $C_+$ given by (\ref{eCplusDef}) for some symmetric $A_+$ with $\jvIm A_+ > 0$, for which $\Phi = \Phi_2$ as in (\ref{ePhi2Matrix}).
\end{proposition}

\begin{proof}
We begin by reducing to the case $G_+ = 1$.  Having seen that $\Phi''_{\bar{x}x}$ is a positive definite Hermitian matrix, we may write
\[
	\Phi''_{\bar{x}x} = \frac{1}{4}G^*G
\]
for some invertible $G \in \Bbb{C}^{n\times n}$.  We then see that $\Phi(G^{-1}x)$ is strictly convex and that
\begin{equation}\label{ePhiB}
	\Phi(G^{-1}x) = \frac{1}{4}(|x|^2-\jvRe(x,Bx))
\end{equation}
for $B = -4 G^{-t}\Phi''_{xx}G^{-1}$, which is clearly symmetric. We naturally propose that
\[
	C_+ = B
\]
and set out to find some symmetric $A_+$ with $\jvIm A_+ > 0$ which yields $C_+$ via (\ref{eCplusDef}).

By Takagi's factorization (Corollary 4.4.4 of \cite{HoJoBook}) there exists some unitary $U \in \Bbb{C}^{n\times n}$ and a diagonal matrix $\Sigma$ whose entries, the singular values of $C_+$, are nonnegative real numbers for which
\[
	C_+ = U\Sigma U^t.
\]
But then, since $U$ is unitary,
\[
	\Phi_1(U^tx) = \frac{1}{4}(|x|^2 - \jvRe(x,\Sigma x)).
\]
(See also Lemma 5.1 of \cite{Ho1997}, which provides the same reduction to normal form.)  It is then immediate that strict convexity of $\Phi_1$ is equivalent to requiring that the diagonal entries of $\Sigma$ must lie in $[0,1)$.  Since the diagonal entries of $\Sigma$ are the square roots of the eigenvalues of $C_+^*C_+$, this in turn is equivalent to requiring that the selfadjoint operator $1-C_+^*C_+$ is positive definite.

Since the spectral radius of a matrix is at most its largest singular value and since the singular values of $C_+$ lie in $[0,1)$, we see that $-1 \notin \opnm{Spec} C_+$.  We may therefore solve (\ref{eCplusDef}) for $A_+$ and propose that
\begin{equation}\label{eAplusFromCplus}
	A_+ = i(1+C_+)^{-1}(1-C_+).
\end{equation}
Symmetry of $A_+$ follows from symmetry of $C_+$.

We then recall that $1-C_+^*C_+$ and $\jvIm A_+$ are related through (\ref{e1MinusAbsCplus}), where it was seen that positive definiteness of $1-C_+^*C_+$ is equivalent to positive definiteness of $\jvIm A_+$.  Since we have established positive definiteness of $1-C_+^*C_+$ through strict convexity of $\Phi_1$, this completes the proof of the proposition.
\end{proof}

\begin{remark}\label{rGCFromConvexity} We also record that the invertible matrix $G$ and the symmetric matrix $C_+$ in (\ref{ePhi2Matrix}) may be written in terms of derivatives of the weight $\Phi$.  We see that we may make the (non-unique) choice
\[
	G = 2(\Phi''_{\bar{x}x})^{1/2},
\]
using the usual square root of a positive definite Hermitian matrix.  We then note that we used
\[
	C_+ = -4G^{-t}\Phi_{xx}''G^{-1}.
\]
We furthermore note the geometric characterizations that $G$ may be regarded as determining the Hermitian part of $\Phi$, defined in (\ref{ePhihermDef}), while $C_+$ determines the pluriharmonic part, defined in (\ref{ePhiplhDef}), of the reduced weight function $\Phi(G^{-1}x)$.
\end{remark}

\subsection{Characterization of the dual basis to $\{x^\alpha\}$ in $H_\Phi$}\label{ssDualBasis}

The main goal of this section is to obtain a formula for $\{\phi_\alpha^\dagger\}_{\alpha\in\Bbb{N}^n_0}$ for which
\[
	\langle x^\alpha, \phi_\beta^\dagger\rangle_{H_\Phi} = \delta_{\alpha\beta}.
\]
Throughout this section, $\Phi:\Bbb{C}^n\rightarrow\Bbb{R}$ is real-valued, strictly convex, and quadratic.  We note that these relations determine the $\phi_\beta^\dagger$ uniquely in $H_\Phi$ since the functional $u \mapsto \langle u, \phi^\dagger_\beta\rangle_{H_\Phi}$ is then prescribed on the polynomials, a dense subset of $H_\Phi$ (see Remark \ref{rPolynomialsDense}).  We show that $\phi_\beta^\dagger \in H_\Phi$ below as a consequence of (\ref{ephidaggerDef}).

We will show that, when
\begin{equation}\label{ephi0daggerDef}
	\phi_0^\dagger(x) = C_0\jvexp\left(\frac{1}{h}(x,\Phi_{xx}'' x)\right)
\end{equation}
we have
\begin{equation}\label{exStarphi0dag}
	x^* \phi_0^\dagger = 0,
\end{equation}
where $x = (x_1,\dots,x_n)$ is a multiplication operator, the adjoint represents $x^* = (x_1^*,\dots,x_n^*)$ acting on $(H_\Phi)^n$, and equality here is naturally in $(H_\Phi)^n$.  So long as the $h$-dependent constant $C_0$ is chosen such that 
\begin{equation}\label{eDualProperty0}
	\langle 1, \phi_0^\dagger\rangle_{H_\Phi} = \int_{\Bbb{C}^n} \overline{C_0\jvexp\left(\frac{1}{h}(x,\Phi_{xx}'' x)\right)}e^{-2\Phi(x)/h}\,dL(x) = 1,
\end{equation}
it will then be immediate from (\ref{exStarphi0dag}) that
\[
	\langle x^\alpha, \phi_0^\dagger\rangle_{H_\Phi} = \langle 1, (x^*)^\alpha \phi_0^\dagger\rangle_{H_\Phi} = \delta_{\alpha,0},
\]

Passing to adjoints, we may then easily see that
\[
	\langle x^\alpha, (\beta!)^{-1}(\partial_x^*)^\beta \phi_0^\dagger\rangle_{H_\Phi} = \delta_{\alpha\beta}.
\]
Indeed, when $\alpha = \beta$ we only need to use that $\partial_x^\beta x^\beta = \beta !\cdot 1$.  When there exists some $j\in \{1,\dots,n\}$ with $\alpha_j < \beta_j$, we have that $\partial_x^\beta x^\alpha = 0$, and if $\alpha_j \geq \beta_j$ for all $j\in\{1,\dots,n\}$ yet $\alpha \neq \beta$, then 
\[
	\langle x^\alpha, (\beta!)^{-1}(\partial_x^*)^\beta\phi_0^\dagger\rangle_{H_\Phi} = C_{\alpha\beta}\langle x^{\alpha - \beta},\phi^\dagger_0\rangle_{H_\Phi} = 0
\]
by (\ref{exStarphi0dag}).

Once we establish (\ref{exStarphi0dag}), we will therefore have that
\begin{equation}\label{ephidaggerTrivialDef}
	\phi_\beta^\dagger = (\beta !)^{-1}(\partial_x^*)^\beta \phi_0^\dagger.
\end{equation}
What remains is to express $x^*$ and $\partial_x^*$ in useful ways.  We will show that both operators may be represented as $Ax + B(hD_x)$ for matrices $A,B$ depending on the weight $\Phi$.

We begin with a bookkeeping rule for adjoints of $n$-tuples of operators.  Let $\mathcal{H}$ be a Hilbert space, and let $\Bff{A} = (A_1,\dots, A_n)$ and $\Bff{B} = (B_1,\dots,B_n)$ be operators on $\mathcal{H}^n$ subject to the rule
\[
	\Bff{B} = M \Bff{A}, \quad M = (m_{jk})_{j,k=1}^n.
\]
Writing $\Bff{B}^* = (B_1^*,\dots,B_n^*)$, we then have the rule
\begin{equation}\label{eVectorAdjoints}
	\Bff{B}^* = \overline{M} \Bff{A}^*,
\end{equation}
obtained by taking complex conjugates of the entries in $M$ but without transposition.

We return to the specific context of $H_\Phi$ with $\Phi$ real-valued, quadratic, and strictly convex, recalling the decomposition (\ref{ePhiDecompose}) and related facts at the beginning of Section \ref{ssReversibility}.

We now compute the operators $x^*$ and $\partial_x^*$, with adjoints hereafter in this section taken in $H_\Phi$.   We remark that formulas (\ref{exStar}) and (\ref{ehDxStar}) below may be obtained from the unique expression of the symbols $\bar{x}$ and $\bar{\xi}$ as holomorphic linear functions of $x$ and $\xi$ when $(x,\xi) \in \Lambda_{\Phi}$ defined in (\ref{eLambdaPhiDef}).  For completeness, we include the usual proof via integration by parts.

Note that holomorphic and antiholomorphic derivatives $\partial_x$ and $\partial_{\bar{x}}$ are formally antisymmetric on the unweighted space $L^2(\Bbb{C}^n,dL(x))$. There is a dense subset of $u,v\in H_\Phi$ with sufficient decay at infinity to justify the following computation:
\begin{eqnarray*} 
	\langle hD_x u, v\rangle_{H_\Phi} &=& \int_{\Bbb{C}^n} (hD_x u(x)) \overline{v(x)}e^{-2\Phi(x)/h}\,dL(x)
	\\ &=& -\int_{\Bbb{C}^n} u(x) hD_x\left[\overline{v(x)}e^{-2\Phi(x)/h}\right]\,dL(x)
	\\ &=& -\int_{\Bbb{C}^n} u(x) \overline{v(x)}\frac{h}{i}\left(-\frac{2}{h}\right)\partial_x\Phi(x) e^{-2\Phi(x)/h}\,dL(x)
	\\ &=& \int_{\Bbb{C}^n} u(x) \overline{v(x)}\frac{2}{i}\left(\Phi''_{xx} x + \Phi''_{x\bar{x}}\bar{x}\right)e^{-2\Phi(x)/h}\,dL(x).
\end{eqnarray*}
For instance, one may take $|u(x)| \leq \BigO_{N,h}(1)\langle x\rangle^{-N}e^{\Phi(x)/h}$ and $v(x)$ obeying similar estimates, as in page 8 of \cite{Sj1996}.  This gives
\[
	\langle (\frac{i}{2}hD_x-\Phi''_{xx}x) u, v\rangle_{H_\Phi} = \langle u, \overline{\Phi''_{x\bar{x}}}x v\rangle_{H_\Phi},
\]
from which
\[
	\left(\overline{\Phi''_{x\bar{x}}}x\right)^* = -\Phi''_{xx}x + \frac{i}{2}hD_x,
\]
which leads immediately to 
\begin{equation}\label{exStar}
	x^* = (\Phi''_{x\bar{x}})^{-1}\left(-\Phi''_{xx}x + \frac{i}{2}hD_x\right)
\end{equation}
in view of (\ref{eVectorAdjoints}).  

We recall from Section \ref{ssReversibility} that $\Phi''_{x\bar{x}}$ is a positive definite Hermitian matrix.  We therefore have that, if
\[
	\phi_0^\dagger(x) = C_0\jvexp\left(\frac{1}{h}(x,Fx)\right), \quad F \in \Bbb{C}^{n\times n}, \quad F^t = F,
\]
then
\[
	x^*\phi_0^\dagger(x) = (\Phi''_{x\bar{x}})^{-1}\left(-\Phi''_{xx}x + \frac{i}{2}\cdot\frac{h}{i}\left(\frac{1}{h}\right)2Fx\right)\phi_0^\dagger(x).
\]
Thus $x^*\phi_0^\dagger(x) = 0$ exactly when
\[
	F = \Phi''_{xx}.
\]

We recall definitions (\ref{ePhihermDef}) and (\ref{ePhiplhDef}). That $\phi_0^\dagger\in H_\Phi$ follows from the convenient fact that
\[
	|C_0^{-1}\phi_0^\dagger(x)|^2 = \jvexp\left(\frac{2}{h}\Phi_{\textnormal{plh}}(x)\right),
\]
and so
\begin{equation}\label{ephidaggerAndWeight}
	|\phi_0^\dagger(x)|^2 e^{-2\Phi(x)/h} = |C_0|^2 \jvexp\left(-\frac{2}{h}\Phi_{\textnormal{herm}}(x)\right),
\end{equation}
where we have seen that $\Phi_{\textnormal{herm}}(x)$ is strictly convex.  That $\langle 1, \phi^\dagger_0\rangle_{H_\Phi}$ is finite follows from the Cauchy-Schwarz inequality, since both $||1||_{H_\Phi}^2$ and $||C_0^{-1}\phi^\dagger_0||_{H_\Phi}^2$ may now be seen as integrals of $1$ against $\jvexp(-2\tilde{\Phi}(x)/h)$ for some $\tilde{\Phi}$ which is a real-valued strictly convex quadratic form.  We may therefore choose an $h$-dependent constant $C_0$ such that (\ref{eDualProperty0}) holds, and we will compute this constant below.

Moving on to $(hD_x)^*$, we obtain from (\ref{exStar}) that
\[
	\overline{\Phi_{x\bar{x}}''}x= -\overline{\Phi''_{xx}}x^* - \frac{i}{2}(hD_x)^*,
\]
and so
\begin{multline}\label{ehDxStar}
	(hD_x)^* = -\frac{2}{i}(\overline{\Phi_{xx}''} x^* + \overline{\Phi''_{x\bar{x}}}x)
	\\ = \frac{2}{i}\left(-\overline{\Phi_{x\bar{x}}''}+\overline{\Phi_{xx}''}(\Phi_{x\bar{x}}'')^{-1}\Phi''_{xx}\right)x - \overline{\Phi''_{xx}}(\Phi''_{x\bar{x}})^{-1}hD_x.
\end{multline}
The former formulation is particularly convenient because $x^*\phi_0^\dagger = 0$.  It follows from (\ref{exStar}) and the Leibniz rule that
\[
	x^*(fg) = f\cdot (x^*g) + \frac{i}{2}((\Phi''_{x\bar{x}})^{-1}hD_x f)\cdot g.
\]
Therefore
\begin{eqnarray*}
	(\partial_x)^*(f\phi_0^\dagger) &=& -\frac{i}{h}(hD_x)^*(f\phi^\dagger_0) 
	\\ &=& \frac{2}{h}\overline{\Phi''_{xx}}x^*(f\phi_0^\dagger) + \frac{2}{h}\overline{\Phi''_{x\bar{x}}}x(f\phi_0^\dagger)
	\\ &=& \frac{2}{h}\overline{\Phi''_{xx}}\left(fx^*\phi_0^\dagger + \frac{i}{2}((\Phi''_{x\bar{x}})^{-1}hD_x f)\phi_0^\dagger\right)+\frac{2}{h}\overline{\Phi_{x\bar{x}}''}xf\phi_0^\dagger
	\\ &=& \left[\left(\frac{2}{h}\overline{\Phi_{x\bar{x}}''} x + \overline{\Phi_{xx}''}(\Phi_{x\bar{x}}'')^{-1}\partial_x\right)f\right]\phi_0^\dagger.
\end{eqnarray*}
We conclude that
\begin{equation}\label{ephidaggerDef}
	\phi^\dagger_\alpha = (\alpha!)^{-1}\left[\left(\frac{2}{h}\overline{\Phi_{x\bar{x}}''} x + \overline{\Phi_{xx}''}(\Phi_{x\bar{x}}'')^{-1}\partial_x\right)^\alpha 1\right]\phi_0^\dagger.
\end{equation}
Since this formula makes it apparent that each $\phi_\alpha^\dagger$ is a polynomial times $\phi_0^\dagger$, we deduce $\phi_\alpha^\dagger \in H_\Phi$ immediately as a consequence of (\ref{ephidaggerAndWeight}).

To compute $C_0$ in (\ref{ephi0daggerDef}), we write
\begin{equation}\label{eC0Computation}
	\langle 1,\phi^\dagger_0\rangle_{H_\Phi} = \int_{\Bbb{C}^n} \overline{C}_0 e^{\frac{i}{h}Q(x)}\,dL(x)
\end{equation}
with
\begin{equation}\label{eQ1phi0}
	Q(x) = -i(\bar{x},\overline{\Phi''_{xx}}\bar{x}) + 2i\Phi(x) =  i(x,\Phi''_{xx}x)+2i(x,\Phi''_{x\bar{x}}\bar{x}).
\end{equation}
We may apply Lemma 13.2 in \cite{ZwBook} to see that, when 
\[
	Q:\Bbb{C}^n \sim \Bbb{R}^n_{\jvRe x} \times \Bbb{R}^n_{\jvIm x} \rightarrow \Bbb{C}
\]
is a quadratic form with $\jvIm Q$ strictly convex, we have
\[
	\int_{\Bbb{C}^n} e^{\frac{i}{h}Q(x)}\,dL(x) = \left(\det\left(\frac{1}{2\pi i h}\nabla^2_{\jvRe x,\jvIm x} Q\right)\right)^{-1/2}.
\]

We then note that 
\[
	\det \left(\frac{1}{2\pi i h}\nabla^2_{\jvRe x, \jvIm x} Q\right) = (2\pi i h)^{-2n}\det \nabla^2_{\jvRe x, \jvIm x} Q,
\]
since $\nabla^2_{\jvRe x, \jvIm x} Q$ is a $2n$-by-$2n$ matrix.  Next, we use block matrices to write 
\begin{multline*}
	Q(x) = \left(\left(\begin{array}{c} \jvRe x \\ \jvIm x\end{array} \right), \frac{1}{2}\nabla^2_{\jvRe x, \jvIm x} Q \left(\begin{array}{c} \jvRe x \\ \jvIm x\end{array} \right)\right) = \left(\left(\begin{array}{c} x \\ \bar{x} \end{array} \right), \frac{1}{2}\nabla^2_{x,\bar{x}} Q \left(\begin{array}{c} x \\ \bar{x}\end{array} \right)\right)
	\\ = \left(\left(\begin{array}{cc} 1 & i \\ 1 & -i \end{array}\right)\left(\begin{array}{c} \jvRe x \\ \jvIm x\end{array} \right), \frac{1}{2}\nabla^2_{x,\bar{x}} Q \left(\begin{array}{cc} 1 & i \\ 1 & -i \end{array}\right)\left(\begin{array}{c} \jvRe x \\ \jvIm x\end{array} \right)\right).
\end{multline*}
We therefore have that
\[
	\nabla^2_{\jvRe x, \jvIm x} Q = \left(\begin{array}{cc} 1 & 1 \\ i & -i \end{array}\right) \nabla^2_{x,\bar{x}} Q \left(\begin{array}{cc} 1 & i \\ 1 & -i \end{array}\right).
\]
Recall that we are considering $2n$-by-$2n$ matrices formed of $n$-by-$n$ blocks, and so in this context
\[
	\det \left(\begin{array}{cc} 1 & 1 \\ i & -i \end{array}\right) = (-2i)^n.
\]
We may conclude that
\begin{multline}\label{eComplexGaussianIntegral}
	\int_{\Bbb{C}^n} e^{\frac{i}{h}Q(x)}\,dL(x) = \left(\det\left(\frac{1}{2\pi i h}\nabla^2_{\jvRe x,\jvIm x} Q\right)\right)^{-1/2} \\ = \left((-2i)^{2n}(2\pi i h)^{-2n} \det \nabla_{x,\bar{x}}^2 Q\right)^{-1/2} = (\pi h)^n \left(\nabla_{x,\bar{x}}^2 Q\right)^{-1/2}.
\end{multline}

With $Q(x)$ given by (\ref{eQ1phi0}), we may write in block form
\[
	\nabla^2_{x,\bar{x}} Q(x) = \left(\begin{array}{cc}\partial_x^2 Q & \partial_{\bar{x}}\partial_x Q \\ \partial_x\partial_{\bar{x}}Q & \partial_{\bar{x}}^2 Q \end{array}\right) = 2i\left(\begin{array}{cc} \Phi''_{xx} & \Phi''_{x\bar{x}} \\ \Phi''_{\bar{x} x} & 0\end{array}\right).
\]
Permuting $n$ rows to interchange rows in the block matrix representation, we therefore have
\[
	\det(\nabla_{x,\bar{x}}^2 Q) = (2i)^{2n}(-1)^n \det \left(\begin{array}{cc} \Phi''_{\bar{x}x} & 0 \\ \Phi''_{xx} & \Phi''_{x\bar{x}}\end{array}\right) = 2^{2n}\det(\Phi''_{\bar{x}x})\det(\Phi''_{x\bar{x}}).
\]
Since $\Phi''_{\bar{x}x} = (\Phi''_{x\bar{x}})^t$, the two determinants are equal.  From (\ref{eC0Computation}) and (\ref{eComplexGaussianIntegral}) we obtain
\begin{equation}\label{eC0Result}
	C_0 = \overline{\left(\frac{2}{\pi h}\right)^n \det \Phi_{x\bar{x}}''} = \left(\frac{2}{\pi h}\right)^n\det \Phi_{x\bar{x}}'',
\end{equation}
since $\Phi_{x\bar{x}}''$ is a positive definite Hermitian matrix.

We end this section with an elementary lemma that proves the well-known fact that the collection $\{\phi^\dagger_\alpha\}_{\alpha \in \Bbb{N}_0^n}$ form the eigenfunctions of the adjoint operator
\[
	(\tilde{Q}(h))^*_{H_\Phi} = (\tilde{q}^w(x,hD_x))^*_{H_\Phi}
\]
when
\[
	\tilde{q}(x,\xi) = (Mx)\cdot\xi
\]
with $M$ in Jordan normal form.

\begin{lemma}\label{lAdjointEigenvectors}
Let both the quadratic form $\tilde{q}(x,\xi) = (Mx)\cdot \xi:\Bbb{C}^{2n}\rightarrow \Bbb{C}$ and the strictly convex quadratic weight $\Phi_2:\Bbb{C}^n\rightarrow \Bbb{R}$ be obtained from applying Proposition \ref{pNormalForm} to a quadratic $q:\Bbb{R}^{2n}\rightarrow \Bbb{C}$ which is partially elliptic and has trivial singular space in the sense of (\ref{eRealSemidef}) and (\ref{eTrivialS}).

Then $\{\phi_\alpha^\dagger\}_{\alpha \in \Bbb{N}_0^n}$ given by (\ref{ephidaggerDef}) form a basis of eigenfunctions for 
\[
	\tilde{q}^w(x,hD_x)^*_{H_{\Phi_2}}:\mathcal{D}(\tilde{q}^w(x,hD_x)^*)\rightarrow H_{\Phi_2},
\]
where specifically $\phi_\alpha^\dagger$ is a generalized eigenvector of $\tilde{q}^w(x,hD_x)^*_{H_{\Phi_2}}$ with eigenvalue $\overline{\mu_\alpha}$, with $\mu_\alpha$ defined in (\ref{emuDef}).
\end{lemma}

\begin{proof}
We work in the space $H_{\Phi_2}$ with corresponding inner products and adjoints throughout.  What follows is essentially the classical proof which generates the eigenfunctions of the harmonic oscillator via creation-annihilation operators, with small modifications.

We recall from Remark \ref{rqtilde} that we may write
\[
	\tilde{q}^w(x,hD_x) = \tilde{Q}_D(h) + \tilde{Q}_N(h)
\]
for $\tilde{Q}_D(h) = \tilde{q}_D^w(x,hD_x)$ and $\tilde{Q}_N(h) = \tilde{q}_N^w(x,hD_x)$.  We begin by focusing on
\[
	\tilde{Q}_D(h) = \sum_{j=1}^n 2\lambda_j x_j hD_{x_j} + \frac{h}{i}\sum_{j=1}^n \lambda_j,
\]
and we show that
\begin{equation}\label{eQDphidagger}
	\tilde{Q}_D(h)^*\phi_\alpha^\dagger = \overline{\mu_\alpha}\phi_\alpha^\dagger.
\end{equation}

We have that
\[
	\tilde{Q}_D(h)^* = \sum_{j=1}^n 2\overline{\lambda_j}(hD_{x_j})^* x_j^* + \overline{\mu_0},
\]
and so the fact that
\[
	\tilde{Q}_D(h)^* \phi^\dagger_0 = \overline{\mu_0}\phi^\dagger_0
\]
follows immediately from the fact that $x^*\phi^\dagger_0 = 0$.

The statement (\ref{eQDphidagger}) for all $\phi_\alpha^\dagger$ follows by induction.  For $u,v\in \mathcal{D}(\tilde{Q}_D(h)) \subset H_\Phi$, using the definition of $\tilde{Q}_D(h)$ and the Leibniz rule for derivatives gives the relation
\begin{eqnarray*}
	\langle u, \tilde{Q}_D(h)^*\partial_{x_j}^* v\rangle &=& \langle \partial_{x_j}\tilde{Q}_D(h)u,v\rangle
	\\ &=& \left\langle 2h\frac{\lambda_j}{i}\partial_{x_j} u, v\right\rangle + \left\langle \tilde{Q}_D(h)\partial_{x_j}u, v\right\rangle
	\\ &=& \left\langle u, -2h\frac{\overline{\lambda_j}}{i}\partial_{x_j}^*v+\partial_{x_j}^*\tilde{Q}_D(h)v\right\rangle.
\end{eqnarray*}
We apply this to $v = \phi_\alpha^\dagger$ under the induction assumption $\tilde{Q}_D(h)^* \phi^\dagger_\alpha = \overline{\mu_\alpha}\phi^\dagger_\alpha$, and immediately see from (\ref{ephidaggerTrivialDef}) that $\tilde{Q}_D(h)^*\phi^\dagger_{\alpha+e_j} = \overline{\mu_{\alpha+e_j}}\phi^\dagger_{\alpha+e_j}$ for $e_j = (\delta_{jk})_{k=1}^n$ the standard basis vector.  Having established the base case $\alpha = 0$, we have proven (\ref{eQDphidagger}).

Now write
\[
	\tilde{Q}_N(h) = \sum_{j=1}^{n-1} \gamma_j x_{j+1}hD_{x_j},
\]
recalling that $\gamma_j \in \{0,1\}$ and that $\gamma_j = 0$ whenever $\lambda_{j+1} \neq \lambda_j$.  We now show that $\tilde{Q}_N(h)^*$ is nilpotent on each $\{\phi_\alpha^\dagger\}_{|\alpha| = m}$, though the degree of nilpotency naturally may increase with $m$.  We continue to write $e_j$ for the standard basis vector.  A similar approach shows that, for $u \in \mathcal{D}(\tilde{Q}_D(h))$,
\begin{eqnarray*}
	\langle u, \tilde{Q}_N(h)^* \phi_\alpha^\dagger \rangle &=& \sum_{j=1}^{n-1}\gamma_j (\alpha!)^{-1} \langle \partial_x^\alpha x_{j+1}hD_{x_j} u, \phi_0^\dagger\rangle
	\\ &=& \sum_{j=1}^{n-1} \gamma_j(\alpha!)^{-1}\langle (\partial_x^{\alpha-e_{j+1}}+x_{j+1}\partial_x^\alpha)hD_{x_j} u, \phi_0^\dagger\rangle
	\\ &=& \sum_{j=1}^{n-1} \gamma_j(\alpha!)^{-1}(-ih\langle u, (\partial_x^*)^{\alpha+e_j-e_{j+1}}\phi_0^\dagger\rangle)
	\\ && ~~~~~~+ \sum_{j=1}^{n-1} \gamma_j(\alpha!)^{-1}\langle \partial_x^\alpha hD_{x_j}u, x_{j+1}^*\phi_0^\dagger\rangle.
\end{eqnarray*}
The second term vanishes because $x^*\phi^\dagger_0 = 0$, and we therefore have that $\tilde{Q}_N(h)^*\phi_\alpha^\dagger$ is a linear combination of $\{\gamma_j\phi_{\alpha+e_j-e_{j+1}}^\dagger\}_{j=1}^{n-1}$.  From (\ref{emuDef}) and the fact that $\gamma_j = 0$ when $\lambda_{j+1}\neq \lambda_j$, we conclude that
\[
	\gamma_j \neq 0 \implies \mu_{\alpha+e_j-e_{j+1}} = \mu_{\alpha_j}.
\]

Having seen that $\tilde{Q}_N(h)^*$ takes $\phi_\alpha^\dagger$ to a linear combination of some $\phi^\dagger_{\beta}$ which are also eigenvectors of $\tilde{Q}_D(h)^*$ with eigenvalue $\overline{\mu_\alpha}$, we therefore have that, for $K \in \Bbb{N}$,
\[
	(\tilde{q}^w(x,hD_x)^* - \overline{\mu_\alpha})^K \phi_\alpha^\dagger = (\tilde{Q}_N(h)^*)^K\phi_\alpha^\dagger.
\]
The argument of Lemma 4.1 in \cite{HiSjVi2011}, which is essentially the observation that $\tilde{Q}_N(h)$ acts to strictly decrease the ordering (\ref{eMultiindexOrdering}), then shows that
\[
	(\tilde{q}^w(x,hD_x)^* - \overline{\mu_\alpha})^K \phi_\alpha^\dagger = 0
\]
for $K$ sufficiently large depending on $\alpha$.  

Since we now know that the $\phi_\alpha^\dagger$ are generalized eigenfunctions of $\tilde{q}^w(x,hD_x)^*$, the proof of the lemma is complete once we show that the $\{\phi_\alpha^\dagger\}_{\alpha \in \Bbb{N}^n_0}$ have dense span in $H_{\Phi_2}$.  Via (\ref{ephidaggerDef}) we see that $\phi_\alpha^\dagger/\phi_0^\dagger$ is a polynomial of degree $|\alpha|$ with leading term $((2/h)\overline{(\Phi_2)_{x\bar{x}}''}x)^\alpha$, and since $\overline{(\Phi_2)_{x\bar{x}}''}$ is invertible by strict plurisubharmonicity of $\Phi_2$, we see that
\[
	\opnm{span}\{\phi_\alpha^\dagger\}_{\alpha\in\Bbb{N}_0^n} = \Bbb{C}[x_1,\dots,x_n]\phi_0^\dagger \subseteq H_{\Phi_2}.
\]
We furthermore see from (\ref{ephidaggerAndWeight}) that, with the strictly convex weight $(\Phi_2)_{\textnormal{herm}}$ defined in (\ref{ePhihermDef}), we have that the map
\[
	H_{\Phi_2} \ni u(x)\mapsto \frac{C_0}{\phi_0^\dagger(x)}u(x) \in H_{(\Phi_2)_{\textnormal{herm}}}
\]
is unitary and takes $\opnm{span}\{\phi^\dagger_\alpha\}$ to the polynomials $\Bbb{C}[x_1,\dots, x_n]$.  Since polynomials are dense in any strictly convex quadratically weighted $H_\Phi$ (Remark \ref{rPolynomialsDense}), this shows that the $\{\phi^\dagger_\alpha\}$ have dense span in $H_{\Phi_2}$ and completes the proof of the lemma.

\end{proof}

\subsection{Proof of Theorem \ref{tOrthogonalProj}}

We are now in a position to prove Theorem \ref{tOrthogonalProj}.

Let $\tilde{q}$ and $\Phi_2$ be provided from Proposition \ref{pNormalForm} applied to some quadratic form $q$ which is partially elliptic with trivial singular space in the sense of (\ref{eRealSemidef}) and (\ref{eTrivialS}).  We will proceed by showing that each condition in Theorem \ref{tOrthogonalProj} is equivalent to showing that $(\Phi_2)_{\textnormal{plh}} = 0$, using definition (\ref{ePhiplhDef}), or equivalently that $(\Phi_2)''_{xx} = 0$.

If $T$ is an $n$-by-$n$ invertible matrix with real entries and $V, W$ are two subspaces of $\Bbb{C}^n$, it is easy to see that
\[
	V = \overline{W} \iff T(V) = \overline{T(W)}.
\]
We have such a transformation on $\Bbb{C}^{2n}$ appearing in the $\kappa$ in (\ref{ekappaDef}) in the reduction to normal form.  For this $\kappa$ we have
\[
	\kappa(\Lambda^-) = \{(y,-iy)\}_{y\in\Bbb{C}^n}, \quad \kappa(\Lambda^+) = \{(y,A_+y)\}_{y\in\Bbb{C}^n},
\]
and so we see that condition \ref{iOrthogbar} in Theorem \ref{tOrthogonalProj} is equivalent to claiming that $A_+ = i$.  By the definition (\ref{eCplusDef}) of $C_+$ in the reduction to normal form and (\ref{ePhi2Matrix}), we see furthermore that this condition is equivalent to the claim that $(\Phi_2)''_{xx} = 0$.

On the FBI transform side and using the notation of (\ref{eSpectralProjection}), write 
\[
	\tilde{\Pi}_{\mu_\alpha} = P_{\{\mu_\alpha\},\tilde{q}^w(x,hD_x)}:H_{\Phi_2}\rightarrow H_{\Phi_2}
\]
for the spectral projection for $\tilde{q}^w(x,hD_x)$ and $\{\mu_\alpha\}$, with $\mu_\alpha$ defined in (\ref{emuDef}).  In view of Theorem \ref{tCharacterizeProj}, condition \ref{iOrthogproj} is equivalent to claiming that $\tilde{\Pi}_{\mu_0}$, whose range is the set of constant functions, is an orthogonal projection.  It is sufficient to show that 
\[
	\langle(1-\tilde{\Pi}_{\mu_0})u, \tilde{\Pi}_{\mu_0} u\rangle_{H_{\Phi_2}} = 0
\]
for each polynomial $u$, because polynomials are dense in $H_{\Phi_2}$ and we know that $\tilde{\Pi}_{\mu_0}$ is continuous (e.g.\ from Lemma \ref{lMainBounds}).  Since
\[
	(1-\tilde{\Pi}_{\mu_0})\left(\sum_{|\alpha| \leq N} a_\alpha x^\alpha\right) = \sum_{1 \leq |\alpha| \leq N} a_\alpha x^\alpha
\]
and
\[
	\tilde{\Pi}_{\mu_0} \left(\sum_{|\alpha| \leq N} a_\alpha x^\alpha\right) = a_0,
\]
condition \ref{iOrthogproj} is equivalent to the condition that the constant function $1$ is orthogonal to any polynomial vanishing at the origin.  This is turn is equivalent to requiring the constant function $1$ to be orthogonal to any polynomial multiplied by any $x_j$:
\[
	\langle x_j p(x), 1\rangle_{H_{\Phi_2}} = 0, \quad \forall p(x) \in \Bbb{C}[x_1,\dots,x_n], \quad \forall j=  1,\dots,n.
\]
Taking adjoints and using the density of polynomials in $H_{\Phi_2}$ gives the equivalent condition
\begin{equation}\label{eOrthogProjPf1}
	x^*_{H_{\Phi_2}} 1 = 0.
\end{equation}
Via (\ref{exStar}), we thus have that
\[
	x^*_{H_{\Phi_2}} 1 = -((\Phi_2)_{x\bar{x}}'')^{-1}(\Phi_2)''_{xx}x = 0,
\]
which establishes that condition \ref{iOrthogproj} is also equivalent to the statement $(\Phi_2)''_{xx}=0$.

The same unitary equivalence from Proposition \ref{pNormalForm} makes condition \ref{iOrthogker} equivalent to
\[
	\ker(\tilde{q}^w(x,hD_x)^*-\overline{\mu_0}) = \ker(\tilde{q}^w(x,hD_x)-\mu_0),
\]
where the right-hand side is the span of the constant function.  From Lemma \ref{lAdjointEigenvectors}, it is clear that this is equivalent to assuming that $\phi^\dagger_0$ is a constant function.  We have already seen in (\ref{ephi0daggerDef}) that this is the same as insisting that $(\Phi_2)''_{xx} = 0$.

This completes the proof of the theorem.  Some numerical analysis based on the resulting decomposition in Remark \ref{rFurtherDecomposition} is presented above in Section \ref{ssNumerics}.

\subsection{Transferring $\partial_x^*$ to a multiplication operator}\label{ssTransferringdxStar}

In the reduction to normal form in Proposition \ref{pNormalForm}, we reduce $q^w(x,hD_x)$ acting on $L^2(\Bbb{R}^n)$ to an operator $\tilde{q}^w(x,hD_x)$ acting on $H_{\Phi_2}(\Bbb{C}^n;h)$ for which the monomials form a basis of generalized eigenvectors.  We have seen in Lemma \ref{lAdjointEigenvectors} that the $\{\phi^\dagger_\alpha\}$ of Section \ref{ssDualBasis} form a basis of generalized eigenvectors for $\tilde{q}^w(x,hD_x)^*_{H_{\Phi_2}}$.  It is therefore natural to expect that applying Proposition \ref{pNormalForm} to $q^w(x,hD_x)^*$ would convert the $\{\phi^\dagger_\alpha\}$ into $\{x^\alpha\}$ in a different weighted space.

We do not reference $q^w(x,hD_x)^*$ here explicitly; we merely use the fact that taking adjoints of quantizations acting on $L^2(\Bbb{R}^n)$ takes the complex conjugate of the principal symbol.  We only need to consider the obvious effect on the stable manifolds defined in (\ref{eLambdaDef}):
\[
	\Lambda^{\pm}(\bar{q}) = \overline{\Lambda^\mp(q)}.
\]
Given this strategy, the computations which follow are natural and routine, but are included for completeness.

In view of Proposition \ref{pReversibility}, we assume that $\Phi$ is in the form given by (\ref{ePhi2Matrix}).  We begin with the case $G = 1$, meaning that we take
\[
	\Phi(x) = \Phi_1(x) = \frac{1}{4}\left(|x|^2-\jvRe(x,C_+x)\right)
\]
for
\[
	C_+ = (1-iA_+)^{-1}(1+iA_+),\quad A_+^t = A_+,\quad \jvIm A_+ > 0.
\]
This will be easily extended, since $(\partial_x)^*_{H_{\Phi_2}} = ((i/h)\xi^w)^*_{H_{\Phi_2}}$ may be computed as an operator on $H_{\Phi_1}$ because the change of variables which maps $H_{\Phi_1}$ to $H_{\Phi_2}$ quantizes the map $(x,\xi)\mapsto (G^{-1} x, G^t\xi)$.  This means that there is a unitary equivalence between $(\partial_x)^*_{H_{\Phi_2}}$ and
\begin{equation}\label{edStarPhi1ToPhi2}
	\left(\left(G^t\frac{i}{h}\xi\right)^w\right)^*_{H_{\Phi_1}} = \overline{G^t}(\partial_x)^*_{H_{\Phi_1}}.
\end{equation}
We therefore continue with $\Phi_1$ and apply this computation afterwards.

We then apply (\ref{ehDxStar}) and (\ref{exStar}) with derivatives taken from (\ref{ePhi2Matrix}) when $G = 1$, the identity matrix.  We obtain
\begin{equation}\label{edStarPhi1}
	(\partial_x)^*_{H_{\Phi_1}} = \frac{1}{2h}\left((1-C_+^*C_+)x+\frac{2}{i}C_+^*hD_x\right)
\end{equation}
and
\begin{equation}\label{exStarPhi1}
	x^*_{H_{\Phi_1}} = C_+x + 2ihD_x.
\end{equation}

We then invert the FBI transform described in the proof of Proposition \ref{pNormalForm} with canonical transformation from (\ref{ekapADef}),
\[
	\kap_{A_+} = \left(\begin{array}{cc} 1 & -i \\ -(1-iA_+)^{-1}A_+ & (1-iA_+)^{-1}\end{array}\right).
\]

Recalling that
\[
	\kap_{A_+}^{-1}(\{x = 0\}) = \Lambda^- = \{(y,-iy)\}_{y\in\Bbb{C}^n}
\]
and
\[
	\kap_{A_+}^{-1}(\{\xi = 0\}) = \Lambda^+ = \{(y,A_+y)\}_{y\in\Bbb{C}^n},
\]
we reverse the roles of $\Lambda^+$ and $\Lambda^-$ by taking their complex conjugates.  We straighten $\overline{\Lambda^+}$ to $\{(y,-iy)\}$ as in the proof of Proposition \ref{pNormalForm}. From (\ref{ekappaDef}) we recall that with
\[
	\kappa = \left(\begin{array}{cc} (\jvIm A_+)^{1/2} & 0 \\ -(\jvIm A_+)^{-1/2}\jvRe A_+ & (\jvIm A_+)^{-1/2}\end{array}\right)
\]
we have 
\[
	\kappa(\overline{\Lambda^+}) = \{(y,-iy)\}_{y\in\Bbb{C}^n},
\]
and we furthermore note that
\[
	\kappa(\overline{\Lambda^-}) = \{(y,\tilde{A}_+ y)\}_{y\in\Bbb{C}^n},
\]
where here
\begin{equation}\label{etildeAplusDef}
	\tilde{A}_+ = (\jvIm A_+)^{-1/2}(i-\jvRe A_+)(\jvIm A_+)^{-1/2}.
\end{equation}

We expect that canonical transformations of these Lagrangian planes should be performed via unitary operators quantizing these canonical transformations, since we have the rule
\[
	\Lambda^{\pm}(q\circ\kappa^{-1}) = \kappa(\Lambda^\pm(q)).
\]
We only need to recall, following the proof of Proposition \ref{pNormalForm}, that there exists some unitary transformation
\[
	\tilde{\mathcal{T}}_0:H_{\Phi_1}\rightarrow H_{\Phi_1^\dagger}
\]
quantizing the complex linear canonical transformation $K^{-1}$ with
\[
	K = \kap_{A_+} \circ \kappa^{-1}\circ \kap_{\tilde{A}_+}^{-1}.
\]
Here
\[
	\Phi_1^\dagger(x) = \frac{1}{4}(|x|^2 - \jvRe(x,\tilde{C}_+ x))
\]
for
\[
	\tilde{C}_+ = (1-i\tilde{A}_+)^{-1}(1+i\tilde{A}_+).
\]

All that remains is to compute $K$ and discover what becomes of the symbols of $(\partial_x)^*_{H_{\Phi_1}}$ and $x^*_{H_{\Phi_1}}$ by computing, using (\ref{edStarPhi1}) and (\ref{exStarPhi1}),
\begin{equation}\label{edStarTransformSymbol}
	\sigma((\partial_x)^*_{H_{\Phi_1}})\circ K = \left.\frac{1}{2h}\left((1-C_+^*C_+)y+\frac{2}{i}C_+^*\eta\right)\right|_{(y,\eta) = K(x,\xi)}
\end{equation}
and
\begin{equation}\label{exStarTransformSymbol}
	\sigma(x^*_{H_{\Phi_1}}) \circ K = \left. (C_+ y + 2i \eta )\right|_{(y,\eta) = K(x,\xi)}.
\end{equation}
Note that the symbols $\sigma((\partial_x)^*_{H_{\Phi_1}})$ and $\sigma(x^*_{H_{\Phi_1}})$ naturally take their inputs from $\Lambda_{\Phi_1}$ defined in (\ref{eLambdaPhiDef}), while $\sigma((\partial_x)^*_{H_{\Phi_1}})\circ K$ and $\sigma(x^*_{H_{\Phi_1}})\circ K$ take their inputs from $\Lambda_{\Phi_1^\dagger}$ for $\Phi_1^\dagger$ associated with $\tilde{A}_+$.

We decompose the maps involved in $K$ as much as possible:
\[
	\kap_{A_+} = \left(\begin{array}{cc} 1 & 0 \\ 0 & (1-iA_+)^{-1}\end{array}\right) \left(\begin{array}{cc} 1 & -i \\ -A_+ & 1\end{array}\right),
\]
\[
	\kappa^{-1} = \left(\begin{array}{cc} 1 & 0 \\ \jvRe A_+ & 1\end{array}\right) \left(\begin{array}{cc} (\jvIm A_+)^{-1/2} & 0 \\ 0 & (\jvIm A_+)^{1/2}\end{array}\right),
\]
and
\[
	\kap_{\tilde{A}_+}^{-1} = \left(\begin{array}{cc} 1 & i \\ \tilde{A}_+ & 1\end{array}\right) \left(\begin{array}{cc} (1-i\tilde{A}_+)^{-1} & 0 \\ 0 & 1\end{array}\right)
\]
In addition to the definition (\ref{etildeAplusDef}) of $\tilde{A}_+$, we compute
\begin{multline}\label{e1mtApInv}
	(1-i\tilde{A}_+)^{-1} = \left((\jvIm A_+)^{-1/2}(\jvIm A_+ - i(i - \jvRe A_+))(\jvIm A_+)^{-1/2}\right)^{-1}
	\\ = (\jvIm A_+)^{1/2}(1+iA_+^*)^{-1}(\jvIm A_+)^{1/2}.
\end{multline}
It is then easy to show that
\[
	\kappa^{-1}\circ\kap_{\tilde{A}_+}^{-1} = \left(\begin{array}{cc} 1 & i \\ i & iA_+^*\end{array}\right)\left(\begin{array}{cc} (1+iA_+^*)^{-1} & 0 \\ 0 & 1\end{array}\right) \left(\begin{array}{cc} (\jvIm A_+)^{1/2} & 0 \\ 0 & (\jvIm A_+)^{-1/2}\end{array}\right)
\]
and, recalling definition (\ref{eCplusDef}) of $C_+$, that
\[
	K = \left(\begin{array}{cc} 2(1+iA_+^*)^{-1}(\jvIm A_+)^{1/2} & i(1-iA_+^*)(\jvIm A_+)^{-1/2} \\ iC_+(1+iA_+^*)^{-1}(\jvIm A_+)^{1/2} & 2(1-iA_+)^{-1}(\jvIm A_+)^{1/2}\end{array}\right).
\]

We now wish to compute the symbols corresponding to $(\partial_x)^*_{H_{\Phi_1}}$ and $x^*_{H_{\Phi_1}}$.  From (\ref{edStarTransformSymbol}), we have
\begin{multline*}
	2h\sigma((\partial_x)^*_{H_{\Phi_1}})\circ K 
	\\ = (1-C_+^*C_+)\left(2(1+iA_+^*)^{-1}(\jvIm A_+)^{1/2}x+i(1-iA_+^*)(\jvIm A_+)^{-1/2}\xi\right)
	\\ + \frac{2}{i} C_+^*\left(iC_+(1+iA_+^*)^{-1}(\jvIm A_+)^{1/2}x + 2(1-iA_+)^{-1}(\jvIm A_+)^{1/2}\xi\right).
\end{multline*}
The coefficient of $x$ is clearly $2(1+iA_+^*)^{-1}(\jvIm A_+)^{1/2}$.

We expect the coefficient of $\xi$ to vanish.  Using (\ref{e1MinusAbsCplus}), we have that the coefficient of $\xi$ is
\begin{multline*}
	i(1-C_+^*C_+)(1-iA_+^*)(\jvIm A_+)^{-1/2} + \frac{4}{i} C_+^*(1-iA_+)^{-1}(\jvIm A_+)^{1/2}
	\\ = 4i(1+iA_+^*)^{-1}\left(\jvIm A_+(1-iA_+)^{-1}(1-iA_+^*)\right.
	\\ \left.- (1-iA_+^*)(1-iA_+)^{-1}\jvIm A_+\right)(\jvIm A_+)^{-1/2}.
\end{multline*}
Since $A_+$ is symmetric, we see that the coefficient of $\xi$ vanishes if and only if the matrix
\[
	\jvIm A_+(1-iA_+)^{-1}(1-iA_+^*)
\]
is symmetric.  The computation
\begin{multline*}
	\jvIm A_+(1-iA_+)^{-1}(1-iA_+^*) = \frac{1}{2i}(A_+ - A_+^*)(1-iA_+)^{-1}(1-iA_+^*)
	\\ = \frac{1}{2}\left(1-iA_+ - (1-iA_+^*)\right)(1-iA_+)^{-1}(1-iA_+^*)
	\\ = \frac{1}{2}\left(1-(1-iA_+^*)(1-iA_+)^{-1}(1-iA_+^*)\right)
\end{multline*}
yields an obviously symmetric matrix.  This completes the proof that the coefficient of $\xi$ is zero, and so
\[
	\sigma((\partial_x)^*_{H_{\Phi_1}})\circ K = \frac{1}{h}(1+iA_+^*)^{-1}(\jvIm A_+)^{1/2} x.
\]

Next, from (\ref{exStarTransformSymbol}), we have
\begin{multline*}
	\sigma(x^*_{H_{\Phi_1}}) \circ K
	\\ = C_+\left(2(1+iA_+^*)^{-1}(\jvIm A_+)^{1/2}x + i(1-iA_+^*)(\jvIm A_+)^{-1/2}\xi\right)
	\\ + 2i\left(iC_+(1+iA_+^*)^{-1}(\jvIm A_+)^{1/2}x + 2(1-iA_+)^{-1}(\jvIm A_+)^{1/2}\xi\right).
\end{multline*}
Here, it is easy to see that the coefficient of $x$ vanishes.  We then compute
\begin{multline*}
	\sigma(x^*_{H_{\Phi_1}}) \circ K = i(1-iA_+)^{-1}\left((1+iA_+)(1-iA_+^*) + 4 \jvIm A_+\right)(\jvIm A_+)^{-1/2}\xi
	\\ = i(1-iA_+)^{-1}(1-iA_+)(1+iA_+^*)(\jvIm A_+)^{-1/2} \xi.
\end{multline*}
We therefore arrive at the conclusion that
\[
	\sigma(x^*_{H_{\Phi_1}}) \circ K = i(1+iA_+^*)(\jvIm A_+)^{-1/2}\xi.
\]

We now extend our discussion from $\Phi_1$ to include $\Phi_2$ as in (\ref{ePhi2Matrix}).  As a result, our new unitary transformation
\[
	\tilde{\mathcal{T}}:H_{\Phi_2}\rightarrow H_{\Phi_1^\dagger}
\]
may be obtained by composing $\tilde{\mathcal{T}}_0$ with the unitary change of variables
\begin{equation}\label{eChgVarsPhi2ToPhi1}
	H_{\Phi_2} \ni u(x)\mapsto |\det G|^{-1}u(G^{-1}x)\in H_{\Phi_1}.
\end{equation}
As in the discussion leading up to (\ref{edStarPhi1ToPhi2}), it is straightforward to use the canonical transformation quantized by the change of variables $u(x) \mapsto |\det G| u(Gx)$ to obtain the following unitary equivalences:
\begin{equation}\label{edStarPhi1Dagger}
	\tilde{\mathcal{T}}(\partial_x)^*_{H_{\Phi_2}}\tilde{\mathcal{T}}^* = \frac{1}{h}\overline{G^t}(1+iA_+^*)^{-1}(\jvIm A_+)^{1/2}x:H_{\Phi_1^\dagger}\rightarrow H_{\Phi_1^\dagger}
\end{equation}
and
\[
	\tilde{\mathcal{T}}x^*_{H_{\Phi_2}}\tilde{\mathcal{T}}^* = i\overline{G^{-1}}(1+iA_+^*)(\jvIm A_+)^{-1/2}hD_x:H_{\Phi_1^\dagger}\rightarrow H_{\Phi_1^\dagger}.
\]

Because $\phi_0^\dagger$ was defined via the equation $x^*_{H_{\Phi_2}}\phi_0^\dagger = 0$, we know that 
\[
	hD_x \tilde{\mathcal{T}}\phi_0^\dagger = 0.
\]
We conclude that $\tilde{\mathcal{T}} \phi_0^\dagger \in H_{\Phi_1^\dagger}$ is a constant function, and since $\tilde{\mathcal{T}}$ is unitary, we may determine the constant through the equality
\begin{equation}\label{ephi0daggerConstantNorm}
	||\tilde{\mathcal{T}} \phi_0^\dagger||_{H_{\Phi_1^\dagger}} = ||\phi_0^\dagger||_{H_{\Phi_2}}.
\end{equation}
We begin with $||\phi_0^\dagger||_{H_{\Phi_2}}$, recalling having already computed $C_0$ in (\ref{eC0Result}).  Noting from (\ref{ePhi2Matrix}) that
\[
	(\Phi_2)''_{x\bar{x}} = \frac{1}{4}G^t\overline{G},
\]
we have that
\[
	C_0 = \left(\frac{2}{\pi h}\right)^n \det (\Phi_2)''_{x\bar{x}} = (2\pi h)^{-n}|\det G|^2.
\]
We refer to the observation (\ref{ephidaggerAndWeight}) to see that
\[
	||\phi_0^\dagger||_{H_{\Phi_2}} = |C_0|\:||1||_{H_{(\Phi_2)_{\textnormal{herm}}}}.
\]
The change of variables (\ref{eChgVarsPhi2ToPhi1}) allows us to see that 
\[
	||1||_{H_{(\Phi_2)_{\textnormal{herm}}}} = |\det G|^{-1}||1||_{H_{(\Phi_1)_{\textnormal{herm}}}},
\]
but it is clear from (\ref{ePhi1Def}) and (\ref{ePhihermDef}) that $(\Phi_1)_{\textnormal{herm}} = \frac{1}{4}|x|^2$.  Therefore
\begin{multline}\label{ephi0daggerNorm}
	||\phi_0^\dagger||_{H_{\Phi_2}} = (2\pi h)^{-n}|\det G| \left(\int_{\Bbb{C}^n} \jvexp\left(-\frac{1}{2h}|x|^2\right)\,dL(x)\right)^{1/2} 
	\\ = (2\pi h)^{-n} |\det G| (2\pi h)^{n/2} = (2\pi h)^{-n/2}|\det G|.
\end{multline}

We know that there exists some $\tilde{C}_0$ for which $\tilde{\mathcal{T}}\phi_0^\dagger = \tilde{C}_0$, and so we need to compute
\begin{equation}\label{etildeC0Computation}
	||\tilde{\mathcal{T}} \phi_0^\dagger||_{H_{\Phi_1^\dagger}} = |\tilde{C}_0|\,||1||_{H_{\Phi_1^\dagger}}.
\end{equation}
(Naturally, we are only interested in the absolute value of $\tilde{C}_0$.)  In order to apply (\ref{eComplexGaussianIntegral}), we write
\[
	\int_{\Bbb{C}^n} \jvexp\left(-\frac{2}{h}\Phi_1^\dagger(x)\right)\,dL(x) = \int_{\Bbb{C}^n} e^{\frac{i}{h}Q(x)}\,dL(x)
\]
for
\[
	Q(x) = \frac{i}{2}\left(|x|^2 - \jvRe (x,\tilde{C}_+ x)\right).
\]
Then
\[
	\nabla_{x,\bar{x}}^2 Q(x) = \frac{i}{2}\left(\begin{array}{cc} -\tilde{C}_+ & 1 \\ 1 & -\tilde{C}_+^*\end{array}\right).
\]
To avoid issues with block matrices, we perform row reduction by adding to the first row the result of postmultiplying the second row by $\tilde{C}_+$ and permute $n$ rows, obtaining
\[
	\det\left(\begin{array}{cc} -\tilde{C}_+ & 1 \\ 1 & -\tilde{C}_+^*\end{array}\right) = \det\left(\begin{array}{cc} 0 & 1-\tilde{C}_+^* \tilde{C}_+ \\ 1 & -\tilde{C}_+^*\end{array}\right) = (-1)^n\det\left(\begin{array}{cc} 1 & - \tilde{C}_+^* \\ 0 & 1-\tilde{C}_+^*\tilde{C_+}\end{array}\right).
\]
At this point, it becomes clear that
\[
	\opnm{det} (\nabla_{x,\bar{x}}^2 Q) = \left(\frac{i}{2}\right)^{2n}(-1)^n\det(1-\tilde{C}_+^* \tilde{C}_+) = 2^{-2n}\det(1-\tilde{C}_+^* \tilde{C}_+).
\]
We therefore conclude from (\ref{eComplexGaussianIntegral}) that
\begin{equation}\label{eNorm1PhiDag}
	||1||_{H_{\Phi_1^\dagger}} = (2\pi h)^{n/2}\det(1-\tilde{C}_+^* \tilde{C}_+)^{-1/4}.
\end{equation}
We see from (\ref{etildeAplusDef}) that $\jvIm \tilde{A}_+ = (\jvIm A_+)^{-1}$.  We may then use (\ref{e1MinusAbsCplus}) and (\ref{e1mtApInv}) to compute that                                                                                                                                                                                                                                                                                                                                                                                                     
\begin{multline*}
	1-\tilde{C}_+^*\tilde{C}_+ = (1+i\tilde{A}_+^*)^{-1}(4\jvIm \tilde{A}_+)(1-i\tilde{A}_+)^{-1}
	\\ = 4(\jvIm A_+)^{1/2}(1-iA_+)^{-1}(1+iA_+^*)^{-1}(\jvIm A_+)^{1/2},
\end{multline*}
and therefore, again using (\ref{e1MinusAbsCplus}),
\begin{equation}\label{edet1MinusCPlus}
	\det (1-\tilde{C}_+^*\tilde{C}_+) = \det(1-C_+^*C_+).
\end{equation}

We therefore use (\ref{etildeC0Computation}) and (\ref{eNorm1PhiDag}) to obtain
\[
	||\tilde{\mathcal{T}} \phi_0^\dagger||_{H_{\Phi_1^\dagger}} = |\tilde{C}_0| (2 \pi h)^{n/2} \det(1-C_+^*C_+)^{-1/4}.
\]
From (\ref{ephi0daggerConstantNorm}) and (\ref{ephi0daggerNorm}) we then compute that
\begin{equation}\label{etildeC0Final}
	|\tilde{C}_0| = (2\pi h)^{-n}|\det G|\det(1-C_+^*C_+)^{1/4}.
\end{equation}

So far, we therefore have from (\ref{ephidaggerTrivialDef}) and (\ref{edStarPhi1Dagger}) a unitary equivalence between $\phi_\alpha^\dagger \in H_{\Phi_2}$ and
\[
	\tilde{C}_0 (\alpha!)^{-1}(2h)^{-|\alpha|} \left(2\overline{G^t}(1+iA_+^*)^{-1}(\jvIm A_+)^{1/2}x\right)^\alpha \in H_{\Phi_1^\dagger}.
\]
We then make a final change of variables $u(x)\mapsto |\det \tilde{G}| u(\tilde{G} x)$ with
\begin{equation}\label{eGtildeDef}
	\tilde{G} = \frac{1}{2}(\jvIm A_+)^{-1/2}(1+iA_+^*)(G^*)^{-1}.
\end{equation}

Writing as usual $\Phi_2^\dagger(x) = \Phi_1^\dagger (\tilde{G} x)$, we have from the unitary equivalence
\[
	||\phi_\alpha^\dagger||_{H_{\Phi_2}} = |\tilde{C}_0| (\alpha!)^{-1}(2h)^{-|\alpha|}|\det \tilde{G}|\,||x^\alpha||_{H_{\Phi_2^\dagger}}.
\]
Using (\ref{e1MinusAbsCplus}) and (\ref{eGtildeDef}), it is easy to see that
\begin{multline*}
	|\det \tilde{G}|^2 = \det(1+iA_+^*)\overline{\det(1+iA_+^*)}\det(4\jvIm A_+)^{-1}|\det G|^{-2} 
	\\ = \det(1-C_+^*C_+)^{-1}|\det G|^{-2}.
\end{multline*}
Combining this with (\ref{etildeC0Final}) gives
\[
	|\tilde{C}_0|\:|\det \tilde{G}| = (2\pi h)^{-n}\det(1-C_+^*C_+)^{-1/4}.
\]

In order to simplify 
\begin{multline*}
	\Phi_2^\dagger(x) = \frac{1}{4}(|\tilde{G} x|^2 - \jvRe(\tilde{G} x, \tilde{C}_+ \tilde{G} x))
	\\ = \frac{1}{4}\left(\langle x, (\tilde{G})^* \tilde{G} x\rangle - \jvRe(x, (\tilde{G})^t \tilde{C}_+ \tilde{G} x)\right),
\end{multline*}
we may compute 
\[
	4(\Phi_2^\dagger)''_{\bar{x}x} = (\tilde{G})^*\tilde{G}, \quad 4(\Phi_2^\dagger)''_{xx} = -(\tilde{G})^t \tilde{C}_+ \tilde{G}.
\]
First, using (\ref{e1MinusAbsCplus}),
\begin{eqnarray*}
	(\tilde{G})^*(\tilde{G}) &=& \frac{1}{4}G^{-1}(1-iA_+)(\jvIm A_+)^{-1}(1+iA_+^*)\overline{G^{-t}}
	\\ &=& G^{-1}(1-C_+^*C_+)^{-1}\overline{G^{-t}}.
\end{eqnarray*}
Then we have
\begin{eqnarray*}
	(\tilde{G})^t\tilde{C}_+(\tilde{G}) &=& \frac{1}{4}\overline{G^{-1}}(1+iA_+^*)(\jvIm A_+)^{-1/2}(\jvIm A_+)^{1/2}(1+iA_+^*)^{-1}
	\\ &&\times (-1-iA_+)(\jvIm A_+)^{-1/2}(\jvIm A_+)^{-1/2} (1+iA_+^*)\overline{G^{-t}}
	\\ &=& -\frac{1}{4}\overline{G^{-1}}(1+iA_+)(\jvIm A_+)^{-1}(1+iA_+^*)\overline{G^{-t}}
	\\ &=& -\overline{G^{-1}}C_+(1-C_+^*C_+)^{-1}\overline{G^{-t}}.
\end{eqnarray*}

Following Remark \ref{rGCFromConvexity}, we seek to write $\Phi_2^\dagger$ using an invertible matrix $G^\dagger$ and a symmetric matrix $C_+^\dagger$ for which
\[
	(\Phi_2^\dagger)_{\bar{x}x}'' = \frac{1}{4}(G^\dagger)^*G^\dagger
\]
and
\[
	(\Phi_2^\dagger)_{xx}'' = -\frac{1}{4}(G^\dagger)^t C_+^\dagger G^\dagger.
\]
It is natural then to define
\[
	E = (1 - C_+^*C_+)^{-1/2}
\]
via the usual selfadjoint positive definite functional calculus.  We then write
\begin{equation}\label{eGdaggerDef}
	G^\dagger = E(G^*)^{-1}
\end{equation}
and
\begin{equation}\label{eCdaggerDef}
	C_+^\dagger = -E^{-t} C_+ E.
\end{equation}
The formula for $\Phi_2^\dagger$ then follows from the formula (\ref{ePhi2Matrix}): in the same way, since $\Phi$ is quadratic and real-valued, it is sufficient to identify the derivatives
\begin{equation}\label{ePhi2daggerMatrix}
	(\Phi_2^\dagger)''_{\bar{x}x} = \frac{1}{4} (G^\dagger)^*G^\dagger, \quad (\Phi_2^\dagger)''_{xx} = -\frac{1}{4}(G^\dagger)^tC_+^\dagger G^\dagger.
\end{equation}

With this $\Phi_2^\dagger$, we have that
\begin{multline*}
	||\Pi_\alpha||_{\mathcal{L}(H_{\Phi_2})} = ||x^\alpha||_{H_{\Phi_2}} \,||\phi^\dagger_\alpha||_{H_{\Phi_2}}
	\\ = (2\pi h)^{-n}\det(1-C_+^*C_+)^{-1/4}(\alpha!)^{-1}(2h)^{-|\alpha|}||x^\alpha||_{H_{\Phi_2}}\, ||x^\alpha||_{H_{\Phi_2^\dagger}}.
\end{multline*}
As the $h$-dependence can be eliminated by a simple change of variables, this proves Theorem \ref{tBoundProj}.

\section{Computation for norms of spectral projections}\label{sComputations}

We here perform explicit computations using Theorem \ref{tBoundProj} in order to obtain information about the norms of spectral projections.  Throughout, we assume $h = 1$, in view of (\ref{eRescalingProjections}).

We begin in Section \ref{ssPolar} by using polar coordinates to reduce (\ref{eBoundProj}) to an integral on the unit sphere $\{|\omega| = 1\}$, which immediately gives Corollary \ref{cUpperBound}.  In Section \ref{ssDim1Asymptotic}, we consider the case of (spatial) dimension $n=1$ and use Laplace's method to deduce the complete asymptotic expansion in Corollary \ref{cDim1Asymptotic}.  We then expand the discussion to arbitrary $n$ in Section \ref{ssGrowthAnyDim}, proving the general rate of exponential growth in Corollary \ref{cExponentialGrowth}.  Finally, we present the numerical computation of these rates of growth for certain examples in Section \ref{ssNumericsExponentialGrowth}.

\subsection{Reduction to the unit sphere}\label{ssPolar}

In view of Theorem \ref{tBoundProj}, it is appropriate to study
\begin{equation}\label{eJPhiDef}
	J(\Phi, \alpha) := (2\pi)^{-n}2^{-|\alpha|}(\alpha!)^{-1}||x^\alpha||_{H_\Phi}^2.
\end{equation}
since
\begin{equation}\label{ePialphaJPhi}
	||\Pi_\alpha||_{\mathcal{L}(H_{\Phi})}^2 = \det(1-C_+^*C_+)^{-1/2}J(\Phi,\alpha) J(\Phi^\dagger, \alpha).
\end{equation}

Write
\begin{eqnarray*}
	||x^\alpha||^2_{H_\Phi} &=& \int_{\Bbb{C}^n} |x^\alpha|^2 e^{-2\Phi(x)}\,dL(x)
	\\ &=& \int_{|\omega|=1}\int_0^\infty r^{2|\alpha|+2n-1} |\omega^\alpha|^2 e^{-2 r^2 \Phi(\omega)}\,dr\,dL(\omega)
	\\ &=& 2^{|\alpha|+n-1}(|\alpha|+n-1)! \int_{|\omega|=1}|\omega^\alpha|^2(4\Phi(\omega))^{-|\alpha|-n}\,dL(\omega),
\end{eqnarray*}
where $dL(\omega)$ is induced by Lebesgue measure restricted to $\{|\omega| = 1\}$.  

We may see that
\begin{equation}\label{eNormalizationFactor}
	(2\pi)^{-n}\frac{(|\alpha|+n-1)!}{\alpha!} 2^{n-1} \int_{|\omega|=1}|\omega^\alpha|^2\,dL(\omega) = 1
\end{equation}
through explicit computation of $||x^\alpha||^2_{H_\Phi}$ where $\Phi(x) = \frac{1}{4}|x|^2$.  Alternately, we may note that $\Phi(x) = \Phi^\dagger(x) = \frac{1}{4}|x|^2$ are the weight function and dual weight function obtained when considering the operator
\begin{equation}\label{eHarmonicRationallyIndependent}
	\sum_{j=1}^n r_j(x_j^2 + (hD_{x_j})^2)
\end{equation}
for $r_j > 0$ chosen rationally independent. This may be directly deduced from computing that $\Lambda^{\pm} = \{(x,\pm i x)\}_{x\in \Bbb{C}^n}$.  The rational independence of the $r_j$ means that all its eigenvalues, obtained from (\ref{emuDef}), are distinct, and since the symbol is real-valued, the operator is selfadjoint and so the spectral projections all have norm one.  We then may combine our computation of $||x^\alpha||^2_{H_\Phi}$ in polar coordinates with (\ref{eJPhiDef}) and (\ref{ePialphaJPhi}) to arrive at the conclusion that, in this case,
\[
	1 = J(\Phi, \alpha)^2 = \left(\frac{2^{|\alpha|+n-1}(|\alpha|+n-1)!}{(2\pi)^n 2^{|\alpha|}\alpha !}\int_{|\omega|=1}|\omega^\alpha|^2 \,dL(\omega)\right)^2,
\]
proving (\ref{eNormalizationFactor}).

Regarding 
\begin{equation}\label{eCalphan}
	C(\alpha;n) = \int_{|\omega|=1}|\omega^\alpha|^2\,dL(\omega) = 2 \pi^n \frac{\alpha!}{(|\alpha|+n-1)!}
\end{equation}
as a normalizing factor, we obtain
\begin{equation}\label{eJPhiPolar}
	J(\Phi,\alpha) = \int_{|\omega| = 1} (4\Phi(\omega))^{-|\alpha|-n}\frac{|\omega^\alpha|^2\,dL(\omega)}{C(\alpha;n)}.
\end{equation}

For $\mu_\alpha$ defined via (\ref{emuDef}) with $h = 1$, recalling that $\jvIm \lambda_j > 0$, we may write
\[
	m = \min \jvIm \lambda_j, \quad M = \max \jvIm \lambda_j.
\]
Then, for any $\alpha \in \Bbb{N}_0^n$,
\[
	(2|\alpha|+n)m \leq \jvRe \mu_\alpha \leq (2|\alpha|+n)M.
\]
We then compute that $\mu_\beta = \mu_\alpha$ implies that
\[
	\frac{m}{M}|\alpha| - \frac{n}{2}\left(1-\frac{m}{M}\right) \leq |\beta| \leq \frac{M}{m}|\alpha| + \frac{n}{2}\left(\frac{M}{m}-1\right).
\]
It is then elementary that
\[
	\#\{\beta\::\:\mu_\beta = \mu_\alpha\} = \BigO(1+|\alpha|^{n-1}).
\]
(See, for instance, the proof of (68) in \cite{Vi2012b}.)

Corollary \ref{cUpperBound} then follows from the triangle inequality and the elementary bound
\begin{multline*}
	J(\Phi, \alpha) = \int_{|\omega| = 1} (4\Phi(\omega))^{-|\alpha|-n}\frac{|\omega^\alpha|^2\,dL(\omega)}{C(\alpha;n)} 
	\\ \leq \left(\inf_{|\omega|=1} 4\Phi(\omega)\right)^{-|\alpha|-n}\int_{|\omega| = 1} \frac{|\omega^\alpha|^2\,dL(\omega)}{C(\alpha;n)}  = \BigO(1)\left(\inf_{|\omega|=1} 4\Phi(\omega)\right)^{-|\alpha|}
\end{multline*}
applied to (\ref{ePialphaJPhi}).

\subsection{Asymptotic expansion in one dimension}\label{ssDim1Asymptotic}

In the dimension 1 case, we may write $N$ instead of $\alpha$.  We note from (\ref{eCalphan}) that $C(N,1) = 2\pi$.  Furthermore, $\{|\omega| = 1\}$, as a subset of $\Bbb{C}$ with measure induced by Lebesgue measure $d\jvRe x \, d\jvIm x$, is the same as $\{e^{it}\::\:0 \leq t < 2\pi\}$ with measure $dt$.  Then
\[
	J(\Phi,N) = \frac{1}{2\pi}\int_0^{2\pi}(4\Phi(e^{it}))^{-N-1}\,dt.
\]
Using (\ref{ePhiDecompose}) in dimension 1, we see that
\[
	\Phi(e^{it}) = \Phi''_{x\bar{x}}\left(1+\jvRe((\Phi''_{x\bar{x}})^{-1}\Phi''_{xx}e^{2it})\right).
\]
The dimension 1 case is particularly simple to analyze because complex numbers commute.  Following Remark \ref{rGCFromConvexity}, we may take
\[
	G = 2(\Phi_{\bar{x}x}'')^{1/2}, \quad C_+ = -\frac{\Phi''_{xx}}{\Phi''_{\bar{x}x}}.
\]
We then have that
\[
	\Phi(e^{it}) = \frac{G^2}{4}\left(1 - \jvRe(C_+ e^{2it})\right) = \frac{G^2}{4}\left(1-|C_+|\cos(2t+\arg C_+)\right).
\]
We furthermore have $G > 0$, since $\Phi''_{\bar{x}x}$ is positive definite Hermitian in any dimension. We also note that $|C_+| < 1$ by strict convexity of $\Phi$, either as shown in the proof of Proposition \ref{pReversibility} or, more simply, by observing that strict convexity requires that $\Phi(e^{-i\arg C_+/2}) > 0$.  Making a change of variables $\tau = t+\arg(C_+)/2$ reduces the study of $J(\Phi,N)$ to the study of
\begin{equation}\label{eJPhivsA}
	J(\Phi,N) = \frac{1}{2\pi}(G^2)^{-N-1}\int_0^{2\pi} (1-|C_+|\cos 2t)^{-N-1}\,dt.
\end{equation}

We turn to $\Phi^\dagger$ in the dimension 1 case, where
\[
	G^\dagger = (1-|C_+|^2)^{-1/2}G^{-1}, \quad C_+^\dagger = -C_+.
\]
We use the same reasoning which provided (\ref{eJPhivsA}) to obtain
\begin{equation}\label{eJPhiBoth}
	J(\Phi,N)J(\Phi^\dagger,N) = (1-|C_+|^2)^{N+1}\left(\frac{1}{2\pi}\int_0^{2\pi} (1-|C_+|\cos 2t)^{-N-1}\,dt\right)^2.
\end{equation}
We turn to finding an asymptotic expansion for the integral.

Such integrals are well-studied by means of Laplace's method (see for instance Chapter 3 of \cite{MiBook}).  We rewrite the integral in (\ref{eJPhiBoth}) as
\begin{equation}\label{eJPhivsA2}
	\int_0^{2\pi} (1-|C_+|\cos 2t)^{-N-1}\,dt = 2\int_{-\pi/2}^{\pi/2} \jvexp\left((N+1)R(t)\right)\,dt,
\end{equation}
for
\[
	R(t) = -\log(1-|C_+|\cos 2t).
\]

Since the case $C_+ = 0$ is trivial (and corresponds to a normal operator), we assume that $C_+ \neq 0$.  Therefore the function $R(t)$ for $t\in [-\pi/2,\pi/2]$ has a unique maximum at $t = 0$ since $|C_+|$ is positive.  We therefore know that there exists a complete asymptotic expansion
\begin{equation}\label{eJPhivsA3}
	e^{-(N+1)R(0)}\int_{-\pi/2}^{\pi/2} \jvexp\left((N+1)R(t)\right)\,dt \sim \sum_{j=0}^\infty a_j N^{-j-1/2}
\end{equation}
for some sequence of real numbers $\{a_j\}_{j=0}^\infty$ and
\begin{equation}\label{eJPhivsA4}
	a_0 = \sqrt{\frac{-2\pi}{R''(0)}} = \sqrt{\frac{\pi(1-|C_+|)}{2|C_+|}}.
\end{equation}
Of course, $e^{-(N+1)R(0)} = (1-|C_+|)^{N+1}$.

We deduce from (\ref{ePialphaJPhi}), (\ref{eJPhiBoth}), and (\ref{eJPhivsA2}) that
\begin{multline*}
	||\Pi_N||^2 = \frac{1}{\pi^2}(1-|C_+|^2)^{-1/2}\left(\frac{1+|C_+|}{1-|C_+|}\right)^{N+1}
	\\ \times \left((1-|C_+|)^{N+1} \int_{-\pi/2}^{\pi/2}\jvexp\left((N+1)R(t)\right)\,dt\right)^2.
\end{multline*}
From (\ref{eJPhivsA3}), we obtain the asymptotic expansion
\[
	\left(\frac{1-|C_+|}{1+|C_+|}\right)^{N/2}||\Pi_N|| \sim \sum_{j=0}^\infty c_j N^{-j-1/2}
\]
for
\[
	c_j = \frac{1}{\pi}(1-|C_+|^2)^{-1/4}\left(\frac{1+|C_+|}{1-|C_+|}\right)^{1/2}a_j.
\]
The computation of $c_0$ in Corollary \ref{cDim1Asymptotic} immediately follows, completing the proof of the corollary.

\subsection{Rates of exponential growth in any dimension}\label{ssGrowthAnyDim}

We now consider the case where $n$ may be arbitrary.  We will prove the exponential rate of growth in Corollary \ref{cExponentialGrowth} as a straightforward consequence of (\ref{eJPhiPolar}), Laplace's method, and Stirling's approximation.  The analysis here is not particularly deep, and we presently do not attempt to analyze the suprema involved or to make estimates uniform.

We consider $\beta \in (\overline{\Bbb{R_+}})^n$ with $|\beta| = \sum_{j=1}^n \beta_j = 1$ fixed throughout this section and analyze $J(\Phi, \lambda\beta)$ defined in (\ref{eJPhiPolar}) as $\lambda \rightarrow \infty$.

We here use only the most elementary version of Lapace's method for multidimensional integrals.  Specifically, we assume that $f\in C^\infty(M;[0,\infty))$ for $M$ a compact boundaryless manifold of real dimension $m$ equipped with a natural volume $\mu$, a positive measure with smooth density with respect to any coordinate chart.  Then there exists some $C > 0$ for which
\begin{equation}\label{eMultiDimLaplace}
	\frac{1}{C}\lambda^{-m/2}\left(\sup_{M}f\right)^\lambda \leq \int_M f(x)^\lambda\,d\mu(x) \leq \mu(M)\left(\sup_{M}f\right)^\lambda,
\end{equation}
valid for $\lambda \in \Bbb{R}_+$ sufficiently large.  The right-hand bound is trivial, and the left-hand bound is easily obtained from estimating $f$ from below in local coordinates
\[
	\varphi:\Bbb{R}^m \supseteq U \rightarrow V \subseteq M, \quad \varphi(0) = x_0
\]
centered at some $x_0$ where $f$ attains its maximum:
\[
	(f\circ \varphi)(y) \geq (\sup_{M}f)-Cy^2.
\]
The error induced by multiplying by a cutoff function localizing to a neighborhood of $y = 0$ is exponentially small in $\lambda$, and so we have the estimate (\ref{eMultiDimLaplace}) for $\lambda$ sufficiently large.

We note that strict convexity of $\Phi$ certainly means that $(\Phi(\omega))^{-n}\sim 1$ when $|\omega|=1$.  We also note that, when $\lambda \in \Bbb{R}_+$,
\[
	|\omega^{\lambda \beta}|^2 = \left|\prod_{j=1}^n \omega^{\lambda\beta_j}\right|^2 = |\omega^\beta|^{2\lambda}.
\]
Applying the upper bound in (\ref{eMultiDimLaplace}) with
\[
	f(\omega) = (4\Phi(\omega))^{-1}|\omega^\beta|^2
\]
for fixed $\beta$ with $|\beta| = 1$ then gives
\begin{multline*}
	\log\int_{|\omega|=1} (4\Phi(\omega))^{-\lambda|\beta|-n} |\omega^{\lambda \beta}|^2 \, dL(\omega)
	\\ \leq \log \left(\sup_{|\omega|=1} (4\Phi(\omega))^{-1} |\omega^\beta|^2\right) + \log C(0,n) + \log \sup_{|\omega| = 1} (4\Phi(\omega))^{-n},
\end{multline*}
recalling from (\ref{eCalphan}) that $C(0,n)$ is the induced Lebesgue measure of $\{|\omega| = 1\}$.  A similar application of the lower bound in (\ref{eMultiDimLaplace}) gives
\begin{multline*}
	\log\int_{|\omega|=1} (4\Phi(\omega))^{-\lambda|\beta|-n} |\omega^{\lambda \beta}|^2 \, dL(\omega)
	\\ \geq \log \left(\sup_{|\omega|=1} (4\Phi(\omega))^{-1} |\omega^\beta|^2\right) - \log C - \frac{2n-1}{2}\log \lambda + \log \inf_{|\omega| = 1} (4\Phi(\omega))^{-n}.
\end{multline*}
Naturally, when we consider the limit $\lambda \rightarrow \infty$, we have $\log \lambda \gg 1$.

We therefore have that, for $\beta \in (\overline{\Bbb{R}_+})^n$ fixed with $|\beta|=1$ and for $\lambda$ sufficiently large,
\begin{multline}\label{elogint}
	\log \int_{|\omega|=1} (4\Phi(\omega))^{-\lambda|\beta|-n} |\omega^{\lambda \beta}|^2 \, dL(\omega) 
	\\ = \lambda \log \left(\sup_{|\omega|=1} (4\Phi(\omega))^{-1} |\omega^\beta|^2\right) + \BigO(\log \lambda).
\end{multline}

We then use Stirling's approximation in the form
\[
	\Gamma(x+1) = \left(\frac{x}{e}\right)^x \sqrt{2\pi x}(1 + \BigO(x^{-1})),\quad x\rightarrow +\infty.
\]
In the end, the projections discussed only make sense when $\lambda \beta \in \Bbb{N}_0^n$, but we may regardless follow (\ref{eCalphan}) and write
\[
	C(\lambda\beta;n) = \frac{2\pi^n}{\Gamma(\lambda + n)}\prod_{j=1}^n \Gamma(\lambda\beta_j+1).
\]
We may ignore terms where $\beta_j = 0$ as these yield a factor of 1; the same result follows if we use the usual convention that $0 \log 0 = 0$.  We also reiterate that we are considering $\beta$ fixed, since this application of Stirling's approximation cannot be said to hold uniformly in $\beta$, particularly when some $\beta_j \rightarrow 0$.

We analyze the three pieces of Stirling's approximation separately: the error factor $1+\BigO(x^{-1})$, the exponential $(x/e)^x$, and the square root $\sqrt{2\pi x}$.  By the Neumann series, for $\lambda$ sufficiently large, we have
\[
	\frac{1}{1+\BigO(\lambda^{-1})}\prod_{j\::\:\beta_j\neq 0} (1+\BigO(\lambda^{-1})) = 1+\BigO(\lambda^{-1}).
\]

For the exponential in the denominator, we note that
\begin{multline*}
	\log \left[\left(\frac{\lambda+n-1}{e}\right)^{\lambda+n-1}\right] = (\lambda+n-1)\left(\log (\lambda+n-1) - 1\right) 
	\\= (\lambda+n-1)\left(\log \lambda - 1 + \BigO(\lambda^{-1})\right)= \lambda \left(\log \lambda - 1\right) + \BigO(\log\lambda),
\end{multline*}
in particular using the Taylor expansion of $\log (x+h) \approx \log x + h/x$ to approximate $\log(\lambda+n-1)$.  We also have that, using $\sum_{j\::\beta_j\neq 0}\beta_j = |\beta| = 1$,
\begin{multline*}
	\log \left(\prod_{j\::\: \beta_j \neq 0}\left(\frac{\lambda \beta_j}{e}\right)^{\lambda\beta_j}\right) = \sum_{j\::\:\beta_j\neq 0} \lambda \beta_j(\log\lambda + \log \beta_j - 1) \\ = \lambda \log\lambda -\lambda + \lambda\sum_{j\::\:\beta_j\neq 0}\beta_j\log\beta_j.
\end{multline*}
Therefore we have the contribution
\[
	\log\left(\frac{1}{((\lambda+n-1)/e)^{\lambda+n-1}}\prod_{j\::\:\beta_j \neq 0} \left(\frac{\lambda\beta_j}{e}\right)^{\lambda\beta_j}\right) = \lambda\sum_{j\::\:\beta_j\neq 0} \beta_j \log \beta_j + \BigO(\log \lambda).
\]

The contribution from the square root is negligible:
\[
	\log\left(\frac{2\pi^n}{\sqrt{2\pi(\lambda+n-1)}}\prod_{j\::\:\beta_j\neq 0}\sqrt{2\pi\lambda\beta_j}\right) = \BigO(\log\lambda),\quad \lambda\rightarrow +\infty.
\]
We therefore conclude that, for $\lambda$ sufficiently large,
\begin{equation}\label{elogC}
	\log C(\lambda\beta;n) = \lambda\sum_{j\::\:\beta_j\neq 0} \beta_j \log \beta_j + \BigO(\log\lambda).
\end{equation}

From (\ref{eJPhiPolar}), (\ref{elogint}), and (\ref{elogC}), we have that
\[
	\log J(\Phi,\lambda\beta) = \lambda\left(\log\left(\sup_{|\omega|=1} (4\Phi(\omega))^{-1}|\omega^\beta|^2\right)- \sum_{j\::\:\beta_j \neq 0} \beta_j\log\beta_j\right) + \BigO(\log\lambda).
\]
Finally, we have from (\ref{ePialphaJPhi}) that
\begin{multline*}
	\lambda^{-1}\log ||\Pi_{\lambda\beta}||_{\mathcal{L}(H_\Phi)} = \frac{1}{2}\log \left(\sup_{|\omega|=1} (4\Phi(\omega))^{-1}|\omega^\beta|^2\right)
	\\ + \frac{1}{2} \log \left(\sup_{|\omega|=1} (4\Phi^\dagger(\omega))^{-1}|\omega^\beta|^2\right) - \sum_{j\::\:\beta_j \neq 0} \beta_j\log\beta_j + \BigO(\lambda^{-1}\log \lambda).
\end{multline*}
This proves Corollary \ref{cExponentialGrowth}.

\begin{remark}
We can see that the role of $\sum \beta_j\log\beta_j$ is nontrivial by examining the case (\ref{eHarmonicRationallyIndependent}), a self-adjoint operator with simple eigenvalues and weights $\Phi(x) = \Phi^\dagger(x) = \frac{1}{4}|x|^2$.  The norms of the spectral projections are uniformly one, and therefore the rate of exponential growth $\lim_{\lambda \rightarrow \infty} \lambda^{-1}\log||\Pi_{\lambda\beta}||$ is zero for every $\beta$. This allows us to indirectly deduce that
\[
	\log \left(\sup_{|\omega|=1} |\omega^\beta|^2\right) = \sum_{j\::\:\beta_j \neq 0} \beta_j\log\beta_j.
\]
\end{remark}

\subsection{Numerical computation of growth rates in dimension 2}\label{ssNumericsExponentialGrowth}

In order to show new information which may be obtained through Corollary \ref{cExponentialGrowth}, we consider Examples \ref{exKFP1} and \ref{exJordan1}.  The author feels that these numerical computations shed some light on the connection between the norms of spectral projections and geometry and may point to interesting directions for future work.

In dimension 2, we naturally parameterize $\beta \in (\overline{\Bbb{R}_+})^2$ with $|\beta| = 1$ via
\[
	\beta = (1-t,t), \quad t \in [0,1].
\]
Following Corollary \ref{cExponentialGrowth}, we will study
\[
	g(\beta, \Phi) = \mathop{\lim_{\lambda \rightarrow \infty}}_{\lambda\beta\in\Bbb{N}_0^n} \lambda^{-1}\log||\Pi_{\lambda \beta}||_{H_{\Phi}}.
\]
We will focus on the case of Examples \ref{exKFP1} and \ref{exJordan1}, as well as Theorem \ref{tOrthogonalProj}, with
\[
	\Phi_G(x) = \frac{1}{4}|Gx|^2, \quad G \in GL_n(\Bbb{C}).
\]

For $G_{\textnormal{KFP}}$ as in (\ref{eGKFP}), we record $g(\beta,\Phi_{G_{\textnormal{KFP}}})$ in Figure \ref{fGrowthKFP}.  It seems apparent that a maximum occurs at $\beta = (1/2,1/2)$.  We furthermore see that the exponential growth rate decreases significantly as $\beta \rightarrow (0,1)$ or $\beta\rightarrow (1,0)$, but certainly does not approach zero; we remark that this corresponds to spectral projections onto eigenvalues whose argument approaches $\pi/4$ or $-\pi/4$.

\begin{figure}
\caption{Growth rates $g(\beta, \Phi_{G_{\textnormal{KFP}}})$ versus $\beta = (1-t,t)$.}
\label{fGrowthKFP}
\centering
\includegraphics[scale=0.30]{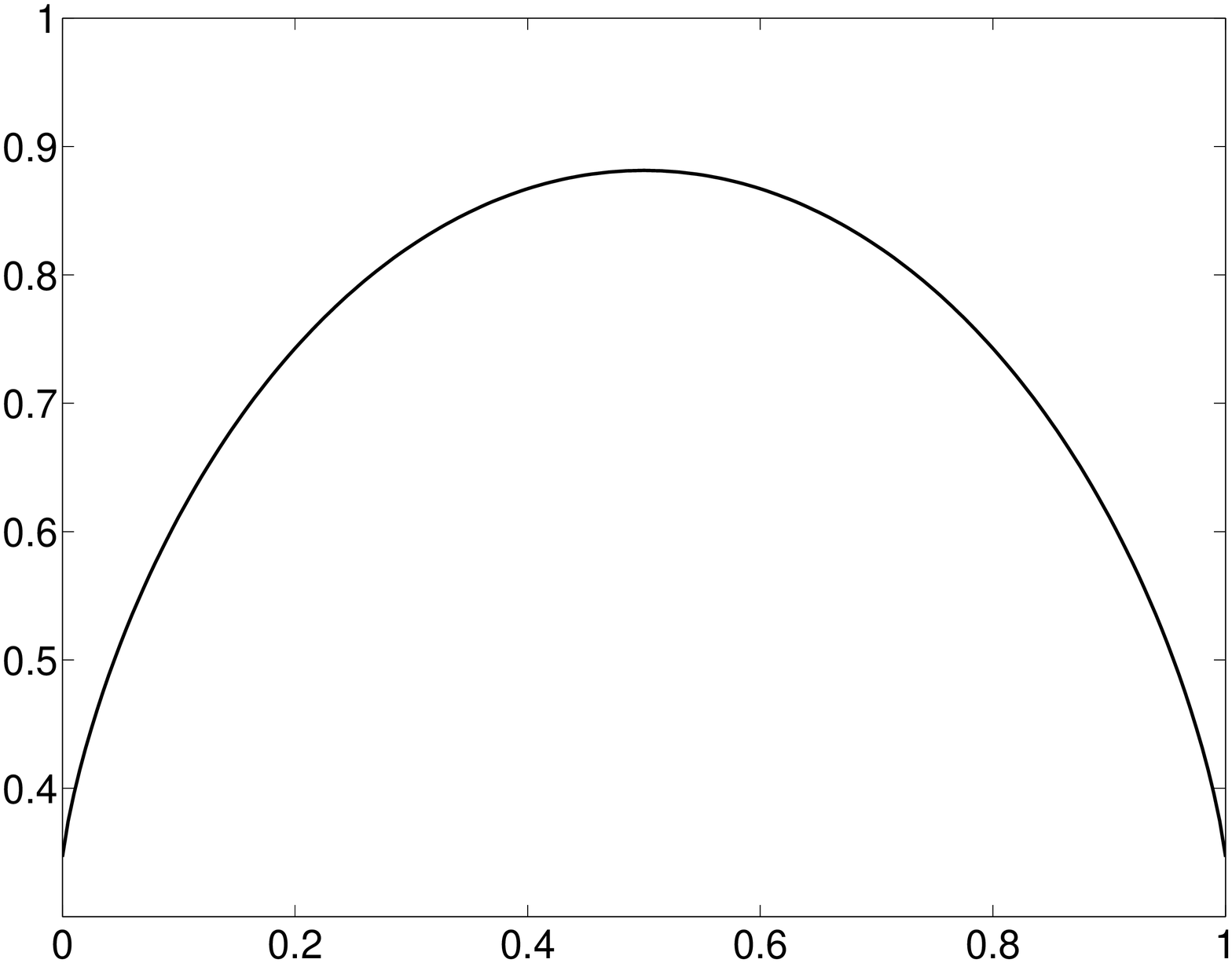}
\end{figure}

\begin{figure}
\caption{Growth rates $g((1/2,1/2), \Phi_{G_{\eps}})$ versus $\log \eps$.}
\label{fGrowthJordan}
\centering
\includegraphics[scale=0.30]{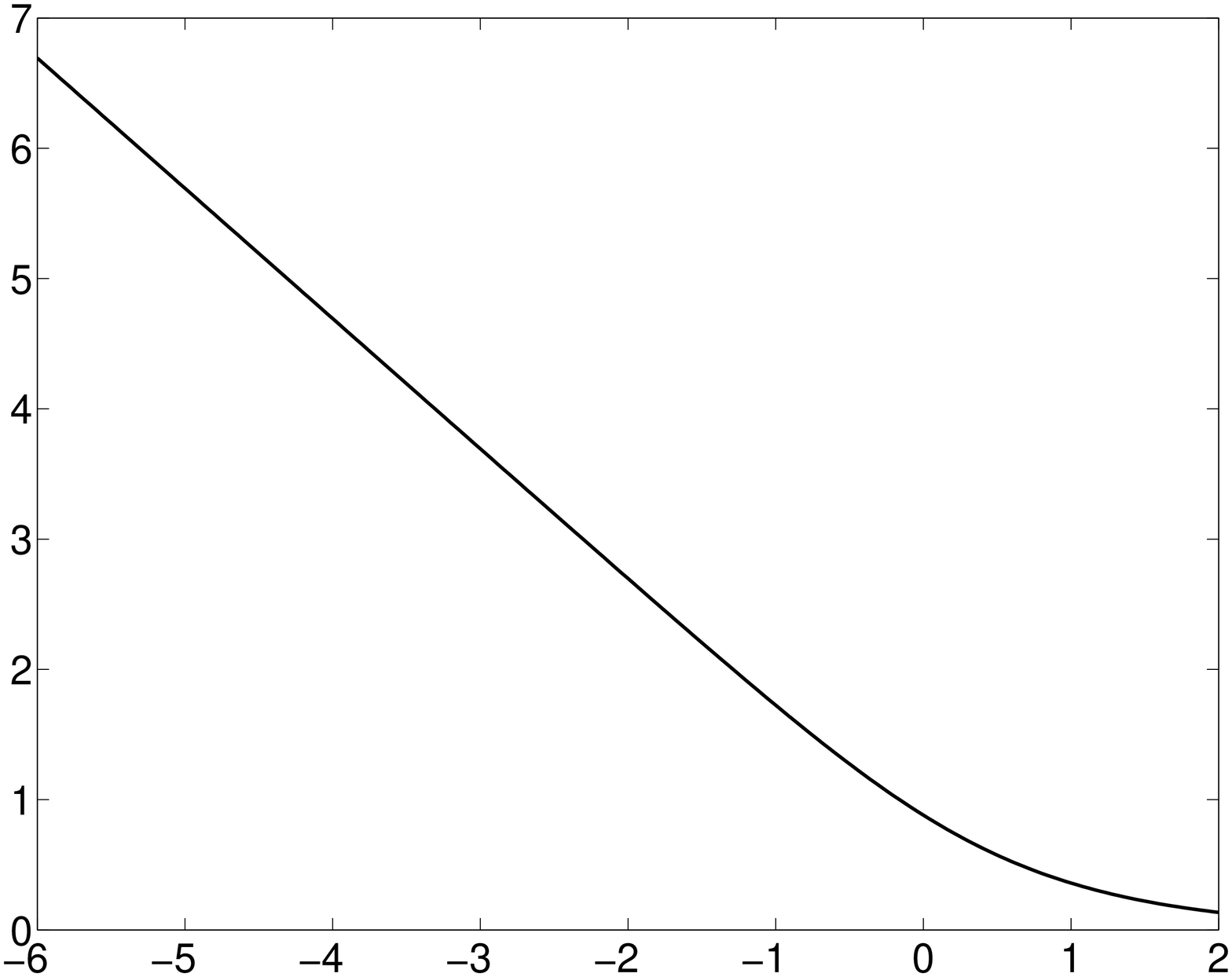}
\end{figure}

We may also observe an apparent connection between the geometry of $\Phi$ and $\Phi^\dagger$ and the exponential growth rates.  From Corollary \ref{cDim1Asymptotic}, we see that in dimension 1 the upper bounds given by Proposition \ref{cElementaryExponential} and Corollary \ref{cUpperBound} are identical and sharp.  In the case where $\Phi(x) = \Phi_G(x)$, it is easy to see that the bounds given by Proposition \ref{cElementaryExponential} and Corollary \ref{cUpperBound} are again identical and are equal to the logarithm of the condition number of $G$:
\begin{equation}\label{egvsCond}
	g(\beta, \Phi_G) \leq \log \left(||G||\:||G^{-1}||\right).
\end{equation}
This simply follows from writing the minimum of $\Phi$ on the unit sphere as the square of the least singular value of $G$ and observing that the maximum of $\Phi$ and the reciprocal of the minimum of $\Phi^\dagger$ are both given by the square of the greatest singular value of $G$.

As discussed in Remark \ref{rNotSharp}, this estimate is not generally sharp.  However, it appears to be sharp or extremely nearly sharp for the maximum growth rate for $G_{\textnormal{KFP}}$, to an error of less than $10^{-12}$:
\[
	g((1/2,1/2),\Phi_{G_{\textnormal{KFP}}}) \approx 0.8814 \approx \log \left(||G_{\textnormal{KFP}}||\: ||G_{\textnormal{KFP}}^{-1}||\right).
\]

We may also consider $G_\eps$ from (\ref{eGJordan}).  One may compute that the profile of $G_\eps$ appears similar to that of $G_{\textnormal{KFP}}$, with a strict maximum at $\beta = (1/2,1/2)$; we therefore examine $g((1/2,1/2),G_\eps)$ in Figure \ref{fGrowthJordan} as a function of $\log \eps$.

We may note that (\ref{egvsCond}) gives an accurate picture of the growth of the spectral projections as $\eps \rightarrow 0^+$.  In fact, the relative error
\[
	\frac{|g((1/2,1/2),G_\eps) - \log (||G_\eps||\:||G_\eps^{-1}||)|}{g((1/2,1/2),G_\eps)}
\]
is extremely small for small $\eps$, e.g.\ error $\leq 10^{-13}$ when $\log \eps = -6$, but increases steadily to nearly $0.1$ when $\eps = 1$.  The author therefore considers the connection between condition number and spectral projection growth to be a nontrivial and interesting question.

As a final note, we may observe that the phenomenon of $g(\beta, G)$ having a maximum at $\beta = (1/2,1/2)$ is not true for all $\Phi$, particularly when $\Phi_{\textnormal{plh}} \neq 0$, defined in (\ref{ePhiplhDef}).  In fact, let us consider the simplest such example:
\[
	\Phi(x_1,x_2) = \Phi_1(x_1)+\Phi_2(x_2) = \frac{1}{4}\left(|x_1|^2 - \jvRe(ax_1^2)\right) + \frac{1}{4}|x_2|^2, \quad |a|< 1.
\]
Using Theorem \ref{tBoundProj} and decomposing for instance
\[
	||x^{\lambda\beta}||_{L^2_\Phi}^2 = \left(\int_{\Bbb{C}}|x_1|^{2\lambda\beta_1}e^{-2\Phi_1(x_1)/h}\,dL(x_1)\right)\left(\int_{\Bbb{C}}|x_2|^{2\lambda\beta_2}e^{-2\Phi_2(x_2)/h}\,dL(x_2)\right),
\]
it becomes clear that
\[
	g(\beta,\Phi) = \beta_1g(1,\Phi_1) + \beta_2g(1,\Phi_2).
\]
From Corollary \ref{cDim1Asymptotic}, we see that $g(\beta,\Phi)$ is linear in $\beta_1$:
\[
	g(\beta,\Phi) = \frac{1+|a|}{1-|a|}\beta_1.
\]

\begin{acknowledgements} The author would like to thank Michael Hitrik for helpful discussions at the outset of this work.  He would also like to thank Fr{\'e}d{\'e}ric H{\'e}rau for helpful suggestions and Fredrik Andersson for a valuable discussion contributing to the proof of Proposition \ref{pReversibility}.  Finally, the author would like to thank the referee for a careful reading and many helpful corrections and suggestions.
\end{acknowledgements}

\bibliographystyle{plain}

\bibliography{MicrolocalBibliography2}
\end{document}